\documentclass[12pt,reqno]{NumPDEsArticle}

\usepackage[T1]{fontenc}
\usepackage[english]{babel}
\usepackage{csquotes}
\usepackage{enumitem}
\usepackage{amsmath,amssymb}
\usepackage{booktabs}
\usepackage{biblatex}
\usepackage{diagbox}
\usepackage{microtype}
\definecolor{pyYellow}{HTML}{bcbd22}
\usepackage{moreverb}
\usepackage{orcidlink}
\usepackage{bm}
\usepackage[normalem]{ulem}
\usepackage{tkz-euclide}
\usetikzlibrary{shapes.geometric, calc, arrows, angles, matrix, colorbrewer, positioning}

\usepackage{NumPDEsMacros}
\usepackage{tikz}
\usetikzlibrary{%
    angles,%
    arrows,%
    calc,%
    colorbrewer,%
    hobby,%
    matrix,%
    shapes,%
    fadings,%
    fpu,%
    cd,%
    spy,%
    3d,%
    positioning,%
    shapes.multipart,%
    shadows,%
    shapes.geometric,%
    shapes.symbols%
}
\usepackage{pgfplots}
\usepackage{pgfplotstable}
\usepackage{tikz-3dplot}
\pgfplotsset{%
    compat=newest,%
    every axis/.style={scale only axis},%
    grid style={densely dotted, semithick},%
}
\usepackage{pdftexcmds}
\makeatletter
\newcommand{\strequal}[2]{\pdf@strcmp{#1}{#2}==0}
\makeatother

\definecolor{TUblue}{rgb}{0,0.4,0.6}
\definecolor{TUgray}{rgb}{0.3922,0.3882,0.3882}
\definecolor{TUgreen}{rgb}{0,0.4941,0.4431}
\definecolor{TUmagenta}{rgb}{0.7294,0.2745,0.5098}
\definecolor{TUyellow}{rgb}{0.8824,0.5373,0.1333}
\definecolor{pyBlue}{HTML}{1f77b4}
\definecolor{pyRed}{HTML}{d62728}
\definecolor{pyGreen}{HTML}{2ca02c}
\definecolor{pyOrange}{HTML}{ff7f0e}
\definecolor{pyPurple}{HTML}{9467bd}
\definecolor{pyYellow}{HTML}{bcbd22}
\definecolor{pyGrey}{HTML}{7f7f7f}
\definecolor{pyCyan}{HTML}{17becf}
\definecolor{pyBrown}{HTML}{8c564b}
\definecolor{pyPink}{HTML}{e377c2}
\definecolor{tolMutedIndigo}{HTML}{332288}
\definecolor{tolMutedBlue}{HTML}{88CCEE}
\definecolor{tolMutedTeal}{HTML}{44AA99}
\definecolor{tolMutedGreen}{HTML}{117733}
\definecolor{tolMutedOlive}{HTML}{999933}
\definecolor{tolMutedSand}{HTML}{DDCC77}
\definecolor{tolMutedRose}{HTML}{CC6677}
\definecolor{tolMutedWine}{HTML}{882255}
\definecolor{tolMutedPurple}{HTML}{AA4499}
\definecolor{tolMutedPaleGrey}{HTML}{DDDDDD}
\definecolor{tolVibrantBlue}{HTML}{0077BB}
\definecolor{tolVibrantCyan}{HTML}{33BBEE}
\definecolor{tolVibrantTeal}{HTML}{009988}
\definecolor{tolVibrantOrange}{HTML}{EE7733}
\definecolor{tolVibrantRed}{HTML}{CC3311}
\definecolor{tolVibrantMagenta}{HTML}{EE3377}
\definecolor{tolVibrantGrey}{HTML}{BBBBBB}

\pgfplotsset{%
    colormap={parula}{%
        rgb=(0.2081,0.1663,0.5292)
        rgb=(0.2116,0.1898,0.5777)
        rgb=(0.2123,0.2138,0.627)
        rgb=(0.2081,0.2386,0.6771)
        rgb=(0.1959,0.2645,0.7279)
        rgb=(0.1707,0.2919,0.7792)
        rgb=(0.1253,0.3242,0.8303)
        rgb=(0.0591,0.3598,0.8683)
        rgb=(0.0117,0.3875,0.882)
        rgb=(0.006,0.4086,0.8828)
        rgb=(0.0165,0.4266,0.8786)
        rgb=(0.0329,0.443,0.872)
        rgb=(0.0498,0.4586,0.8641)
        rgb=(0.0629,0.4737,0.8554)
        rgb=(0.0723,0.4887,0.8467)
        rgb=(0.0779,0.504,0.8384)
        rgb=(0.0793,0.52,0.8312)
        rgb=(0.0749,0.5375,0.8263)
        rgb=(0.0641,0.557,0.824)
        rgb=(0.0488,0.5772,0.8228)
        rgb=(0.0343,0.5966,0.8199)
        rgb=(0.0265,0.6137,0.8135)
        rgb=(0.0239,0.6287,0.8038)
        rgb=(0.0231,0.6418,0.7913)
        rgb=(0.0228,0.6535,0.7768)
        rgb=(0.0267,0.6642,0.7607)
        rgb=(0.0384,0.6743,0.7436)
        rgb=(0.059,0.6838,0.7254)
        rgb=(0.0843,0.6928,0.7062)
        rgb=(0.1133,0.7015,0.6859)
        rgb=(0.1453,0.7098,0.6646)
        rgb=(0.1801,0.7177,0.6424)
        rgb=(0.2178,0.725,0.6193)
        rgb=(0.2586,0.7317,0.5954)
        rgb=(0.3022,0.7376,0.5712)
        rgb=(0.3482,0.7424,0.5473)
        rgb=(0.3953,0.7459,0.5244)
        rgb=(0.442,0.7481,0.5033)
        rgb=(0.4871,0.7491,0.484)
        rgb=(0.53,0.7491,0.4661)
        rgb=(0.5709,0.7485,0.4494)
        rgb=(0.6099,0.7473,0.4337)
        rgb=(0.6473,0.7456,0.4188)
        rgb=(0.6834,0.7435,0.4044)
        rgb=(0.7184,0.7411,0.3905)
        rgb=(0.7525,0.7384,0.3768)
        rgb=(0.7858,0.7356,0.3633)
        rgb=(0.8185,0.7327,0.3498)
        rgb=(0.8507,0.7299,0.336)
        rgb=(0.8824,0.7274,0.3217)
        rgb=(0.9139,0.7258,0.3063)
        rgb=(0.945,0.7261,0.2886)
        rgb=(0.9739,0.7314,0.2666)
        rgb=(0.9938,0.7455,0.2403)
        rgb=(0.999,0.7653,0.2164)
        rgb=(0.9955,0.7861,0.1967)
        rgb=(0.988,0.8066,0.1794)
        rgb=(0.9789,0.8271,0.1633)
        rgb=(0.9697,0.8481,0.1475)
        rgb=(0.9626,0.8705,0.1309)
        rgb=(0.9589,0.8949,0.1132)
        rgb=(0.9598,0.9218,0.0948)
        rgb=(0.9661,0.9514,0.0755)
        rgb=(0.9763,0.9831,0.0538)
    }
}

\newcommand\drawslopetriangle[4][ST]{
    \pgfplotsextra
    {
        \pgfkeys{/pgf/fpu=true}
        \pgfmathsetmacro\leftcoord{#3}
        \pgfmathsetmacro\rightcoord{10*#3}
        \pgfmathsetmacro\bottomcoord{#4}
        \pgfmathsetmacro\topcoord{10^(#2)*#4}
        \pgfkeys{/pgf/fpu=false}
        \coordinate (#1-BL) at (axis cs:\leftcoord,\bottomcoord);
        \coordinate (#1-BR) at (axis cs:\rightcoord,\bottomcoord);
        \coordinate (#1-TL) at (axis cs:\leftcoord,\topcoord);
        \shadedraw[%
            bottom color = black!20,%
            middle color = black!5,%
            top color    = white,%
            draw         = black,%
            font         = \scriptsize%
        ]
        (#1-TL) -- (#1-BL) node[midway, left] {\(#2\)} -- (#1-BR) node[midway, below] {\(1\)} -- (#1-TL);
    }
}
\newcommand\drawswappedslopetriangle[4][SST]{
    \pgfplotsextra
    {
        \pgfkeys{/pgf/fpu=true}
        \pgfmathsetmacro\leftcoord{#3/10}
        \pgfmathsetmacro\rightcoord{#3}
        \pgfmathsetmacro\topcoord{#4}
        \pgfmathsetmacro\bottomcoord{10^(-#2)*#4}
        \pgfkeys{/pgf/fpu=false}
        \coordinate (#1-TR) at (axis cs:\rightcoord,\topcoord);
        \coordinate (#1-BR) at (axis cs:\rightcoord,\bottomcoord);
        \coordinate (#1-TL) at (axis cs:\leftcoord,\topcoord);
        \shadedraw[%
            bottom color = black!20,%
            middle color = black!5,%
            top color    = white,%
            draw         = black,%
            font         = \scriptsize%
        ]
        (#1-BR) -- (#1-TR) node[midway, right] {\(#2\)} -- (#1-TL) node[midway, above] {\(1\)} -- (#1-BR);
    }
}

\makeatletter
\pgfdeclareplotmark{halftriangle*}{%
    \pgfpathmoveto{\pgfqpoint{0pt}{\pgfplotmarksize}}%
    \pgfpathlineto{\pgfqpointpolar{-30}{\pgfplotmarksize}}%
    \pgfpathlineto{\pgfqpoint{0pt}{-.5\pgfplotmarksize}}%
    \pgfpathclose
    \pgfusepathqfill
    \ifx\pgf@mark@color@none\pgf@mark@color
    \else
        \pgfscope
        \pgf@set@mark@color
        \pgfpathmoveto{\pgfqpoint{0pt}{\pgfplotmarksize}}%
        \pgfpathlineto{\pgfqpoint{0pt}{-.5\pgfplotmarksize}}%
        \pgfpathlineto{\pgfqpointpolar{-150}{\pgfplotmarksize}}%
        \pgfpathclose
        \pgfusepathqfill
        \endpgfscope
    \fi
    \pgfpathmoveto{\pgfqpoint{0pt}{\pgfplotmarksize}}%
    \pgfpathlineto{\pgfqpointpolar{-30}{\pgfplotmarksize}}%
    \pgfpathlineto{\pgfqpointpolar{-150}{\pgfplotmarksize}}%
    \pgfpathclose
    \pgfusepathqstroke
}
\pgfdeclareplotmark{halfpentagon*}{%
    \pgfpathmoveto{\pgfqpoint{0pt}{\pgfplotmarksize}}%
    \pgfpathlineto{\pgfqpointpolar{18}{\pgfplotmarksize}}%
    \pgfpathlineto{\pgfqpointpolar{-54}{\pgfplotmarksize}}%
    \pgfpathlineto{\pgfqpoint{0pt}{-.8090169943749475\pgfplotmarksize}}%
    \pgfpathclose
    \pgfusepathqfill
    \ifx\pgf@mark@color@none\pgf@mark@color
    \else
        \pgfscope
        \pgf@set@mark@color
        \pgfpathmoveto{\pgfqpoint{0pt}{\pgfplotmarksize}}%
        \pgfpathlineto{\pgfqpoint{0pt}{-.8090169943749475\pgfplotmarksize}}%
        \pgfpathlineto{\pgfqpointpolar{234}{\pgfplotmarksize}}%
        \pgfpathlineto{\pgfqpointpolar{162}{\pgfplotmarksize}}%
        \pgfpathclose
        \pgfusepathqfill
        \endpgfscope
    \fi
    \pgfpathmoveto{\pgfqpoint{0pt}{\pgfplotmarksize}}%
    \pgfpathlineto{\pgfqpointpolar{18}{\pgfplotmarksize}}%
    \pgfpathlineto{\pgfqpointpolar{-54}{\pgfplotmarksize}}%
    \pgfpathlineto{\pgfqpointpolar{234}{\pgfplotmarksize}}%
    \pgfpathlineto{\pgfqpointpolar{162}{\pgfplotmarksize}}%
    \pgfpathclose
    \pgfusepathqstroke
}
\def\pgfplotmarkfivestarinner{.3819660112501051\pgfplotmarksize}
\def\pgfplotmarkfivestarpath{%
    \pgfpathmoveto{\pgfqpoint{0pt}{\pgfplotmarksize}}%
    \pgfpathlineto{\pgfqpointpolar{54}{\pgfplotmarkfivestarinner}}%
    \pgfpathlineto{\pgfqpointpolar{18}{\pgfplotmarksize}}%
    \pgfpathlineto{\pgfqpointpolar{-18}{\pgfplotmarkfivestarinner}}%
    \pgfpathlineto{\pgfqpointpolar{-54}{\pgfplotmarksize}}%
    \pgfpathlineto{\pgfqpointpolar{-90}{\pgfplotmarkfivestarinner}}%
    \pgfpathlineto{\pgfqpointpolar{-126}{\pgfplotmarksize}}%
    \pgfpathlineto{\pgfqpointpolar{-162}{\pgfplotmarkfivestarinner}}%
    \pgfpathlineto{\pgfqpointpolar{162}{\pgfplotmarksize}}%
    \pgfpathlineto{\pgfqpointpolar{126}{\pgfplotmarkfivestarinner}}%
    \pgfpathclose
}
\pgfdeclareplotmark{fivestar}{%
    \pgfplotmarkfivestarpath
    \pgfusepathqstroke
}
\pgfdeclareplotmark{fivestar*}{%
    \pgfplotmarkfivestarpath
    \pgfusepathqfillstroke
}
\pgfdeclareplotmark{halffivestar*}{%
    \pgfpathmoveto{\pgfqpoint{0pt}{\pgfplotmarksize}}%
    \pgfpathlineto{\pgfqpointpolar{54}{\pgfplotmarkfivestarinner}}%
    \pgfpathlineto{\pgfqpointpolar{18}{\pgfplotmarksize}}%
    \pgfpathlineto{\pgfqpointpolar{-18}{\pgfplotmarkfivestarinner}}%
    \pgfpathlineto{\pgfqpointpolar{-54}{\pgfplotmarksize}}%
    \pgfpathlineto{\pgfqpointpolar{-90}{\pgfplotmarkfivestarinner}}%
    \pgfpathclose
    \pgfusepathqfill
    \ifx\pgf@mark@color@none\pgf@mark@color
    \else
        \pgfscope
        \pgf@set@mark@color
        \pgfpathmoveto{\pgfqpoint{0pt}{\pgfplotmarksize}}%
        \pgfpathlineto{\pgfqpointpolar{-90}{\pgfplotmarkfivestarinner}}%
        \pgfpathlineto{\pgfqpointpolar{-126}{\pgfplotmarksize}}%
        \pgfpathlineto{\pgfqpointpolar{-162}{\pgfplotmarkfivestarinner}}%
        \pgfpathlineto{\pgfqpointpolar{162}{\pgfplotmarksize}}%
        \pgfpathlineto{\pgfqpointpolar{126}{\pgfplotmarkfivestarinner}}%
        \pgfpathclose
        \pgfusepathqfill
        \endpgfscope
    \fi
    \pgfplotmarkfivestarpath
    \pgfusepathqstroke
}
\def\pgfplotmarksixstarinner{.5\pgfplotmarksize}
\def\pgfplotmarksixstarpath{%
    \pgfpathmoveto{\pgfqpointpolar{90}{\pgfplotmarksize}}%
    \pgfpathlineto{\pgfqpointpolar{60}{\pgfplotmarksixstarinner}}%
    \pgfpathlineto{\pgfqpointpolar{30}{\pgfplotmarksize}}%
    \pgfpathlineto{\pgfqpointpolar{0}{\pgfplotmarksixstarinner}}%
    \pgfpathlineto{\pgfqpointpolar{-30}{\pgfplotmarksize}}%
    \pgfpathlineto{\pgfqpointpolar{-60}{\pgfplotmarksixstarinner}}%
    \pgfpathlineto{\pgfqpointpolar{-90}{\pgfplotmarksize}}%
    \pgfpathlineto{\pgfqpointpolar{-120}{\pgfplotmarksixstarinner}}%
    \pgfpathlineto{\pgfqpointpolar{-150}{\pgfplotmarksize}}%
    \pgfpathlineto{\pgfqpointpolar{180}{\pgfplotmarksixstarinner}}%
    \pgfpathlineto{\pgfqpointpolar{150}{\pgfplotmarksize}}%
    \pgfpathlineto{\pgfqpointpolar{120}{\pgfplotmarksixstarinner}}%
    \pgfpathclose
}
\pgfdeclareplotmark{sixstar}{%
    \pgfplotmarksixstarpath
    \pgfusepathqstroke
}
\pgfdeclareplotmark{sixstar*}{%
    \pgfplotmarksixstarpath
    \pgfusepathqfillstroke
}
\pgfdeclareplotmark{halfsixstar*}{%
    \pgfpathmoveto{\pgfqpointpolar{90}{\pgfplotmarksize}}%
    \pgfpathlineto{\pgfqpointpolar{60}{\pgfplotmarksixstarinner}}%
    \pgfpathlineto{\pgfqpointpolar{30}{\pgfplotmarksize}}%
    \pgfpathlineto{\pgfqpointpolar{0}{\pgfplotmarksixstarinner}}%
    \pgfpathlineto{\pgfqpointpolar{-30}{\pgfplotmarksize}}%
    \pgfpathlineto{\pgfqpointpolar{-60}{\pgfplotmarksixstarinner}}%
    \pgfpathlineto{\pgfqpointpolar{-90}{\pgfplotmarksize}}%
    \pgfpathclose
    \pgfusepathqfill
    \ifx\pgf@mark@color@none\pgf@mark@color
    \else
        \pgfscope
        \pgf@set@mark@color
        \pgfpathmoveto{\pgfqpointpolar{90}{\pgfplotmarksize}}%
        \pgfpathlineto{\pgfqpointpolar{-90}{\pgfplotmarksize}}%
        \pgfpathlineto{\pgfqpointpolar{-120}{\pgfplotmarksixstarinner}}%
        \pgfpathlineto{\pgfqpointpolar{-150}{\pgfplotmarksize}}%
        \pgfpathlineto{\pgfqpointpolar{180}{\pgfplotmarksixstarinner}}%
        \pgfpathlineto{\pgfqpointpolar{150}{\pgfplotmarksize}}%
        \pgfpathlineto{\pgfqpointpolar{120}{\pgfplotmarksixstarinner}}%
        \pgfpathclose
        \pgfusepathqfill
        \endpgfscope
    \fi
    \pgfplotmarksixstarpath
    \pgfusepathqstroke
}
\makeatother

\colorlet{col1}{pyBlue}
\colorlet{col2}{pyOrange}
\colorlet{col3}{pyGreen}
\colorlet{col4}{pyRed}
\colorlet{col5}{pyCyan}
\colorlet{col6}{pyPurple}
\colorlet{col7}{pyYellow}
\colorlet{col8}{pyBrown}
\colorlet{col9}{pyPink}
\colorlet{col10}{pyGrey}
\pgfplotsset{
    markerdefault/.style = {%
        every mark/.append style = {solid},%
        gray,%
        every mark/.append style = {fill = gray!60!white}%
    },%
    marker1a/.style = {%
        markerdefault,%
        mark = o,%
        mark size = 1.66pt,%
        #1,%
        every mark/.append style = {fill = #1!60!white}%
    },%
    marker1b/.style = {%
        marker1a = #1,%
        mark = halfcircle*,%
    },%
    marker1c/.style = {%
        marker1a = #1,%
        mark = *,%
    },%
    marker2a/.style = {%
        markerdefault,%
        mark = square,%
        mark size = 1.5pt,%
        #1,%
        every mark/.append style = {fill = #1!60!white}%
    },%
    marker2b/.style = {%
        marker2a = #1,%
        mark = halfsquare*,%
    },%
    marker2c/.style = {%
        marker2a = #1,%
        mark = square*,%
    },%
    marker3a/.style = {%
        markerdefault,%
        mark = triangle,%
        mark size = 2.2pt,%
        #1,%
        every mark/.append style = {fill = #1!60!white}%
    },%
    marker3b/.style = {%
        marker3a = #1,%
        mark = halftriangle*,%
    },%
    marker3c/.style = {%
        marker3a = #1,%
        mark = triangle*,%
    },%
    marker4a/.style = {%
        markerdefault,%
        mark = diamond,%
        mark size = 2.75pt,%
        #1,%
        every mark/.append style = {fill = #1!60!white}%
    },%
    marker4b/.style = {%
        marker4a = #1,%
        mark = halfdiamond*,%
    },%
    marker4c/.style = {%
        marker4a = #1,%
        mark = diamond*,%
    },%
    marker5a/.style = {%
        markerdefault,%
        mark = pentagon,%
        mark size = 2pt,%
        #1,%
        every mark/.append style = {fill = #1!60!white}%
    },%
    marker5b/.style = {%
        marker5a = #1,%
        mark = halfpentagon*,%
    },%
    marker5c/.style = {%
        marker5a = #1,%
        mark = pentagon*,%
    },%
    marker6a/.style = {%
        markerdefault,%
        mark = fivestar,%
        mark size = 2.5pt,%
        #1,%
        every mark/.append style = {fill = #1!60!white}%
    },%
    marker6b/.style = {%
        marker6a = #1,%
        mark = halffivestar*,%
    },%
    marker6c/.style = {%
        marker6a = #1,%
        mark = fivestar*,%
    },%
    marker7a/.style = {%
        markerdefault,%
        mark = sixstar,%
        mark size = 2.25pt,%
        #1,%
        every mark/.append style = {fill = #1!60!white}%
    },%
    marker7b/.style = {%
        marker7a = #1,%
        mark = halfsixstar*,%
    },%
    marker7c/.style = {%
        marker7a = #1,%
        mark = sixstar*,%
    },%
    uniform/.style = {%
        dashed,%
        every mark/.append style = {fill = black!20!white}%
    },%
    adaptive/.style = {%
        solid%
    },%
    reference/.style = {%
        thick,%
        dashed%
    }%
}

\usepackage{subcaption}

\def\kk{{\underline{k}}}
\def\eell{{\underline{\ell}}}
\def\MM{\mathcal{M}}
\def\QQ{\mathcal{Q}}
\def\RR{\mathcal{R}}
\def\TT{\mathcal{T}}
\def\UU{\mathcal{U}}
\def\XX{\mathcal{X}}
\def\NN{\mathcal{N}}
\def\A{\mathbb{A}}
\def\N{\mathbb{N}}
\def\R{\mathbb{R}}
\def\T{\mathbb{T}}
\def\Mu{{\sf M}}
\def\Eta{{\sf H}}
\def\Crate{C_{\rm rate}}
\def\Clin{C_{\rm lin}}
\def\Copt{C_{\rm opt}}
\def\Cell{C_{\rm ell}}
\def\Cbnd{C_{\rm bnd}}
\def\Cchild{C_{\rm child}}
\def\Ccls{C_{\rm cls}}
\def\Cshape{C_{\rm shape}}

\def\Cdrel{C_{\textup{drel}}}
\def\Cmon{C_{\textup{mon}}}
\def\Cmesh{C_{\rm mesh}}
\def\Cctr{C_{\rm ctr}}
\def\Caux{C_{\rm aux}}
\def\Ccost{C_{\rm cost}}
\def\copt{c_{\rm opt}}
\def\ceff{c_{\rm eff}}
\def\crel{c_{\rm rel}}
\def\qaux{q_{\rm aux}}
\def\qmesh{q_{\rm mesh}}
\def\qlin{q_{\rm lin}}

\def\qctr{q_{\rm ctr}}
\def\CstabTilde{C_{\rm stab}} 
\def\CrelTilde{C_{\rm rel}}
\def\CmonTilde{C_{\rm mon}}
\def\CdrelTilde{C_{\rm drel}}
\def\qredTilde{q_{\rm red}}
\def\vvvert{|\!|\!|}
\newcommand{\Cwseq}{C_{\textup{eq}}}
\def\CCstab{\widetilde{C}_{\textup{stab}}}

\newcommand{\qqtheta}{q_{\overline{\theta}}}
\def\norm#1#2{\|#1\|_{#2}}
\def\enorm#1{|\!|\!|#1|\!|\!|}
\def\cost{{\sf cost}}
\def\refine{{\sf refine}}
\def\hmod{{\sf h}}
\def\with{\,\colon}
\def\Cres{C_{\rm res}}

\title[AFEM driven by non-residual estimators, Part I: Symmetric PDEs]{Optimal complexity of adaptive
FEM\\ for second-order linear elliptic PDEs\\driven by non-residual estimators\\ Part I: Symmetric PDEs}



\author{Philipp Bringmann\orcidlink{0000-0002-4546-5165}}
\author{Aleksandar Dadic}
\author{Dario Ferloni}
\author{Gregor Gantner\orcidlink{0000-0002-0324-5674}}
\author{Dirk Praetorius\orcidlink{0000-0002-1977-9830}}
\author{Julian Streitberger\orcidlink{0000-0003-1189-0611}}

\address{TU Wien, Institute of Analysis and Scientific Computing, Wiedner Hauptstra\ss{}e 8--10/E101, 1040 Wien, Austria}

\email{philipp.bringmann@asc.tuwien.ac.at}
\email{aleksandar.dadic@asc.tuwien.ac.at}
\email{dirk.praetorius@asc.tuwien.ac.at\quad\rm(corresponding author)}
\email{julian.streitberger@asc.tuwien.ac.at}

\address{University of Twente, Faculty of Electrical Engineering, Mathematics and Computer Science, Hallenweg 19, 7522 NH Enschede, Netherlands}

\email{dario.ferloni@utwente.nl}
\email{gregor.gantner@utwente.nl}

\thanks{This research was funded in whole or in part by the Austrian Science Fund (FWF)
  [\href{https://www.fwf.ac.at/en/research-radar/10.55776/F65}{10.55776/F65},
  \href{https://www.fwf.ac.at/en/research-radar/10.55776/P33216}{10.55776/P33216},
  \href{https://www.fwf.ac.at/en/research-radar/10.55776/PAT3699424}{10.55776/PAT3699424},
  and
  \href{https://www.fwf.ac.at/forschungsradar/10.55776/PAT3446525}{10.55776/PAT3446525}].
  DF and GG were further funded by the German Research Foundation (DFG) under grant 545527047 through the German Excellence Strategy EXC-2047/1-390685813.
  For open access purposes, the authors have applied a CC BY public copyright license
to any author accepted manuscript version arising from this submission.}

\begin{document}

\maketitle

\begin{abstract}
We consider adaptive finite element methods for symmetric second-order linear elliptic PDEs, where the adaptive algorithm steers the local mesh refinement as well as an iterative algebraic solver.
Under abstract assumptions on the underlying \textsl{a-posteriori} error estimator and the solver, we prove that the usual adaptive algorithm leads to unconditional full R-linear convergence, independently of the user-chosen adaptivity parameters.
For sufficiently small parameters, this guarantees optimal complexity in the sense that the decay rate of an appropriate quasi-error is optimal with respect to the overall computation cost (and hence time) measured in terms of the usual nonlinear approximation classes.
Unlike available results in the literature, the main focus is on the analytical 
understanding 
of non-residual estimators like averaging-based estimators as proposed by Zienkiewicz and Zhu or estimators based on equilibrated fluxes.
\end{abstract}

\section{Introduction}
\label{section:introduction}

For a bounded and polyhedral Lipschitz domain $\Omega \subset \R^d$ with $d \ge 1$, we consider the symmetric
second-order linear elliptic PDE
\begin{equation}\label{eq:general_second_order_problem}
  -\operatorname{div}\bigl(\boldsymbol{A} \, \nabla u^\star \bigr) + c \, u^\star
  =
  f - \operatorname{div}(\boldsymbol{f})
  \quad \text{in } \Omega \quad \text{with} \quad u^\star = 0 \quad \text{on } \partial \Omega,
\end{equation}
where $\operatorname{div}(\cdot)$ is understood in the distributional sense; see Section~\ref{section:model-problem} for the precise assumptions on the given data.
We suppose that~\eqref{eq:general_second_order_problem} fits into the setting of the Lax--Milgram lemma and thus ensures existence and uniqueness of the weak solution $u^\star \in H^1_0(\Omega)$.
We aim for the approximation of $u^\star$ by adaptive finite element methods (AFEM) steering the local mesh refinement as well as an iterative algebraic solver.
To this end, the sequential adaptive algorithm loops over the modules
\begin{equation}\label{eq:intro:semr}
  \begin{tikzpicture}[
    baseline=(current bounding box.center),
    >=stealth,
    node distance=2.0cm,
    every node/.style={align=center},
    module/.style={
      draw,
      rounded corners=3pt,
      thick,
      fill=black!3,
      minimum height=8mm,
      text depth=0pt,
      inner xsep=7pt,
      font=\sffamily\bfseries\small
    }
  ]
    \node[module] (solve) {\strut SOLVE \& ESTIMATE};
    \node[module, right=of solve] (mark) {\strut MARK};
    \node[module, right=of mark] (refine) {\strut REFINE};

    \draw[->, thick] (solve) -- (mark);
    \draw[->, thick] (mark) -- (refine);
    \draw[->, thick, rounded corners=8pt]
      (refine.south) |- ++(0,-0.55) -| (solve.south);
  \end{tikzpicture}
\end{equation}
where the classical modules \textsf{\textbf{SOLVE}} and \textsf{\textbf{ESTIMATE}} are intertwined (realized via an inner loop) to ensure sufficient accuracy of the final solver iterate together with a minimal number of solver steps.
Starting from an initial conforming mesh $\TT_0$ of the domain $\Omega$ into compact simplices and an initial guess $u_0^0$, the algorithm creates a sequence of conforming refinements $\TT_\ell$ together with discrete solutions $u_\ell^k$ for solver steps $k = 0, 1, \dots, \kk[\ell]$ that approximate $u^\star$.
Here, the $\ell$-dependent index $\kk[\ell] \in \N$ denotes the number of iterations of the iterative solver on the $\ell$-th mesh $\TT_\ell$, which is determined by the adaptive algorithm and might vary with the mesh level \(\ell\).
Naturally, the algorithm leads to a countably infinite index set $\QQ \subset \N_0^2$ equipped with the inherent lexicographic ordering $|\ell',k'| < |\ell,k|$ meaning that $u_{\ell'}^{k'}$ has been computed before $u_\ell^k$.

Under 
general assumptions on the underlying \textsl{a-posteriori} error estimator $\mu_\ell$ (including, e.g., ZZ-type error estimators~\cite{ZZ87} or 
equilibrated-flux estimators~\cite{bs2008,bps09,vohralik2011,ev15}; see Section~\ref{section:nonresidual-estimator}) and the iterative algebraic solver (covering, e.g., an optimally preconditioned CG method; see Section~\ref{section:contractive-solver}), this work shows that an appropriate quasi-error $\Mu_\ell^k$ (read $\Mu$ as capital-$\mu$) defined as the sum of the discretization error (measured in terms of the \textsl{a-posteriori} error estimator $\mu_\ell(u_\ell^\star)$ evaluated at the \emph{never computed} exact FE solution $u_\ell^\star \approx u^\star$) and the algebraic error (measured by the norm difference of $u_\ell^\star$ and the computed iterate $u_\ell^k$) guarantees \emph{unconditional full R-linear convergence}
\begin{align}\label{eq:intro:full-R-linear}
 \Mu_\ell^k \le \Clin \qlin^{|\ell,k| - |\ell',k'|} \, \Mu_{\ell'}^{k'}
\quad \text{for all } (\ell,k), (\ell',k') \in \QQ \text{ with } |\ell',k'| \le |\ell,k|.
\end{align}
Here, the constants $\Clin > 0$ and $0 < \qlin < 1$ depend only on the problem setting and the user-chosen adaptivity parameters for the Dörfler marking as well as the adaptive termination of the algebraic solver.
Indeed,~\eqref{eq:intro:full-R-linear} is \emph{unconditional} in the sense that it is guaranteed for any choice of these parameters.
In particular, this yields existence of a generic constant $C > 0$ such that
\begin{align}\label{eq:intro:convergence}
 0  \le \norm{u^\star - u_\ell^k}{H^1(\Omega)}
 \le
 \norm{u^\star - u_\ell^\star}{H^1(\Omega)} + \norm{u_\ell^\star - u_\ell^k}{H^1(\Omega)}
 \le C \, \Mu_\ell^k \xrightarrow{|\ell,k| \to \infty} 0
\end{align}
of the adaptive algorithm so that the proposed AFEM with iterative algebraic solver always succeeds to converge.
Moreover, it is an immediate consequence of~\eqref{eq:intro:full-R-linear} and the geometric series
that, for all $s > 0$,
\begin{align}\label{eq:intro:rates}
 \sup_{(\ell,k) \in \QQ} (\#\TT_\ell)^s \, \Mu_\ell^k
 \le
 \sup_{(\ell,k) \in \QQ} \Big(\sum_{\substack{(\ell',k') \in \QQ \\ |\ell',k'| \le |\ell,k|}} \#\TT_{\ell'}\Big)^s \, \Mu_\ell^k
 \le \Crate(s) \, \sup_{(\ell,k) \in \QQ} (\#\TT_\ell)^s \, \Mu_\ell^k,
\end{align}
where $\Crate(s) > 0$ depends only on $\Clin$, $\qlin$, and $s$.
The important interpretation of~\eqref{eq:intro:rates} is that if (and only if) the quasi-error $\Mu_\ell^k$ decays at algebraic rate $s$ with respect to the number of elements (and hence the number of degrees of freedom), then --- under the valid assumption that the individual modules of the AFEM loop~\eqref{eq:intro:semr} are realized in linear cost (see Section~\ref{subsection:abstract_adaptive_algorithm}) --- the quasi-error $\Mu_\ell^k$ decays at algebraic rate $s$ with respect to the overall computation cost (and hence the overall computation time elapsed).

Finally, provided that the user-chosen adaptivity parameters are sufficiently small, the considered AFEM guarantees \emph{optimal complexity}: the decay rate of the quasi-error $\Mu_\ell^k$ is optimal with respect to the cumulative computation cost in the sense that it matches the decay rate of the best approximation error in the nonlinear approximation class $\A_s$, i.e.,
\begin{align}\label{eq:intro:nonlinear_approximation}
 \copt(s) \, \norm{u^\star}{\A_s}
 \le
 \sup_{(\ell,k) \in \QQ} \Big(\sum_{\substack{(\ell',k') \in \QQ \\ |\ell',k'| \le |\ell,k|}} \#\TT_{\ell'}\Big)^{s} \, \Mu_\ell^k
 \le
 \Copt(s) \, \max\{ \norm{u^\star}{\A_s}, \Mu_0^0 \}.
\end{align}
Here, $\copt(s), \Copt(s) > 0$ are generic constants and 
$\norm{u^\star}{\A_s} < \infty$ if and only if 
rate $s$ can be achieved with respect to the number of elements for the standard residual-based estimator and the exact FE solution along a (theoretical) sequence of 
optimal meshes.

Having stated the main results of this work, we focus on the relation and differences to the existing literature.
In the last three decades, the mathematical understanding of AFEM has matured.
Early works on lowest-order AFEM for the 2D Poisson problem are~\cite{doerfler1996, mns2000, bdd2004,stevenson2007}.
Building on these and the important observation of~\cite{dk2008} that separate treatment of error estimator and data oscillations is not needed,
the influential work~\cite{ckns2008} proved optimal convergence rates of AFEM for symmetric second-order linear elliptic PDEs with residual-based error estimator and exact FE solutions based on piecewise polynomials of fixed yet arbitrary degree.
Corresponding results for general non-symmetric second-order linear elliptic PDEs have been established in~\cite{cn2012, ffp2014} and non-residual error estimators were first considered in~\cite{ks2011, cn2012}, all again subject to the exact solution of the corresponding FE systems.
We note that~\cite{cn2012} allows for general polynomial degrees but requires a sufficiently fine initial mesh $\TT_0$ and exploits discrete efficiency (and hence requires appropriate local mesh-refining strategies guaranteeing an interior-node property).
In contrast, the work~\cite{ks2011} avoids the use of (discrete) efficiency as well as any assumption on the initial mesh $\TT_0$, but restricts to lowest-order FEM and applies only to the Poisson problem instead.

While the inexact solution of the FE systems is excluded in~\cite{ckns2008, ks2011, cn2012, ffp2014} and the focus is on optimal convergence rates with respect to the number of degrees of freedom, it is indeed considered in the seminal work~\cite{stevenson2007} (even with a focus on optimal complexity).
Therein, however, the proof relies on a perturbation argument so that even plain convergence is only guaranteed for a sufficiently small choice of the adaptivity parameters.
In essence, the same applies to the work~\cite{bm2009} that includes an iterative solver into the analysis of~\cite{ckns2008} and to the abstract framework of~\cite{cfpp2014}, which allows for non-symmetric PDEs and equivalent error estimators, as long as the arising FE systems are solved with sufficiently high accuracy.

The interplay of adaptive mesh refinement and iterative algebraic solvers in AFEM analysis has recently been reconsidered in~\cite{ghps2018, fhps2019, ghps2021} with a focus on lowest-order FEM for symmetric problems and residual-based error estimators.
The works~\cite{fhps2019, ghps2021} observed that (unconditional) full R-linear convergence~\eqref{eq:intro:full-R-linear} and its consequence for rates~\eqref{eq:intro:rates} is the key to obtain optimal complexity~\eqref{eq:intro:nonlinear_approximation}.
This is potentially best stated in the recent work~\cite{bfmps2025}, which extends the analysis of~\cite{fhps2019, ghps2021} to AFEM for general second-order linear elliptic PDEs with contractive algebraic solvers and residual-based error estimators.
Altogether, the present work unifies and extends the analysis in~\cite{ks2011,cn2012} and~\cite{bfmps2025} to obtain full R-linear convergence and optimal complexity for \emph{symmetric} second-order linear elliptic PDEs with contractive algebraic solvers but non-residual error estimators.
Future work~\cite{bdfgps2026+} will deal with \emph{non-symmetric} second-order linear elliptic PDEs, where the adaptive algorithm must be further refined to retain unconditional convergence.

Overall, the contributions
of the present paper to the existing literature can be summarized as follows: Unlike~\cite{ks2011, cn2012}, we include the inexact solution of the FE systems into the adaptive algorithm and focus on optimal complexity instead of optimal rates only.
Throughout, the focus is on unconditional full R-linear convergence~\eqref{eq:intro:full-R-linear}, whereas optimal complexity~\eqref{eq:intro:nonlinear_approximation} is guaranteed for sufficiently small adaptivity parameters.
Unlike~\cite{ks2011}, our analysis requires only to refine marked elements and applies to piecewise polynomials of fixed yet arbitrary degree, while~\cite{ks2011} requires to refine all marked elements plus their immediate neighbors and restricts to lowest-order FEM.
Unlike~\cite{cn2012}, we do neither require a sufficiently fine initial mesh $\TT_0$ nor the interior-node property for refined elements, which are both critical assumptions in~\cite{cn2012} to ensure convergence of the adaptive algorithm.
In the spirit of~\cite{cfpp2014}, our analysis builds only on estimator properties and does not require a special treatment of the so-called data oscillations.
Unlike~\cite{bfmps2025}, we allow for non-residual error estimators, while the algorithm at first glance takes indeed the same form as in the prior works~\cite{bm2009, ghps2018, fhps2019, ghps2021, bfmps2025}.

The given proofs exploit the local equivalence of non-residual error estimators to the residual-based error estimator, which ---as in~\cite{ks2011}--- is the key property of the considered non-residual estimators.
However, this equivalence cannot be exploited directly as in~\cite{ks2011}, since it holds only for the exact FE solutions, while the particular focus of this work is on inexact iterative solvers.
Hence, we use a subtle but decisive modification of the analysis of~\cite{bfmps2025}.

\subsection*{Outline}
The paper is structured as follows:
Section~\ref{section:abstract_framework} introduces the general framework and notation of this work.
We formulate the model problem (Section~\ref{section:model-problem}), recall mesh refinement by newest-vertex bisection (Section~\ref{section:refine}--\ref{section:zz:notation}), formulate the conforming FE discretization (Section~\ref{section:fem}), recall the related residual-based error estimator (Section~\ref{section:residual-estimator}--\ref{section:modified-residual-estimator}), introduce our notion of a \emph{locally equivalent} error estimator (Section~\ref{section:nonresidual-estimator}) and of contractive algebraic solvers (Section~\ref{section:contractive-solver}), and finally state the related adaptive algorithm (Section~\ref{subsection:abstract_adaptive_algorithm}).
Section~\ref{section:convergence} deals with the convergence of the adaptive algorithm in this setting.
Theorem~\ref{theorem:full_linear_convergence} states unconditional full R-linear convergence~\eqref{eq:intro:full-R-linear}, whereas Theorem~\ref{theorem:optimal-complexity} provides the formal statement of optimal complexity~\eqref{eq:intro:nonlinear_approximation}.
To underline the applicability of the developed abstract framework, we consider ZZ-type error estimators based on averaging in Section~\ref{section:zz} and error estimators based on equilibrated fluxes in Section~\ref{section:equiflux}.
We prove that both \textsl{a-posteriori} error estimators satisfy the assumptions in Section~\ref{section:nonresidual-estimator} so that the convergence results of Section~\ref{section:convergence} apply.
We stress that~\cite{ks2011} provides further examples for locally equivalent estimators, like hierarchical estimators or variants of the residual-based error estimator, and shows that these also fit in our setting, at least for the Poisson model problem.
Numerical experiments in Section~\ref{section:numerics} underline the theoretical results and conclude the work.

\section{Adaptive algorithm with iterative algebraic solver}
\label{section:abstract_framework}

This section provides
the prerequisites on the model problem~\eqref{eq:general_second_order_problem} to formulate the adaptive algorithm that is subsequently analyzed; see Section~\ref{subsection:abstract_adaptive_algorithm} for 
its statement.
In particular, we give 
all details on the modules {\bf\textsf{SOLVE}} (Section~\ref{section:contractive-solver}), {\bf\textsf{ESTIMATE}} (Section~\ref{section:nonresidual-estimator}), {\bf\textsf{MARK}} (Section~\ref{subsection:abstract_adaptive_algorithm}), and {\bf\textsf{REFINE}} (Section~\ref{section:refine}) of the AFEM loop~\eqref{eq:intro:semr}.

\subsection{Model problem}
\label{section:model-problem}

We consider the model problem~\eqref{eq:general_second_order_problem}, where $\boldsymbol{A} \in [ L^\infty(\Omega) ]^{d \times d}_\textup{sym}$ is a symmetric diffusion matrix
and $c \in L^\infty(\Omega)$ is the reaction coefficient.
For the right-hand side of~\eqref{eq:general_second_order_problem}, we suppose that $f \in L^2(\Omega)$ and
$\boldsymbol{f} \in [ L^2(\Omega) ]^d$.
With respect to the later (only theoretical) use of the residual-based error estimator, we suppose additional regularity of the diffusion matrix $\boldsymbol{A}|_T \in W^{1,\infty}(T)$ and the vector part of the right-hand side $\boldsymbol{f}|_T \in H^1(T)$ for all elements $T \in \TT_0$ of an initial mesh $\TT_0$ of $\Omega$; see~\eqref{eq:residual_based_estimator} below.

With $\dual{\cdot}{\cdot}_\Omega$ denoting the usual $L^2(\Omega)$-scalar product, we define, for all $u, v \in H^1_0(\Omega)$,
\begin{align} \label{eq:definition of a}
 a(u, v) \coloneqq \dual{\boldsymbol{A} \nabla u}{\nabla v}_\Omega
 + \dual{c \, u}{v}_\Omega
 \,\,\, \text{and} \,\,\,
 F(v) \coloneqq \dual{f}{v}_\Omega + \dual{\boldsymbol{f}}{\nabla v}_\Omega.
\end{align}
Then, the weak formulation of~\eqref{eq:general_second_order_problem} reads: Find $u^\star \in H^1_0(\Omega)$ such that
\begin{equation}\label{eq:weak_formulation}
  a(u^\star, v) = F(v) \quad \text{for all } v \in H^1_0(\Omega).
\end{equation}
We note that $F \colon H^1_0(\Omega) \to \R$ is a linear and continuous functional on $H^1_0(\Omega)$.
Moreover, we suppose that the symmetric bilinear form $a \colon H^1_0(\Omega) \times H^1_0(\Omega) \to \R$ is bounded and elliptic, i.e., there exist constants $0 < \Cell \le \Cbnd$ such that
\begin{equation}\label{eq:continuity_ellipticity}
  \Cell \, \norm{v}{H^1(\Omega)}^2 \leq a(v, v)
  \,\,\,
  \text{and}
  \,\,\,
  a(v, w) \leq \Cbnd \, \norm{v}{H^1(\Omega)} \, \norm{w}{H^1(\Omega)}
  \,\,\,
  \text{for all $v, w \in H^1_0(\Omega)$}.
\end{equation}
Therefore, the Lax--Milgram lemma ensures existence and uniqueness of the solution $u^\star \in H^1_0(\Omega)$ to 
\eqref{eq:weak_formulation}.
We note that the boundedness of $a(\cdot, \cdot)$ is indeed ensured by the above assumptions on $\boldsymbol{A}$
and $c$, while ellipticity requires additional properties of the coefficients.

Finally, we
note that~\eqref{eq:continuity_ellipticity} also ensures that the \emph{energy norm} $\enorm{v} \coloneqq a(v,v)^{1/2}$ induced by the scalar product $a(\cdot,\cdot)$ is an equivalent norm on $H^1_0(\Omega)$.

\subsection{Mesh refinement (\textsf{REFINE} module)}
\label{section:refine}

Recall the initial mesh $\TT_0$ from the previous section.
Throughout, we use the terminus \emph{mesh} as a synonym for a conforming triangulation $\TT_H$ of $\Omega$ into compact simplices $T \in \TT_H$.
For local mesh refinement, we employ newest-vertex bisection (NVB).
We refer to~\cite{stevenson2008} for NVB with admissible $\TT_0$ and $d \ge 2$, to~\cite{Karkulik2013a} for NVB with general $\TT_0$ for $d=2$, and to the recent work~\cite{dgs2023} for NVB with general $\TT_0$ in any space dimension $d \ge 2$.
For $d = 1$, we refer to~\cite{eps60}.

For each mesh $\TT_H$ and $\MM_H \subseteq \TT_H$, let
$
  \TT_h \coloneqq \refine(\TT_H, \MM_H)
$
be the coarsest conforming refinement of $\TT_H$ such that
(at least) all elements $T \in \MM_H$
have been refined by NVB, i.e.,
$\MM_H  \subseteq \TT_H \backslash \TT_h$.
To abbreviate notation, we write
$
  \TT_h \in \T(\TT_H)
$
if $\TT_h$ can be obtained from $\TT_H$ by finitely many
steps of NVB. In particular,
$\T \coloneqq \T(\TT_0)$ denotes the set of all meshes that can be obtained from $\TT_0$.

We recall the following four properties of NVB that will be exploited below.
\begin{enumerate}[font=\upshape, label=\textbf{\textrm{(R\arabic*)}}, ref=R\arabic*,leftmargin=5em]
\item \label{refinement:splitting} \textbf{child estimate:}
Each refined element $T \in \TT_H$ is split into finitely many children, i.e., there exists a constant $\Cchild \ge 2$ such that, for all $\TT_H \in \T$ and all $\MM_H \subseteq \TT_H$, it holds that
\begin{equation*}
 \# (\TT_H \backslash \TT_h) + \# \TT_H
 \le \# \TT_h
 \le \Cchild \, \# (\TT_H \backslash \TT_h) + \# (\TT_H \cap \TT_h).
\end{equation*}
\item \label{refinement:overlay} \textbf{overlay estimate:}
For all meshes $\TT_H, \TT_h \in \T$, there exists a common refinement $\TT_H \oplus \TT_h \in \T(\TT_H) \cap \T(\TT_h)$ such that
\begin{equation*}
 \# (\TT_H \oplus \TT_h) \le \# \TT_H + \# \TT_h - \# \TT_0.
\end{equation*}
\item \label{refinement:closure} \textbf{mesh-closure estimate:}
There exists a constant $\Ccls \ge 1$ such that, for
each sequence $(\TT_\ell)_{\ell \in \N_0}$ of successively refined meshes, i.e., $\TT_{\ell + 1} \coloneqq \refine(\TT_\ell, \MM_\ell)$ with $\MM_\ell \subseteq \TT_\ell$ for all $\ell \in \N_0$, it holds that
\begin{equation*}
 \# \TT_\ell - \# \TT_0
 \le \Ccls \, \sum_{\ell' = 0}^{\ell-1} \# \MM_{\ell'}
 \quad \text{for all $\ell \in \N_0$}.
\end{equation*}
\item \label{refinement:shape-regular} \textbf{uniform shape regularity:}
There exists a constant $\Cshape \ge 1$ such that for all $\TT_H \in \T$ and all $T \in \TT_H$, there holds $\diam(T)/|T|^{1/d} \le \Cshape$.
\end{enumerate}
For the corresponding proofs, we refer to~\cite{gss2014} for~\eqref{refinement:splitting}, to~\cite{stevenson2007, ckns2008} for~\eqref{refinement:overlay}, to~\cite{bdd2004, stevenson2008, eps60, Karkulik2013a, dgs2023} for~\eqref{refinement:closure}, and to~\cite{stevenson2007,dgs2023} for~\eqref{refinement:shape-regular}.

\subsection{Vertices, facets, patches, and stars}
\label{section:zz:notation}

For any mesh $\TT_H \in \T$, we introduce the following notation.
Let $\VV_H$ denote the set of vertices of $\TT_H$ and $\VV_H^\Omega \coloneqq \VV_H \cap \Omega$ denote the set of interior vertices.
Let $\EE_H$ denote the set of all $(d-1)$-dimensional facets of $\TT_H$, e.g., edges for $d = 2$ and faces for $d = 3$.
Let $\EE_H^\Omega \subset \EE_H$ denote the subset of interior facets of $\TT_H$, i.e., the subset of all $E = T \cap T' \in \EE_H$ for two distinct simplices $T, T' \in \TT_H$.

For any subset $\UU_H \subseteq \TT_H$, we let $\bigcup\UU_H \coloneqq \bigcup_{T \in \UU_H} T \subseteq \overline\Omega$.
Let $m \in \N$. For any $\omega \subseteq \overline\Omega$, we then define the patch of $\omega$ by
\begin{equation*}
 \TT_H[\omega] \coloneqq \set{T \in \TT_H \with \omega \cap T \neq \emptyset}
 \quad \text{with domain} \quad
 \Omega_H[\omega] \coloneqq \mbox{$\bigcup$} \TT_H[\omega]
 \subseteq \overline\Omega.
\end{equation*}
With $\TT_H^1[\omega] \coloneqq \TT_H[\omega]$, we may proceed inductively and define the $(m+1)$-patch of $\omega$ by
\begin{align*}
 \TT_H^{m+1}[\omega] \coloneqq \TT_H^m[\Omega_H[\omega]]
 \quad \text{with domain} \quad
 \Omega_H^{m+1}[\omega] \coloneqq \mbox{$\bigcup$} \TT_H^{m+1}[\omega] \subseteq \overline\Omega.
\end{align*}
For any subset $\UU_H \subseteq \TT_H$, we define
\begin{align*}
 \TT_H^m[\UU_H] \coloneqq \TT_H^m[\mbox{$\bigcup$} \UU_H]
 \quad \text{with domain} \quad
 \Omega_H^m[\UU_H] \coloneqq \mbox{$\bigcup$} \TT_H^m[\UU_H]
\end{align*}
and 
abbreviate $\TT_H[\UU_H] \coloneqq \TT_H^1[\UU_H]$ and $\Omega_H[\UU_H] \coloneqq \Omega_H^1[\UU_H]$.
For 
$z \in \VV_H$, we define 
\begin{equation*}
 \TT_H^m[z] \coloneqq \TT_H^m[\{z\}]
 \quad \text{with domain} \quad
 \Omega_H^m[z] \coloneqq \Omega_H^m[\{z\}],
\end{equation*}
and we abbreviate $\TT_H[z] \coloneqq \TT_H^1[z]$ and $\Omega_H[z] \coloneqq \Omega_H^1[z]$.
Similarly, we define the corresponding vertex star of all interior facets attached to $z$
\begin{equation*}
  \EE_H^\Omega[z] \coloneqq \set{E \in \EE^\Omega_H \with z \in E}
  \quad \text{with skeleton} \quad
  \Sigma_H^\Omega[z] \coloneqq \mbox{$\bigcup$} \EE_H^\Omega[z] \subseteq \overline\Omega.
\end{equation*}

\color{black}

\subsection{Discrete formulation and quasi-orthogonality}
\label{section:fem}

Let $p \in \N$ be a fixed polynomial degree.
With each mesh $\TT_H \in \T$, we associate the discrete spaces
\begin{align}\label{eq:finite_element_space}
\begin{split}
 \PP^p(\TT_H) &\coloneqq \set{ v_H \in L^2(\Omega) \with \forall T \in \TT_H, \, v_H \vert_T \text{ is a polynomial of degree at most $p$} },
 \\
 \SS^p(\TT_H) &\coloneqq \PP^p(\TT_H) \cap C(\Omega).
\end{split}
\end{align}
With the set of Lagrange nodes $\NN_H$ of $\SS^p(\TT_H)$, we denote by $\set{\phi_{H,z}\with z\in\NN_H}$ the associated nodal basis of $\SS^p(\TT_H)$, i.e.,
\begin{equation*}
\phi_{H,z}(z') = \delta_{zz'} \quad \text{for all } z,z' \in\NN_H.
\end{equation*}
Similarly, we denote by $\set{\varphi_{H,z}\with z\in\VV_H}$ the set of hat functions in $\SS^1(\TT_H)$, i.e., $\varphi_{H,z}(z') = \delta_{zz'}$ for all $z,z'\in\VV_H$, and note that $\VV_H = \NN_H$ and $\varphi_{H,z} = \phi_{H,z}$ if $p=1$.

To discretize the variational formulation~\eqref{eq:weak_formulation}, we define
\begin{align}
 \XX_H
 \coloneqq \SS^p(\TT_H) \cap H^1_0(\Omega).
\end{align}
Then, the Lax--Milgram lemma applies to $\XX_H$ and ensures existence and uniqueness of the FE solution $u_H^\star \in \XX_H$ to the discrete problem
\begin{equation}\label{eq:discrete_problem}
 a(u_H^\star, v_H) = F(v_H) \quad \text{for all } v_H \in \XX_H.
\end{equation}
We note that mesh refinement is compatible with the discrete spaces in the sense that $\TT_h \in \T(\TT_H)$ results in $\XX_H \subseteq \XX_h$.

The corresponding FE solutions $u_\ell^\star \in \XX_\ell$ to~\eqref{eq:discrete_problem} satisfy 
the Galerkin orthogonality and the resulting Pythagorean identity
\begin{equation}\label{eq:pythagoras}
  \enorm{u^\star - u_{\ell}^\star}^2
  =
  \enorm{u^\star - u_{\ell+1}^\star}^2 + \enorm{u_{\ell+1}^\star - u_\ell^\star}^2
  \quad \text{provided that } \XX_\ell \subseteq \XX_{\ell+1} \subset H^1_0(\Omega).
\end{equation}
Together with the telescopic series, this proves the following quasi-orthogonality, for any sequence $(\TT_\ell)_{\ell \in \N_0}$ in $\T$ with corresponding nested spaces $\XX_\ell \subseteq \XX_{\ell+1}$ for all $\ell \in \N_0$,
\begin{enumerate}[font=\upshape, label=\textbf{\textrm{(QO)}}, ref=QO, leftmargin=5em]
\item \label{axiom:quasi_orthogonality}
$\displaystyle\sum_{\ell' = \ell}^{\infty} \enorm{u_{\ell'+1}^\star - u_{\ell'}^\star}^2
\le
\enorm{u^\star - u_{\ell}^\star}^2
\quad \text{for all $\ell \in \N_0$}$.
\end{enumerate}
This estimate will be key to prove full R-linear convergence of the adaptive algorithm.

\subsection{Standard residual-based error estimator}
\label{section:residual-estimator}

For purely theoretical reasons, we recall the standard residual-based error estimator; see, e.g.,~\cite{ao2000,verfuerth2013}.
For the model problem~\eqref{eq:general_second_order_problem}, its local contributions read as follows: For all $\TT_H \in \T$, all simplices $T \in \TT_H$, and all discrete functions $v_H \in \XX_H$, we define
\begin{subequations}\label{eq:residual_based_estimator}
\begin{align}
 \begin{split}
    \eta_H(T; v_H)^{2}
    &\coloneqq
    | T |^{2/d} \, \| \operatorname{div}( \boldsymbol{A} \nabla v_H - \boldsymbol{f})
    - c \, v_H +f \|_{L^2(T)}^{2}
    \\
    & \qquad + | T |^{1/d} \, \| \lbrack \! \lbrack ( \boldsymbol{A} \nabla v_H - \boldsymbol{f})
    \cdot \boldsymbol{n} \rbrack \! \rbrack \|_{L^2(\partial T \cap \Omega)}^{2},
 \end{split}
\end{align}
where $\boldsymbol{n}$ denotes the outward unit normal vector on $\partial T \cap \Omega$ and $\lbrack \! \lbrack \cdot \rbrack \! \rbrack$ denotes the jump across the skeleton of $\TT_H$.
For any subset $\UU_H \subseteq \TT_H$, we abbreviate
\begin{align}
  \eta_H(\UU_H; v_H) \coloneqq \Bigl(\sum_{T \in \UU_H} \eta_H(T; v_H)^2 \Bigr)^{1/2}
  \quad \text{for all $v_H \in \XX_H$},
  \label{eq:global_residual_based_estimator}
\end{align}
\end{subequations}
and $\eta_H(v_H) \coloneqq \eta_H(\TT_H; v_H)$ denotes the global residual-based error estimator on the whole mesh $\TT_H$.
It is well-known that $\eta_H$ satisfies the following \emph{axioms of adaptivity}.

\begin{lemma}[axioms of adaptivity~{\cite{cfpp2014}}]\label{lemma:axiom-of-adaptivity}
With constants $\Cstab, \Crel, \Cdrel > 0$ and $0 < \qred < 1$, the residual-based error estimator $\eta_{H}$ satisfies the following properties for any mesh $\TT_H \in \T$ and 
refinement $\TT_h \in \T(\TT_H)$ with 
corresponding 
solutions $u_H^\star \in \XX_H$ and $u_h^\star \in \XX_h$ to~\eqref{eq:discrete_problem}, any subset $\UU_H \subseteq \TT_H \cap \TT_h$, and arbitrary $v_H \in \XX_H$ and $v_h \in \XX_h$:
\begin{enumerate}[font=\upshape, label=\textbf{\textrm{(A\arabic*)}}, ref=A\arabic*,leftmargin=5em]
\item\label{axiom:stability}
\textbf{stability:}
$\vert \eta_h (\UU_H; v_h) - \eta_H(\UU_H; v_H) \vert \le \Cstab \, \enorm{v_h - v_H}$.
\item\label{axiom:reduction}
\textbf{reduction:}
$\eta_h(\TT_h \backslash \TT_H; v_H) \le \qred \, \eta_H(\TT_H \backslash \TT_h; v_H)$.
\item\label{axiom:reliability}
\textbf{reliability:}
$\enorm{u_h^\star - u_H^\star} \le \enorm{u^\star - u_H^\star} \le \Crel \, \eta_H(u_H^\star)$.
\renewcommand{\theenumi}{A3$_{+}$}
\item[\textbf{(A3$_{\boldsymbol{+}}$)}]\refstepcounter{enumi}\label{axiom:discrete_reliability}
\textbf{discrete reliability:}
$\enorm{u_h^\star - u_H^\star} \le\Cdrel \, \eta_H(\TT_H \backslash\TT_h; u_H^\star)$.
\end{enumerate}
For NVB, reduction~\eqref{axiom:reduction} holds with $\qred \coloneqq 2^{-1/(2d)}$.
The constant $\Crel$ depends only on the space dimension $d$, the coefficients $\boldsymbol{A}$ and $c$
in~\eqref{eq:general_second_order_problem}, and on uniform shape regularity~\eqref{refinement:shape-regular}, whereas $\Cdrel$ additionally depends on the polynomial degree $p$, and $\Cstab$ additionally depends on $p$ and $|\Omega|$.
Moreover, there holds \textbf{quasi-monotonicity}
\begin{align}\label{axiom:qm}
\eta_h(u_h^\star) \le \Cmon \, \eta_H(u_H^\star)
\end{align}
with a constant $0 < \Cmon \le 1 + \Cstab \, \Crel$.
\qed
\end{lemma}

\begin{remark}
We refer to~\cite{vohralik24,cdgv26} for the fact that, at the expense of two additional layers of elements around the refined elements $\TT_H \backslash \TT_h$, the discrete reliability~\eqref{axiom:discrete_reliability} holds with a constant independent of the polynomial degree $p$, i.e.,~\eqref{axiom:discrete_reliability} holds in the form
\begin{align*}
 \enorm{u_h^\star - u_H^\star} \le\Cdrel' \, \eta_H(\TT_H^2[\TT_H \backslash\TT_h]; u_H^\star)
 \quad \text{for all } \TT_H \in \T \text{ and } \TT_h \in \T(\TT_H),
\end{align*}
where $\Cdrel' > 0$ depends only on the space dimension $d$, the coefficients $\boldsymbol{A}$ and $c$ in~\eqref{eq:general_second_order_problem}, and on uniform shape regularity~\eqref{refinement:shape-regular}.\qed
\end{remark}

\subsection{Modified residual-based error estimator}
\label{section:modified-residual-estimator}

Again for purely theoretical reasons, we exploit some ideas from~\cite{cfpp2014} and modify the residual-based error estimator.

\begin{lemma}[modified mesh-size function~{\cite[Proposition~8.6]{cfpp2014}}]\label{lemma:mesh_size_function}
Let \(n \in \N_0\).
Then, there exist constants $\Cmesh^{(n)} \geq 1$ and $0 < \qmesh^{(n)} < 1$ and, for all $\TT_H \in \T$, a mesh-size function $\hmod_H^{(n)} \colon \TT_H \to\R_{>0}$, such that the following properties are satisfied for all $T \in \TT_H$, all $\TT_h \in \T(\TT_H)$, and all $T' \in
  \TT_h$ with $T' \subseteq T$:
\begin{enumerate}[,label=\upshape(\roman*), ref=\ensuremath{\mathrm{\roman*}}]
\item\label{axiom:local-equivalence}
\textbf{local equivalence:}\ $\hmod_H^{(n)}(T) \leq |T|^{1/d} \leq \Cmesh^{(n)} \, \hmod_H^{(n)}(T)$;
\item\label{axiom:contraction}
\textbf{contraction on \(\boldsymbol{n}\)-patch:}\
$\hmod_h^{(n)}(T') \leq \qmesh^{(n)} \, \hmod_H^{(n)}(T)$ if $T \in \RR_{H\backslash h}^{(n)} \coloneqq \TT_H^{n}[\,\TT_H \backslash \TT_h\,]$;
\item\label{axiom:identity}
\textbf{identity on complement of \(\boldsymbol{n}\)-patch:}\ $\hmod_h^{(n)}(T) = \hmod_H^{(n)}(T)$
if $T \in \TT_H \backslash\RR_{H\backslash h}^{(n)}$. 
\end{enumerate}
The constants $\Cmesh^{(n)}$ and $\qmesh^{(n)}$ depend only on $n$ and uniform shape regularity~\eqref{refinement:shape-regular}.\qed
\end{lemma}

For fixed $n \in \N_0$, 
define the local contributions of the modified
error estimator $\eta_H^{(n)}$ by
\begin{align}\label{eq:modified_residual_based_estimator}
 \begin{split}
    \eta_H^{(n)}(T; v_H)^{2}
    &\coloneqq
    \hmod_H^{(n)}(T)^2 \, \| \operatorname{div}( \boldsymbol{A} \nabla v_H - \boldsymbol{f})
    - c \, v_H +f \|_{L^2(T)}^{2}
    \\
    & \qquad + \hmod_H^{(n)}(T) \, \| \lbrack \! \lbrack (\boldsymbol{A} \nabla v_H - \boldsymbol{f} )
    \cdot \boldsymbol{n} \rbrack \! \rbrack \|_{L^2(\partial T \cap \Omega)}^{2},
 \end{split}
\end{align}
which differs from the standard residual-based error estimator $\eta_H$ defined in~\eqref{eq:residual_based_estimator} only by changing the weights $|T|^{1/d} \simeq \hmod_H^{(n)}(T)$.
Based on these local contributions, we employ the same summation convention for $\eta_H^{(n)}(\UU_H; v_H)$ and $\eta_H^{(n)}(v_H)$ as in~\eqref{eq:residual_based_estimator}.

The following lemma is also implicitly found in~\cite[Section~8]{cfpp2014}.
It shows that the modified error estimator $\eta_H^{(n)}$ satisfies some \emph{modified axioms of adaptivity}.
In its statement, the set $\RR_{H\backslash h}^{(n)}$ takes the role of the refined elements $\TT_H \backslash \TT_h$, while $\RR_{h\backslash H}^{(n)}$ takes the role of the elements generated by refinement $\TT_h \backslash \TT_H$.
The proof exploits that
the modified mesh-size function contracts on $\RR_{H\backslash h}^{(n)}$.

\begin{lemma}[local equivalence of \( \boldsymbol{\eta_H^{(n)}} \) and \( \boldsymbol{\eta_H} \)]\label{lemma:modified_axiom_of_adaptivity}
Let $n \in \N_0$ and recall $\Cmesh^{(n)} \ge 1$ and $0 < \qmesh^{(n)} < 1$ from Lemma~\ref{lemma:mesh_size_function}.
Then, the modified residual-based error estimator $\eta_H^{(n)}$ from~\eqref{eq:modified_residual_based_estimator} satisfies that
\begin{align}\label{eq:equivalence:modified-estimator}
 \eta_H^{(n)}(\UU_H;\! v_H\!)
 \le \eta_H(\UU_H;\! v_H\!)
 \le \Cmesh^{(n)} \eta_H^{(n)}(\UU_H;\! v_H\!)
 \text{ for all }\, \UU_H \subseteq \TT_H \in \T
 \text{ and } v_H \in \XX_H.
\end{align}
With constants $\CstabTilde^{(n)}, \CrelTilde^{(n)}, \CdrelTilde^{(n)} > 0$ and \(0 < \qredTilde^{(n)} < 1\), there hold the following properties for any mesh $\TT_H \in \T$ and 
refinement $\TT_h \in \T(\TT_H)$ with 
corresponding 
solutions $u_H^\star \in \XX_H$ and $u_h^\star \in \XX_h$ to~\eqref{eq:discrete_problem}, the subsets $\RR_{H\backslash h}^{(n)} \coloneqq \TT_H^{n} [\TT_H \backslash \TT_h]$ and $\RR_{h\backslash H}^{(n)} \coloneqq \TT_h^n[\TT_h \backslash \TT_H]$,
any subset $\UU_H \subseteq \TT_H \backslash \RR_{H \backslash h}^{(n)}\), and arbitrary $v_H \in \XX_H$ and $v_h \in \XX_h$:
\begin{enumerate}[font=\upshape, label=\textbf{\textrm{(A\arabic*$^{\boldsymbol{n}}$)}}, ref=A\arabic*$^{n}$, leftmargin=5em]
\item\label{axiom:patch_stability}
\textbf{stability:}
$\vert \eta_h^{(n)} (\UU_H; v_h) - \eta_H^{(n)}(\UU_H; v_H) \vert
 \le \CstabTilde^{(n)} \, \enorm{v_h - v_H}.$
\item\label{axiom:patch_reduction}
\textbf{reduction:}
$\eta_h^{(n)}(\RR_{h \backslash H}^{(n)}; v_H)
 \le \qredTilde^{(n)} \, \eta_H^{(n)}(\RR_{H \backslash h}^{(n)}; v_H)$.
\item\label{axiom:patch_reliability}
\textbf{reliability:}
$\enorm{u_h^\star - u_H^\star} \le \enorm{u^\star - u_H^\star} \le \CrelTilde^{(n)} \, \eta_H^{(n)}(u_H^\star)$.
\renewcommand{\theenumi}{A3$^n_{+}$}
\item[\textbf{(A3$_{\boldsymbol+}^{\boldsymbol n}$)}]\refstepcounter{enumi}\label{axiom:patch_discrete_reliability}
\textbf{discrete reliability:}
$\enorm{u_h^\star - u_H^\star} \le \CdrelTilde^{(n)} \, \eta_H^{(n)}(\TT_H \backslash \TT_h; u_H^\star)\).
\end{enumerate}
In particular, there holds \textbf{quasi-monotonicity}
\begin{align}\label{axiom:patch_qm}
 \eta_h^{(n)}(u_h^\star) \le \CmonTilde^{(n)} \, \eta_H^{(n)}(u_H^\star).
\end{align}
Since $\eta_H$ and $\eta_H^{(n)}$ differ only in the local mesh-size weight, the constants satisfy $\CstabTilde^{(n)} \le \Cstab$, $\CrelTilde^{(n)} \le \Crel \Cmesh^{(n)}$, $\CdrelTilde^{(n)} \le \Cdrel \Cmesh^{(n)}$, and $0 < \CmonTilde^{(n)} \le \Cmon \Cmesh^{(n)}$ with the constants $\Cstab$, $\Crel$, $\Cdrel$, and $\Cmon$ from Lemma~\ref{lemma:axiom-of-adaptivity}.
Moreover, it holds that $\qredTilde^{(n)} = (\qmesh^{(n)})^{1/2}$.
\end{lemma}

\begin{proof}[Sketch of proof]
Comparing~\eqref{eq:residual_based_estimator} and~\eqref{eq:modified_residual_based_estimator}, the equivalence~\eqref{eq:equivalence:modified-estimator} is an immediate consequence of~Lemma~\ref{lemma:mesh_size_function}\eqref{axiom:local-equivalence}.
In particular, the properties~\eqref{axiom:patch_reliability}, \eqref{axiom:patch_discrete_reliability}, and~\eqref{axiom:patch_qm} follow directly from their counterparts~\eqref{axiom:reliability}, \eqref{axiom:discrete_reliability}, and~\eqref{axiom:qm}.
We only note that $\bigcup (\TT_H \backslash \TT_h) = \bigcup (\TT_h \backslash \TT_H)$, that $\RR_{H\backslash h}^{(n)}$ consists of $\TT_H \backslash \TT_h$ and certain elements from $\TT_H \cap \TT_h$, that an analogous observation applies to $\RR_{h\backslash H}^{(n)}$, and that therefore $\bigcup \RR_{H\backslash h}^{(n)} = \bigcup \RR_{h\backslash H}^{(n)}$.
The proofs of~\eqref{axiom:patch_stability}--\eqref{axiom:patch_reduction} follow along the lines of the corresponding proofs of~\eqref{axiom:stability}--\eqref{axiom:reduction}, where $\CstabTilde^{(n)} \le \Cstab$ is a direct consequence of $\hmod_H^{(n)}(T) \leq |T|^{1/d}$ from~Lemma~\ref{lemma:mesh_size_function}\eqref{axiom:local-equivalence}.
Details are left to the reader.
\end{proof}

\subsection{Abstract non-residual error estimator (\textsf{ESTIMATE} module)}
\label{section:nonresidual-estimator}

As stated in the introduction, this work focuses on AFEM driven by non-residual error estimators that are locally equivalent to the standard residual-based error estimator $\eta_H$ from~\eqref{eq:residual_based_estimator}. More precisely, we require that $\mu_H$ satisfies the following two properties~\eqref{eq:equivalence}--\eqref{axiom:weak_stability}.

We suppose that, for any $\TT_H \in \T$ and any $v_H \in \XX_H$, we can compute the local contributions $\mu_H(T; v_H) \ge 0$ of a non-residual error estimator $\mu_H$.
Based on these local contributions, we employ the same summation convention for $\mu_H(\UU_H; v_H)$ and $\mu_H(v_H)$ as in~\eqref{eq:residual_based_estimator}.
In addition, we suppose that $\mu_H$ satisfies \textbf{local equivalence to the residual-based error estimator $\eta_H$} from~\eqref{eq:residual_based_estimator}, i.e.,
there exists a constant $\Cwseq \ge 1$ and a patch level $m \in \N_0$ such that, for all $\UU_H \subseteq \TT_H$,
the error estimator $\mu_H$ satisfies for the exact discrete solution $u_H^\star \in \XX_H$ to~\eqref{eq:discrete_problem} that
\begin{equation}\label{eq:equivalence}
 \eta_H(\UU_H; u_H^\star)
 \le
 \Cwseq \, \mu_H(\TT_H^{m}[\UU_H]; u_H^\star)
 \quad\text{and} \quad
 \mu_H(\UU_H; u_H^\star)
 \le
 \Cwseq \, \eta_H(\TT_H^{m}[\UU_H]; u_H^\star).
\end{equation}
In addition,
we suppose a \textbf{weak stability} of $\mu_H$: There exists a constant $\CCstab \ge 1$ and a patch level $r \in \N_0$ such that, for all  $\UU_H \subseteq \TT_H \in \T$ and all $v_H, w_H \in \XX_H$, it holds that
\begin{align}\label{axiom:weak_stability}
 \mu_H(\UU_H; v_H)
 \le
 \CCstab \, \bigl[ \mu_H(\TT_H^{r}[\UU_H]; w_H) + \vvvert v_H - w_H \vvvert \bigr].
\end{align}

\subsection{Contractive algebraic solver (\textsf{SOLVE} module)}
\label{section:contractive-solver}

In practice, the discrete formulation~\eqref{eq:discrete_problem} corresponds to a large linear system of equations.
Although being sparse for the canonical FE basis functions, it is not reasonable to assume that it can be solved at linear cost beyond $d = 1$.
As a consequence, the exact FE solution $u_H^\star$ to~\eqref{eq:discrete_problem} cannot be computed at linear cost $\OO(\#\TT_H)$ in general.
Hence, we resort to a contractive algebraic solver, i.e., we aim for an iterative algebraic solver such that, for all $\TT_H \in \T$, the corresponding iteration map $\Psi_H \colon \XX_H \to \XX_H$ guarantees that
\begin{align}\label{eq:algebra:contraction}
 \enorm{u_H^\star - \Psi_H(v_H)} \le \qctr \, \enorm{u_H^\star - v_H}
 \quad \text{for all } v_H \in \XX_H,
\end{align}
with a uniform constant $0 < \qctr < 1$ that is independent of $\TT_H$ and $v_H$.
In addition, we suppose that there exists a computable \textsl{a-posteriori} error estimator $\zeta_H(v_H)$ that provides a two-sided bound for the algebraic error with constants $\ceff, \crel > 0$,
\begin{align}\label{eq:algebra:aposteriori}
 \ceff^{-1} \, \zeta_H(v_H) \le \enorm{u_H^\star - v_H} \le \crel \, \zeta_H(v_H)
 \quad \text{for all } v_H \in \XX_H.
\end{align}

Optimally preconditioned CG methods~\cite{cnx2012} or geometric multigrid methods~\cite{wz2017, imps2022} satisfy the contraction~\eqref{eq:algebra:contraction} with the additional feature that the evaluation of $\Psi_H(v_H)$ can be performed at linear cost $\OO(\#\TT_H)$. The contraction constant $0 < \qctr < 1$ in~\eqref{eq:algebra:contraction} is usually independent of the right-hand side, but depends only on 
$\Cbnd/\Cell$ and uniform shape regularity~\eqref{refinement:shape-regular}, and potentially also on the polynomial degree $p$. 
The geometric multigrid method from~\cite{imps2022} 
is $p$-robust contraction, i.e., $\qctr$ does not depend on $p$.
We refer to~\cite{hmp2026} for further $hp$-robust preconditioners and the use of multigrid as optimal preconditioner for generalized PCG methods 
for AFEM.

\begin{remark}
Note that the triangle inequality and~\eqref{eq:algebra:contraction} prove that
\begin{align*}
 \enorm{\Psi_H(v_H) - v_H}
 \le \enorm{u_H^\star - \Psi_H(v_H)} + \enorm{u_H^\star - v_H}
 \le (\qctr + 1) \, \enorm{u_H^\star - v_H}
\end{align*}
as well as
\begin{align*}
 \enorm{u_H^\star - v_H} \le \enorm{u_H^\star - \Psi_H(v_H)} + \enorm{\Psi_H(v_H) - v_H}
 \le \qctr \, \enorm{u_H^\star - v_H} + \enorm{\Psi_H(v_H) - v_H}.
\end{align*}
In particular,~\eqref{eq:algebra:aposteriori} is satisfied for $\zeta_H(v_H) \coloneqq \enorm{\Psi_H(v_H) - v_H}$ with $\ceff = 1 + \qctr$ and $\crel = 1/(1-\qctr)$, 
i.e.,
measuring the algebraic error of $v_H \in \XX_H$ can involve the computation of one additional solver step $\Psi_H(v_H)$.
However,~\eqref{eq:algebra:contraction}--\eqref{eq:algebra:aposteriori} also imply that
\begin{align*}
 \enorm{u_H^\star - \Psi_H(v_H)}
 \le \qctr \, \enorm{u_H^\star - v_H} \le \qctr \crel \, \zeta_H(v_H)
\end{align*}
providing reliable control of $\enorm{u_H^\star - \Psi_H(v_H)}$.\qed
\end{remark}

\begin{remark}
The multigrid method of~\cite{imps2022} includes a built-in error estimator $\zeta_H(u_H^k)$ that satisfies~\eqref{eq:algebra:aposteriori}.
If a preconditioned CG method is used, where the preconditioner $\mathbf{P}_H$ is spectrally equivalent to the inverse FE matrix $\mathbf{A}_H^{-1}$ (e.g., the additive Schwarz or symmetric multigrid preconditioner from~\cite{cnx2012,hmp2026}),
then the \textsl{a-posteriori} error control~\eqref{eq:algebra:aposteriori} is solver-inherent:
With the usual vector notation (e.g., $\mathbf{x}_H^k$ is the coefficient vector of $u_H^k$ and $\mathbf{r}_H^k \coloneqq \mathbf{b}_H - \mathbf{A}_H \mathbf{x}_H^k$ is the corresponding CG residual etc.), it holds that
\begin{align*}
 \enorm{u_H^\star - u_H^k}^2
 =(\mathbf{x}_H^\star - \mathbf{x}_H^k)^\top \mathbf{A}_H(\mathbf{x}_H^\star - \mathbf{x}_H^k)
 = (\mathbf{A}_H^{-1} \mathbf{r}_H^k)^\top \mathbf{r}_H^k
 \simeq (\mathbf{P}_H \mathbf{r}_H^k)^\top \mathbf{r}_H^k
 \eqqcolon \norm{\mathbf{r}_H^k}{\mathbf{P}_H}^2,
\end{align*}
i.e., $\zeta_H(u_H^k) \coloneqq \norm{\mathbf{r}_H^k}{\mathbf{P}_H}^2$ satisfies~\eqref{eq:algebra:aposteriori} and $\ceff, \crel > 0$ depend only on the spectral equivalence constants of $\mathbf{P}_H \simeq \mathbf{A}_H^{-1}$.\qed
\end{remark}

\subsection{Adaptive algorithm}
\label{subsection:abstract_adaptive_algorithm}

Having collected the necessary prerequisites, we can finally state the adaptive algorithm that is analyzed in the following.

\begin{algorithm}[AFEM with contractive solver]\label{algorithm:afem}
Given an initial mesh $\TT_0$, adaptivity parameters $0 < \theta \le 1$
and $\Cmark \ge 1$, a solver-stopping parameter
$\lambda > 0$, and an initial guess $u_0^0 \in \mathcal{X}_0$, iterate the
following steps~\ref{alg:single_i}--\ref{alg:single_iv} for all
$\ell = 0, 1, 2, 3, \dots$:
\begin{enumerate}[label=(\roman*), ref = {\rm (\roman*)}, font = \upshape]
  \item\label{alg:single_i} {\bf\textsf{SOLVE \& ESTIMATE:}} For all
    $k = 1, 2, 3, \dots$, repeat~\ref{alg:single_a}--\ref{alg:single_b} until
    \begin{equation}\label{eq:single:termination}
      \zeta_\ell(u_\ell^k) \le \lambda \, \mu_\ell(u_\ell^k).
    \end{equation}
    \begin{enumerate}[label=(\alph*), ref = {\rm (\alph*)}, font = \upshape]
      \item\label{alg:single_a} Compute
        $u_\ell^k \coloneqq \Psi_\ell(u_\ell^{k-1})$ by one step of the
        contractive solver.
      \item\label{alg:single_bb} Compute the algebraic error estimator $\zeta_\ell(u_\ell^k)$ satisfying~\eqref{eq:algebra:aposteriori}.
      \item\label{alg:single_b} Compute the local estimator contributions
        $\mu_\ell(T; u_\ell^k)$ for all $T \in \TT_\ell$.
    \end{enumerate}
  \item\label{alg:single:ii} Upon termination of the iterative solver, define
    $\kk[\ell] \coloneqq k \in \N$ and $u_\ell^{\kk} \coloneqq u_\ell^{\kk[\ell]}$.
  \item {\bf\textsf{MARK:}} Determine a set
    $\MM_\ell \subseteq \TT_\ell$
    satisfying
    the D\"orfler marking criterion
    \begin{equation}\label{eq:doerfler}
      \#\MM_\ell \le \Cmark \!\!\!\!
      \min_{\UU_\ell \in \mathbb{M}_\ell[\theta,
      u_\ell^{\kk}]}\!\! \#\UU_\ell,
      \, \text{where} \,\,
      \mathbb{M}_\ell[\theta,u_\ell^{\kk}]
      \coloneqq \bigl\{\UU_\ell \subseteq \TT_\ell
        \,:\, \theta \mu_\ell(u_\ell^{\kk})^2
        \le \mu_\ell(\UU_\ell; u_\ell^{\kk})^2\bigr\}.
    \end{equation}
  \item\label{alg:single_iv} {\bf\textsf{REFINE:}} Generate the locally refined mesh
    $\TT_{\ell+1} \coloneqq
    \refine(\TT_\ell,\MM_\ell)$ and employ nested
    iteration $u_{\ell+1}^0 \coloneqq u_\ell^{\kk}$.
\end{enumerate}
\end{algorithm}

\begin{remark}
Instead of $\zeta_\ell(u_\ell^k)$, one can also use $\zeta_\ell(u_\ell^{k-1})$ in~\eqref{eq:single:termination}, e.g., if $\zeta_\ell(u_\ell^{k-1}) = \enorm{u_\ell^k - u_\ell^{k-1}}$ is used to steer the solver (and already involves the computation of $u_\ell^k$).
In this case, the main results of Section~\ref{section:convergence} hold accordingly with only minor modifications of the proofs.
Details are left to the reader.\qed
\end{remark}

For the numerical analysis of Algorithm~\ref{algorithm:afem}, we consider the index set
\begin{align}
 \QQ \coloneqq \set{(\ell,k) \in \N_0 \times \N_0 \with u_\ell^k \text{ appears in Algorithm~\ref{algorithm:afem}}}.
\end{align}
In addition, we define
\begin{subequations}
\begin{align}
 \eell &\coloneqq \sup \set{\ell \in \N_0 \with (\ell,0) \in \QQ},
 \\
 \kk[\ell] &\coloneqq \sup \set{k \in \N_0 \with (\ell,k) \in \QQ}
 \quad \text{for all } 0 \le \ell < \eell.
\end{align}
\end{subequations}
To abbreviate notation, we recall $u_\ell^\kk \coloneqq u_\ell^{\kk[\ell]}$ and similarly write $(\ell,\kk) \coloneqq (\ell, \kk[\ell])$ for all $\ell < \eell$.
We stress that these definitions are consistent with those of Algorithm~\ref{algorithm:afem}.

We note that Algorithm~\ref{algorithm:afem} is inherently sequential and thus naturally comes with the lexicographic ordering
$|\cdot, \cdot| \colon \QQ \to \N_0$
\begin{align}
 |\ell', k'| < |\ell, k|
 \quad :\Longleftrightarrow \quad
 \text{$u_{\ell'}^{k'}$ appears earlier in Algorithm~\ref{algorithm:afem} than $u_\ell^k$}.
\end{align}
Note that the explicit formula of the lexicographic ordering reads $|\ell, k| \coloneqq k + \sum_{\ell' = 0}^{\ell-1} \kk[\ell']$.

To summarize the computation cost of Algorithm~\ref{algorithm:afem}, note that $\#\TT_\ell \simeq \dim \XX_\ell$ since 
$p \in \N$ is fixed.
In practice, the modules of the algorithm have linear complexity $\OO(\#\TT_\ell)$:
\begin{itemize}
\item {\bf\textsf{SOLVE:}} One solver step, i.e., the evaluation $\Psi_\ell(v_\ell)$ for some given $v_\ell \in \XX_\ell$ has usually linear cost $\OO(\#\TT_\ell)$; see, e.g., the discussion in~\cite{cnx2012, wz2017, imps2022}.
The computation of the built-in \textsl{a-posteriori} error estimator $\zeta_\ell(v_\ell)$ is an integral part of the multigrid algorithm from~\cite{imps2022}, but also the direct use of $\zeta_\ell(u_\ell^k) = \enorm{u_\ell^k - u_\ell^{k-1}}$ has linear cost $\OO(\#\TT_\ell)$.
\item {\bf\textsf{ESTIMATE:}} Up to numerical quadrature, for some given $v_\ell \in \XX_\ell$, all contributions $\mu_\ell(T; v_\ell)$ for all $T \in \TT_\ell$ can usually be computed in parallel at linear cost $\OO(\#\TT_\ell)$; see, e.g.,~\eqref{eq:residual_based_estimator} for the contributions of the residual-based error estimator (though not being the primary focus of the present work).
\item {\bf\textsf{MARK:}} For $\Cmark = 2$, the seminal work~\cite{stevenson2007} provides a binsearch-based realization of the D\"orfler marking that has linear cost $\OO(\#\TT_\ell)$. For $\Cmark = 1$, we refer to~\cite{pp2020} for a realization of the D\"orfler marking at linear cost $\OO(\#\TT_\ell)$.
\item {\bf\textsf{REFINE:}} For NVB, generating the mesh $\TT_{\ell+1}$ from $\TT_\ell$ and $\MM_\ell \subseteq \TT_\ell$ has linear cost $\OO(\#\TT_\ell)$, which is a consequence of the mesh-closure estimate~\eqref{refinement:closure}. We refer to the discussion in~\cite{stevenson2007}.
\end{itemize}
Since the computation of $u_\ell^k$ hinges on the full adaptive history, the quantity
\begin{align}\label{eq:def:cost}
 \cost(\ell,k) \coloneqq \sum_{\substack{(\ell',k') \in \QQ \\ |\ell',k'| \le |\ell,k|}} \#\TT_{\ell'}
 = k \cdot \#\TT_\ell + \sum_{\ell'=0}^{\ell-1} \kk[\ell'] \cdot \#\TT_{\ell'}
 \quad \text{for all } (\ell,k) \in \QQ
\end{align}
is proportional to the total cost (and hence the total computation time) to compute $u_\ell^k$.

\section{Convergence of Algorithm~\ref{algorithm:afem}}
\label{section:convergence}

This section states and proves convergence of Algorithm~\ref{algorithm:afem} under the general assumptions of Section~\ref{section:abstract_framework} which are assumed throughout.
As stated in the introduction, we prove unconditional convergence in the sense that full R-linear convergence holds for any choice of adaptivity parameters $0 < \theta \le 1$, $\Cmark \ge 1$, and $\lambda > 0$ in Algorithm~\ref{algorithm:afem}, while optimal complexity follows for sufficiently small parameters $0 < \theta \ll  1$ and $0 < \lambda \ll 1$.

\subsection{Full R-linear convergence}

The first main result concerns the unconditional full R-linear convergence of the quasi-error
\begin{equation}\label{eq2:single:quasi-error}
 \Mu_\ell^k
 \coloneqq
 \enorm{u_\ell^\star - u_\ell^k} + \mu_\ell(u_\ell^\star)
 \quad \text{for all $(\ell, k) \in \mathcal{Q}$}.
\end{equation}
We note that $\Mu_\ell^k$ includes the algebraic error $\enorm{u_\ell^\star - u_\ell^k}$ controlling the error of the contractive solver and the error estimator $\mu_\ell(u_\ell^\star)$ controlling the discretization error by $\enorm{u^\star - u_\ell^\star} \lesssim \eta_\ell(u_\ell^\star) \simeq \mu_\ell(u_\ell^\star)$ via reliability~\eqref{axiom:reliability} and estimator equivalence~\eqref{eq:equivalence}.

\begin{theorem}[unconditional full R-convergence of Algorithm~\ref{algorithm:afem}]
\label{theorem:full_linear_convergence}
Let $0 < \theta \le 1$, $\Cmark \ge 1$, $\lambda > 0$, and $u_0^0 \in \XX_0$ be arbitrary.
Suppose that the (potentially non-residual) error estimator $\mu_\ell$ is equivalent to the residual-based error estimator $\eta_\ell$ in the sense of Section~\ref{section:nonresidual-estimator}.
Then, Algorithm~\ref{algorithm:afem} guarantees full R-linear convergence of the quasi-error $\Mu_\ell^k$ from~\eqref{eq2:single:quasi-error} in the sense that there exist constants
$0 < \qlin < 1$ and $\Clin > 0$
such that
\begin{equation}\label{eq:single:convergence}
  \Mu_\ell^k
  \le
  \Clin \, \qlin^{|\ell,k| - |\ell',k'|} \,
  \Mu_{\ell'}^{k'}
  \quad \text{for all } (\ell',k'),(\ell,k) \in \QQ
  \text{ with } |\ell',k'| \le |\ell,k|.
\end{equation}
With $n = m + r$, where $m \in \N_0$ stems from local equivalence~\eqref{eq:equivalence} and $r \in \N_0$ stems from weak stability~\eqref{axiom:weak_stability}, the constants $\Clin$ and $\qlin$ depend only on
$\CstabTilde^{(n)}$, $\qredTilde^{(n)}$, $\CmonTilde^{(n)}$, $\CrelTilde^{(n)}$, $\Cmesh^{(n)}$, $\Cwseq$, $\CCstab$, $\qctr$,
$\ceff$, $\theta$, and $\lambda$.
\end{theorem}

Before we give the proof of Theorem~\ref{theorem:full_linear_convergence}, we state three immediate consequences.
The first corollary asserts guaranteed convergence of Algorithm~\ref{algorithm:afem} and follows immediately from~\eqref{eq:single:convergence},
estimator equivalence $\mu_\ell(u_\ell^\star) \simeq \eta_\ell(u_\ell^\star)$ from~\eqref{eq:equivalence}, weak stability~\eqref{axiom:weak_stability} of $\mu_\ell(u_\ell^k)$, and reliability~\eqref{axiom:reliability} of $\eta_\ell(u_\ell^\star)$.

\begin{corollary}[unconditional convergence of Algorithm~\ref{algorithm:afem}]
\def\Cerr{C_{\rm err}}
Under the assumptions of Theorem~\ref{theorem:full_linear_convergence} and, in particular, for any choice of the adaptivity parameters $0 < \theta \le 1$, $\Cmark \ge 1$, and $\lambda > 0$, Algorithm~\ref{algorithm:afem} guarantees convergence
\begin{align}\label{eq:cor:convergence}
 \enorm{u^\star - u_\ell^k}\!\! \!+\! \enorm{u^\star - u_\ell^\star}\!\! \!+\! \enorm{u_\ell^\star - u_\ell^k}\!\! + \eta_\ell(u_\ell^\star) + \mu_\ell(u_\ell^\star) + \mu_\ell(u_\ell^k)
 \le \Cerr \Mu_\ell^k
 \xrightarrow{|\ell,k| \to \infty} 0,
\end{align}
where $\Cerr > 0$ depends only on $\Crel$, $\Cwseq$, and $\CCstab$.\qed
\end{corollary}

Since $\#\QQ = \#\N$, the two distinct cases of the second corollary are obvious by the loop structure of Algorithm~\ref{algorithm:afem}. The identities for $\eell < \infty$ follow from the convergence~\eqref{eq:cor:convergence} as $k \to \infty = \kk[\eell]$ for the $k$-independent terms on the left-hand side.

\begin{corollary}[lucky breakdown of Algorithm~\ref{algorithm:afem}]\label{corollary:lucky}
Under the assumptions of Theorem~\ref{theorem:full_linear_convergence} and, in particular, for any choice of the adaptivity parameters $0 < \theta \le 1$, $\Cmark \ge 1$, and $\lambda > 0$, Algorithm~\ref{algorithm:afem} guarantees
\begin{itemize}
\item either $\eell = \infty$ and $\kk[\ell] < \infty$ for all $\ell \in \N_0$,
\item or $\eell <\infty$ and $\kk[\eell] = \infty$.
\end{itemize}
In the second case, it follows that $u^\star = u_\eell^\star$ with $\mu_\eell(u_\eell^\star) = 0 = \eta_\eell(u_\eell^\star)$.\qed
\end{corollary}

The quasi-error $\Mu_\ell^k$ from~\eqref{eq2:single:quasi-error} cannot be evaluated algorithmically, since $u_\ell^\star$ is never computed by Algorithm~\ref{algorithm:afem}. The following corollary provides an equivalent and computable quasi-error. The proof follows immediately from weak stability~\eqref{axiom:weak_stability} of $\mu_\ell(u_\ell^k)$, \textsl{a-posteriori} error control~\eqref{eq:algebra:aposteriori} of the algebraic solver, and full R-linear convergence~\eqref{eq:single:convergence}.

\begin{corollary}[full R-linear convergence of computable quasi-error]
\def\CClin{\widetilde C_{\rm lin}}
Under the assumptions of Theorem~\ref{theorem:full_linear_convergence}, there holds
\begin{align}\label{eq2+:single:quasi-error}
 \max\{\CCstab, (\CCstab+1)\crel\}^{-1} \, \Mu_\ell^k
 \le \widetilde\Mu_\ell^k  \coloneqq \zeta_\ell(u_\ell^k) + \mu_\ell(u_\ell^k)
 \le (\CCstab + \ceff) \, \Mu_\ell^k.
\end{align}
In particular, this proves that
\begin{equation}\label{eq+:single:convergence}
  \widetilde\Mu_\ell^k
  \le
  \CClin \, \qlin^{|\ell,k| - |\ell',k'|} \,
  \widetilde\Mu_{\ell'}^{k'}
  \quad \text{for all } (\ell',k'),(\ell,k) \in \QQ
  \text{ with } |\ell',k'| \le |\ell,k|,
\end{equation}
where $\CClin \coloneqq \Clin \, (\CCstab + \ceff)(\CCstab + \crel)$, whereas
$0 < \qlin < 1$ 
stems from Theorem~\ref{theorem:full_linear_convergence}.
\qed
\end{corollary}

The proof of Theorem~\ref{theorem:full_linear_convergence} relies on the following auxiliary result that links full R-linear convergence~\eqref{eq:single:convergence} to tail summability.

\begin{lemma}[{characterization of R-linear convergence~\cite[Lemma~4.9]{cfpp2014}}]\label{lemma2:bfmps2025}
	For any scalar sequence $(a_\ell)_{\ell \in \N_0}$ 
	in $\R_{\ge0}$ and $t > 0$, the
	following statements are equivalent:
	\begin{itemize}\label{eq:summability}
		\item[\rm(i)] \textbf{tail summability:} There exists a constant $C_t >
			      0$ such that
		      \begin{equation}\label{eq:tail-summability}
			      \sum_{\ell' = \ell+1}^\infty a_{\ell'}^t \le C_t \, a_\ell^t
			      \quad \text{for all } \ell \in \N_0.
		      \end{equation}
		\item[\rm(ii)] \textbf{R-linear convergence:}
		There exist $\Caux > 0$ and $0 < \qaux < 1$ such that
			\begin{equation}\label{eq:Rlinear:convergence}
		a_{\ell+n} \le \Caux \, \qaux^n \, a_\ell
		\quad
		\text{for all \(\ell, n \in \N_0\).}
	\end{equation}
	\end{itemize}
	Moreover, the proof reveals $0 < \Caux \le (1 + C_t)^{1/t}$, while $0 < \qaux \le (C_t/[1+C_t])^{1/t}$. \qed
\end{lemma}

\begin{proof}[{\bfseries Proof of Theorem~\ref{theorem:full_linear_convergence}}]
We consider the modified residual-based error estimator $\eta_\ell^{(n)}$.
Since the error estimator $\mu_\ell$ is equivalent to the residual-based error estimator $\eta_\ell$ in the sense of~\eqref{eq:equivalence} and $\eta_\ell$ is (even elementwise) equivalent to $\eta_\ell^{(n)}$ in the sense of~\eqref{eq:equivalence:modified-estimator}, the quasi-error $\Mu_\ell^k$ satisfies the equivalence
\begin{equation*}
  \Mu_\ell^k
  =
  \enorm{u_\ell^\star - u_\ell^k} + \mu_\ell(u_\ell^\star)
  \simeq
  \enorm{u_\ell^\star - u_\ell^k} + \eta_\ell^{(n)}(u_\ell^\star)
  \eqqcolon
  \Eta_\ell^k.
\end{equation*}
The remainder of the proof of~\eqref{eq:single:convergence} consists of four steps.

\textbf{Step~1 (perturbed estimator reduction of $\boldsymbol{\eta_\ell^{(n)}}$).}
Let $\ell \in \N_0$ with $(\ell, \kk) \in \QQ$ be arbitrary.
The local equivalences~\eqref{eq:equivalence:modified-estimator} and~\eqref{eq:equivalence}, weak stability~\eqref{axiom:weak_stability}, and the Dörfler marking criterion~\eqref{eq:doerfler} result in
\begin{equation*}
  \begin{aligned}
    &\theta^{1/2} \, (\Cmesh^{(n)})^{-1} \, \Cwseq^{-2} \, \CCstab^{-2} \,
    \eta_\ell^{(n)}(u_\ell^\star)
    \eqreff{eq:equivalence:modified-estimator}{\leq}
    \theta^{1/2} \, (\Cmesh^{(n)})^{-1} \, \Cwseq^{-2} \, \CCstab^{-2} \, \eta_\ell(u_\ell^{\star})
    \\& \quad
    \eqreff{eq:equivalence}{\leq}
    \theta^{1/2} \, (\Cmesh^{(n)})^{-1} \, \Cwseq^{-1} \, \CCstab^{-2} \,
    \mu_\ell(u_\ell^{\star})
    \\& \quad
    \eqreff{axiom:weak_stability}{\leq}
    (\Cmesh^{(n)})^{-1} \, \Cwseq^{-1} \CCstab^{-1} \, \big[
    \theta^{1/2} \, \mu_\ell(u_\ell^{\kk})
    + \theta^{1/2} \, \vvvert u_\ell^{\star} - u_\ell^{\kk} \vvvert \big]
    \\& \quad
    \eqreff{eq:doerfler}{\leq}
    (\Cmesh^{(n)})^{-1} \, \Cwseq^{-1} \CCstab^{-1} \, \big[
    \mu_\ell(\MM_\ell; u_\ell^{\kk})
    + \theta^{1/2} \, \vvvert u_\ell^{\star} - u_\ell^{\kk} \vvvert \big]
    \\& \quad
    \eqreff{axiom:weak_stability}{\leq}
    (\Cmesh^{(n)})^{-1} \, \Cwseq^{-1} \, \big[
    \mu_\ell(\TT_\ell^{r}[\MM_\ell]; u_\ell^{\star})
    +
    ( 1 + \theta^{1/2} \, \CCstab^{-1} )
    \vvvert u_\ell^{\star} - u_\ell^{\kk} \vvvert \big]
    \\& \quad
    \eqreff{eq:equivalence}{\leq}
    (\Cmesh^{(n)})^{-1} \, \big[
    \eta_\ell(\TT_\ell^{m+r}[\MM_\ell]; u_\ell^{\star})
    +
    \Cwseq^{-1} \,
    ( 1 + \theta^{1/2} \, \CCstab^{-1} )
    \vvvert u_\ell^{\star} - u_\ell^{\kk} \vvvert \big]
    \\& \quad
    \eqreff{eq:equivalence:modified-estimator}{\leq}
    \eta_\ell^{(n)}(\TT_\ell^{m+r}[\MM_\ell]; u_\ell^{\star}) +
    (\Cmesh^{(n)})^{-1} \,\Cwseq^{-1} \,
    ( 1 + \theta^{1/2} \, \CCstab^{-1} )
    \vvvert u_\ell^{\star} - u_\ell^{\kk} \vvvert.
  \end{aligned}
\end{equation*}
Recall that $\Cmesh^{(n)}, \Cwseq, \CCstab \ge 1$ and hence $(\Cmesh^{(n)})^{-1} \, \Cwseq^{-1} \, ( 1 + \theta^{1/2} \, \CCstab^{-1} ) \le 2$.
With $0 < \overline{\theta} \coloneqq \theta \, (\Cmesh^{(n)})^{-2} \, \Cwseq^{-4} \, \CCstab^{-4} / 2 \le 1/2$ and $n = m + r$, the last estimate yields that
\begin{align}\label{eq:perturbed_reduction}
    &\overline{\theta} \, \eta_\ell^{(n)}(u_\ell^{\star})^2
    \leq
    \frac{1}{2}
    \bigl[
      \eta_\ell^{(n)}(\TT_\ell^{n}[\MM_\ell]; u_\ell^{\star})
      + 2 \, \vvvert u_\ell^{\star} - u_\ell^{\kk} \vvvert
    \bigr]^2
    \leq
    \eta_\ell^{(n)}(\TT_\ell^{n}[\MM_\ell]; u_\ell^{\star})^2
    +
    4 \, \vvvert u_\ell^{\star} - u_\ell^{\kk} \vvvert^2
    \nonumber
    \\& \qquad
    \leq
    \eta_\ell^{(n)}(\TT_\ell^{n}[\TT_\ell \backslash \TT_{\ell+1}]; u_\ell^{\star})^2 +
    4  \, \vvvert u_\ell^{\star} - u_\ell^{\kk} \vvvert^2
    =
    \eta_\ell^{(n)}(\RR_{\ell\backslash\ell+1}^{(n)}; u_\ell^{\star})^2 +
    4 \, \vvvert u_\ell^{\star} - u_\ell^{\kk} \vvvert^2.
\end{align}
Since
$\TT_{\ell+1} \backslash \TT_\ell \subseteq \RR_{\ell+1\backslash\ell}^{(n)}$ and $\TT_\ell \backslash \TT_{\ell+1} \subseteq \RR_{\ell\backslash\ell+1}^{(n)}$,
it follows that $\TT_{\ell + 1}\backslash \RR_{\ell+1\backslash\ell}^{(n)} = \TT_\ell \backslash \RR_{\ell \backslash \ell+1}^{(n)} \subseteq \TT_{\ell+1} \cap \TT_\ell$.
Hence, stability~\eqref{axiom:patch_stability} and reduction~\eqref{axiom:patch_reduction} lead to
\begin{align}\label{eq:patch_reduction}
 \begin{split}
    \eta_{\ell+1}^{(n)}(u_\ell^{\star})^2
    &\ \eqreff*{axiom:patch_stability}{=} \
    \eta_\ell^{(n)}(\TT_\ell \backslash \RR_{\ell \backslash \ell+1}^{(n)}; u_\ell^{\star})^2
    +
    \eta_{\ell+1}^{(n)}(\RR_{\ell+1\backslash\ell}^{(n)}; u_\ell^{\star})^2
    \\
    &\ \eqreff*{axiom:patch_reduction}{\leq} \
    \eta_\ell^{(n)}(\TT_\ell \backslash \RR_{\ell \backslash \ell+1}^{(n)}; u_\ell^{\star})^2
    +
    (\qredTilde^{(n)})^{2} \, \eta_\ell^{(n)}(\RR_{\ell \backslash \ell+1}^{(n)}; u_\ell^{\star})^2
    \\
    &\ = \
    \eta_\ell^{(n)}(u_\ell^{\star})^2
    - (1 - (\qredTilde^{(n)})^{2}) \,
    \eta_\ell^{(n)}(\RR_{\ell \backslash \ell+1}^{(n)}; u_\ell^{\star})^2
  \end{split}
\end{align}
With~\eqref{eq:perturbed_reduction} and $0 < \qqtheta \coloneqq \bigl[ 1 - (1 - (\qredTilde^{(n)})^{2})
\overline{\theta} \, \bigr]^{1/2} \! < \! 1$, this leads to
\begin{align}\label{eq+:patch_reduction}
    \eta_{\ell+1}^{(n)}(u_\ell^{\star})^2
    &\ \eqreff*{eq:perturbed_reduction}{\leq}
    \qqtheta^2
\, \eta_\ell^{(n)}(u_\ell^{\star})^2
    + 4  \, (1 - (\qredTilde^{(n)})^{2}) \, \vvvert u_\ell^{\star} - u_\ell^{\kk} \vvvert^2
    \le
        \qqtheta^2
\, \eta_\ell^{(n)}(u_\ell^{\star})^2
    + 4 \, \vvvert u_\ell^{\star} - u_\ell^{\kk} \vvvert^2.
\end{align}
With
$(a^2+b^2)^{1/2}\le a + b$ for $a, b \ge 0$,
stability~\eqref{axiom:patch_stability} validates the estimator reduction
\begin{equation}\label{eq:estimator_reduction}
  \begin{aligned}
    \eta_{\ell+1}^{(n)}(u_{\ell+1}^\star)
    &\eqreff*{axiom:patch_stability}{\leq} \
    \eta_{\ell+1}^{(n)}(u_\ell^{\star})
    + \CstabTilde^{(n)} \, \vvvert u_{\ell+1}^\star - u_\ell^{\star} \vvvert
    \\
    &\eqreff*{eq+:patch_reduction}{\leq} \
    \qqtheta \, \eta_\ell^{(n)}(u_\ell^\star)
    +  2  \,
    \vvvert u_\ell^\star - u_\ell^{\kk} \vvvert
    + \CstabTilde^{(n)} \, \vvvert u_{\ell+1}^\star - u_\ell^{\star} \vvvert.
  \end{aligned}
\end{equation}

\medskip
\textbf{Step~2 (tail summability of $\boldsymbol{\Mu_\ell^{\kk}}$ in $\boldsymbol{\ell}$).}
Let $\ell \in \N_0$ with $(\ell + 1, \kk) \in \QQ$ be arbitrary.  We prove a
contraction-type estimate for the weighted quasi-error $\Eta_\ell \coloneqq \vvvert u_\ell^\star -
  u_\ell^{\kk} \vvvert
+ \gamma \, \eta_\ell^{(n)}(u_\ell^\star)$ for an appropriate $\gamma > 0$ explicitly given below and conclude with the
equivalence $\Eta_\ell \simeq \Eta_\ell^{\kk}
\simeq \Mu_\ell^{\kk}$ the tail-summability of the quasi-error $\Mu_\ell^{\kk}$ with
respect to the mesh level $\ell$.

Solver contraction
\eqref{eq:algebra:contraction}, nested iteration $u_{\ell+1}^0 = u_\ell^{\kk}$, and
$\kk[\ell + 1] \ge 1$ lead to
\begin{equation}\label{eq:contraction_algebraic_error}
  \vvvert u_{\ell+1}^{\star} - u_{\ell+1}^{\kk} \vvvert
  \eqreff{eq:algebra:contraction}{\leq}
  \qctr^{\kk[\ell + 1]} \, \vvvert u_{\ell+1}^{\star} - u_{\ell + 1}^{0} \vvvert
  =
  \qctr^{\kk[\ell + 1]} \, \vvvert u_{\ell+1}^{\star} - u_\ell^{\kk} \vvvert
  \le
  \qctr \, \vvvert u_{\ell+1}^{\star} - u_\ell^{\kk} \vvvert.
\end{equation}
This and the estimator reduction~\eqref{eq:estimator_reduction} verify that
\begin{equation*}
  \begin{aligned}
    \Eta_{\ell+1} &= \vvvert u_{\ell+1}^\star - u_{\ell+1}^{\kk} \vvvert
    + \gamma \, \eta_{\ell+1}^{(n)}(u_{\ell+1}^\star)
    \eqreff*{eq:contraction_algebraic_error}{\leq} \
    \qctr \, \vvvert u_{\ell+1}^\star - u_\ell^{\kk} \vvvert
    + \gamma \, \eta_{\ell+1}^{(n)}(u_{\ell+1}^\star)
    \\
    &\eqreff*{eq:estimator_reduction}{\leq} \
    \qctr \, \vvvert u_{\ell+1}^\star - u_\ell^{\kk} \vvvert
    + \gamma \, \bigl[ \qqtheta \, \eta_\ell^{(n)}(u_\ell^\star)
      + 2 \, \vvvert u_\ell^\star - u_\ell^{\kk} \vvvert
    + \CstabTilde^{(n)} \, \vvvert u_{\ell+1}^\star - u_\ell^{\star} \vvvert \bigr]
    \\
    &\leq \
    \bigl[ \qctr + 2 \gamma \bigr] \,
    \vvvert u_\ell^\star - u_\ell^{\kk} \vvvert
    + \gamma \, \qqtheta \, \eta_\ell^{(n)}(u_\ell^\star)
    + \bigl[\qctr + \CstabTilde^{(n)} \gamma \bigr] \,\vvvert u_{\ell+1}^\star - u_\ell^{\star} \vvvert.
  \end{aligned}
\end{equation*}
Thus, the selection $0 < \gamma < (1 - \qctr) / 2$ ensures
that $ 0 < q \coloneqq \max \bigl\{ \qctr + 2 \gamma \,,\,
    \qqtheta
\bigr\} < 1$. Therefore, the latter estimate verifies that
\begin{equation}\label{eq:contraction_quasi_error}
  \Eta_{\ell+1} \le q \, \Eta_\ell + \bigl[\qctr + \CstabTilde^{(n)} \gamma \bigr] \,\vvvert u_{\ell+1}^\star -
  u_\ell^{\star} \vvvert
  \quad \text{for all } \ell \in \N_0.
\end{equation}
Quasi-orthogonality~\eqref{axiom:quasi_orthogonality} and reliability~\eqref{axiom:patch_reliability}
finally reveal that
\begin{align}\label{eq:quasi_orthogonality}
\begin{split}
  \sum_{\ell' = \ell}^{\eell-1} \vvvert u_{\ell'+1}^\star - u_{\ell'}^\star \vvvert^2
  \eqreff{axiom:quasi_orthogonality}{\le}
  \vvvert u^\star - u_\ell^\star \vvvert^2
  \eqreff{axiom:patch_reliability}{\le}
  (\CrelTilde^{(n)})^2 \, \eta_\ell^{(n)}(u_\ell^\star)^2
  \,\le\,
  \gamma^{-2} \, (\CrelTilde^{(n)})^2 \, \, (\Eta_\ell)^2.
\end{split}
\end{align}
Together with the Young inequality and the geometric series,~\eqref{eq:contraction_quasi_error}--\eqref{eq:quasi_orthogonality} prove
\begin{align*}
 \sum_{\ell'=\ell+1}^{\eell-1} \Eta_{\ell'}^2 \lesssim \Eta_\ell^2
 \quad \text{for all } \ell \in \N_0 \text{ with } \ell < \eell.
\end{align*}
where the hidden constant depends only on $\qctr$, $\qqtheta$, $\CstabTilde^{(n)}$, and $\CrelTilde^{(n)}$.
Exploiting Lemma~\ref{lemma2:bfmps2025} with $t \in \{1, 2\}$ and
the equivalence $\Eta_\ell \simeq \Eta_\ell^{\kk} \simeq \Mu_\ell^{\kk}$, we conclude that
\begin{equation}\label{eq:tail_summability_ell}
  \sum_{\ell' = \ell + 1}^{\eell - 1} \Mu_{\ell'}^{\kk}
  \simeq
  \sum_{\ell' = \ell + 1}^{\eell - 1} \Eta_{\ell'}
  \lesssim
  \Eta_\ell
  \simeq
  \Mu_\ell^{\kk}
  \quad \text{for all } \ell \in \N_0
  \text{ with } \ell < \eell.
\end{equation}

\medskip
\textbf{Step~3 (tail summability of $\boldsymbol{\Mu_\ell^k}$ in $\boldsymbol{k}$).}
Let $\ell \in \N_0$ with $(\ell , \kk) \in \QQ$ be arbitrary. Consider $0
\le k \le k' \le \kk[\ell] - 1$. The failure of the termination
criterion~\eqref{eq:single:termination}, reliable algebraic error control~\eqref{eq:algebra:aposteriori}, and contraction of the solver~\eqref{eq:algebra:contraction} lead to
\begin{equation*}
  \begin{aligned}
    &\Mu_\ell^{k'} = \vvvert u_\ell^\star - u_\ell^{k'} \vvvert + \mu_\ell(u_\ell^\star)
    \eqreff{axiom:weak_stability}{\leq}
    \bigl(1 + \CCstab\bigr) \, \vvvert u_\ell^\star - u_\ell^{k'} \vvvert +  \CCstab
    \,\mu_\ell(u_\ell^{k'})
    \\& \quad
    \eqreff*{eq:single:termination}{\leq}
    \bigl(1 + \CCstab\bigr) \, \vvvert u_\ell^\star - u_\ell^{k'} \vvvert +  \CCstab \,
    \lambda^{-1} \, \zeta_\ell(u_\ell^{k'})
    \eqreff*{eq:algebra:aposteriori}\le
    \bigl(1 + \CCstab + \CCstab \ceff \, \lambda^{-1} \bigr) \, \vvvert u_\ell^\star - u_\ell^{k'} \vvvert
      \\& \quad
      \eqreff*{eq:algebra:contraction}{\leq}
      \qctr^{k' - k} \, \bigl( 1 + \CCstab + \CCstab \ceff \, \lambda^{-1} \bigr) \, \vvvert
      u_\ell^\star - u_\ell^{k} \vvvert
      \leq
      \qctr^{k' - k} \, \bigl( 1 + \CCstab + \CCstab \ceff \, \lambda^{-1} \bigr) \, \Mu_\ell^{k}.
    \end{aligned}
  \end{equation*}
  For $k' = \kk[\ell]$, contraction
  of the inexact solver~\eqref{eq:algebra:contraction} results in
  \begin{equation*}
    \Mu_\ell^{\kk} = \vvvert u_\ell^\star - u_\ell^{\kk} \vvvert + \mu_\ell(u_\ell^\star)
    \eqreff{eq:algebra:contraction}{\leq}
    \qctr \, \vvvert u_\ell^\star - u_\ell^{\kk - 1} \vvvert
    + \mu_\ell(u_\ell^{\star})
    \le
    \Mu_\ell^{\kk - 1}.
  \end{equation*}
  With $\Cctr \coloneqq \max \{ \qctr^{-1} \,,\, 1 + \CCstab + \CCstab \ceff \, \lambda^{-1} \}$, we thus obtain the quasi-contraction
  \begin{equation}\label{eq:contraction_quasi_error_k}
    \Mu_\ell^{k'} \le \Cctr \, \qctr^{k' - k} \, \Mu_\ell^{k}
    \quad \text{for all } 0 \le k \le k' \le \kk[\ell].
  \end{equation}
  Finally, we prove stability of the quasi-error $\Mu_{\ell+1}^0 \lesssim \Mu_\ell^{\kk}$.
   Indeed, nested iteration, stability~\eqref{axiom:patch_stability}, reduction~\eqref{axiom:patch_reduction},
   and reliability~\eqref{axiom:patch_reliability}
   verify that
  \begin{equation}\label{eq:stability_k}
    \begin{aligned}
      &\Mu_{\ell+1}^0
      \simeq
      \Eta_{\ell+1}^0 = \vvvert u_{\ell+1}^\star - u_{\ell+1}^0 \vvvert + \eta_{\ell+1}^{(n)}(u_{\ell+1}^\star)
      \ \eqreff*{axiom:patch_stability}{\leq}\
      \vvvert u_{\ell+1}^\star - u_\ell^{\kk} \vvvert + \eta_{\ell+1}^{(n)}(u_\ell^\star)
      + \CstabTilde^{(n)} \,
      \vvvert u_{\ell+1}^\star - u_\ell^{\star} \vvvert
      \\& \quad
      \eqreff*{eq:patch_reduction}{\leq}
      \vvvert u_{\ell+1}^\star - u_\ell^{\kk} \vvvert + \eta_\ell^{(n)}(u_\ell^\star) +
      \CstabTilde^{(n)} \,
      \vvvert u_{\ell+1}^\star - u_\ell^{\star} \vvvert
      \le (1+\CstabTilde^{(n)}) \, \vvvert u_{\ell+1}^\star - u_\ell^{\star} \vvvert +
      \Eta_\ell^{\kk}
      \\& \quad
      \eqreff*{axiom:patch_reliability}{\leq}
      \bigl[
        ( 1 + \CstabTilde^{(n)} ) \, \CrelTilde^{(n)}  + 1
      \bigr] \,
      \Eta_\ell^{\kk}
      \simeq \Mu_\ell^{\kk}.
     \end{aligned}
  \end{equation}

  \medskip
  \textbf{Step~4 (full R-linear convergence of $\boldsymbol{\Mu_\ell^k}$).}
  Exploiting the geometric series, tail summability of the quasi-error $\Mu_\ell^k$ in $\ell$ and $k$ is established via
    \begin{equation*}
    \begin{aligned}
      &\sum_{\substack{(\ell', k') \in \QQ \\ |\ell', k'| > |\ell, k|}} \Mu_{\ell'}^{k'}
      =
      \sum_{\ell' = \ell+1}^{\eell} \sum_{k' = 0}^{\kk[\ell']} \Mu_{\ell'}^{k'}
      +
      \sum_{k' = k+1}^{\kk[\ell]} \Mu_\ell^{k'}
      \\ & \quad
      \eqreff{eq:contraction_quasi_error_k}{\lesssim}
      \sum_{\ell' = \ell+1}^{\eell} \, \Mu_{\ell'}^{0}
      +
      \Mu_\ell^{k}
      \eqreff{eq:stability_k}{\lesssim}
      \sum_{\ell' = \ell}^{\eell-1} \, \Mu_{\ell'}^{\kk}
      +
      \Mu_\ell^{k}
      \eqreff{eq:tail_summability_ell}{\lesssim}
      \Mu_\ell^{\kk}
      +
      \Mu_\ell^{k}
      \eqreff{eq:contraction_quasi_error_k}\lesssim
      \Mu_\ell^{k}
      \quad \text{for all } (\ell, k) \in \QQ.
    \end{aligned}
  \end{equation*}
  Since $\QQ$ is countable and linearly ordered with respect to 
  $| \cdot, \cdot |$,
  this proves tail-summability \eqref{eq:tail-summability} for $t \coloneqq 1$ and $a_{|\ell,k|} \coloneqq \Mu_\ell^{k}$.
  Thus, the equivalence of tail-summability and full R-linear convergence from Lemma~\ref{lemma2:bfmps2025} concludes the proof of~\eqref{eq:single:convergence}.
\end{proof}

\subsection{Optimal complexity of Algorithm~\ref{algorithm:afem}}
\label{section:optimal_complexity}

An important consequence of Theorem~\ref{theorem:full_linear_convergence} is the following corollary which shows that the convergence rate with respect to the number of degrees of freedom $\dim \XX_\ell \simeq \#\TT_\ell$ coincides with the convergence rate with respect to the computation cost $\cost(\ell,k)$ defined in~\eqref{eq:def:cost}.
This observation will be the key to relate optimal convergence rates with optimal complexity in the proof of Theorem~\ref{theorem:optimal-complexity}.
Stated implicitly in~\cite[Proof of Theorem~8]{ghps2021} and explicitly in~\cite{bfmps2025} for the residual-based error estimator, Corollary~\ref{corollary:rates:complexity} transfers to the present setting, since its proof is fully abstract and relies only on full R-linear convergence~\eqref{eq:single:convergence} and the geometric series.

\begin{corollary}[{rates = complexity~\cite[Corollary~1]{bfmps2025}}]
\label{corollary:rates:complexity}
Recall the quasi-error $\Mu_\ell^k$ from~\eqref{eq2:single:quasi-error}.
	For $s > 0$, full R-linear convergence~\eqref{eq:single:convergence} yields that
	\begin{align}\label{eq:corollary:complexity}
	\begin{split}
		M(s)
		\coloneqq
		\sup_{(\ell,k) \in \mathcal{Q}} (\#\TT_\ell)^s \, \Mu_\ell^k
		&\le \sup_{(\ell,k) \in \mathcal{Q}} \cost(\ell,k)^s \, \Mu_\ell^k
		\\&
		\le \sup_{(\ell,k) \in \mathcal{Q}} \Big(\sum_{\substack{(\ell',k') \in \mathcal{Q} \\ |\ell',k'| \le |\ell,k|}}  \sum_{\ell'' =
			0}^{\ell'} \#\TT_{\ell''}\Big)^s \Mu_\ell^k
		\le \Ccost(s) \, M(s).
	\end{split}
	\end{align}
	The proof reveals that $0 < \Ccost(s) \le (\Clin^{1/s}/[1-\qlin^{1/s}])^2$.
	Moreover, there exists a rate $s_0 > 0$ such that $M(s) < \infty$ for all $0 < s \le s_0$ with
	\(s_0
	= \infty\) if \(\underline{\ell} < \infty\).\qed
\end{corollary}

The next lemma shows that, for sufficiently small parameters $0 \le \theta, \lambda \ll 1$, the D\"orfler marking for the residual-based error estimator $\eta_\ell(u_\ell^\star)$ implies the D\"orfler marking for the equivalent error estimator $\mu_\ell(u_\ell^\kk)$ as used in~\eqref{eq:doerfler} of Algorithm~\ref{algorithm:afem}.
Since the numerical analysis of optimal convergence rates for adaptivity based on $\eta_\ell(u_\ell^\star)$ with exact computation of the FE solution $u_\ell^\star$ is well understood, this observation will turn out to be the key to optimal convergence rates (and hence optimal complexity) of Algorithm~\ref{algorithm:afem}.

\begin{lemma}[Dörfler marking for $\boldsymbol{\eta_\ell(u_\ell^\star)}$ implies that for $\boldsymbol{\mu_\ell(u_\ell^\kk)}$]\label{lemma:doerfler}
Under the assumptions of Section~\ref{section:abstract_framework}, we suppose that the adaptivity parameters $\theta, \lambda > 0$ of Algorithm~\ref{algorithm:afem} are sufficiently small in the sense that
\begin{align}\label{eq1:lemma:doerfler}
 \lambda < \lambda^\star \coloneqq \frac{1}{\crel \CCstab}
 \,\, \text{and} \,\,
 0 < \theta_{\rm mark}
 \coloneqq
 \Big[ \frac{\lambda^\star}{\lambda^\star - \lambda} \theta^{1/2} + \CCstab^{-1} \, \frac{\lambda}{\lambda^\star - \lambda} \Big]^2 \, \Cwseq^4 \CCstab^4  \le 1.
\end{align}
Then, for any $\ell \in \N_0$ with $(\ell,\kk) \in \QQ$ and $\RR_\ell \subseteq \TT_\ell$, Algorithm~\ref{algorithm:afem} guarantees the implication
\begin{align}\label{eq2:lemma:doerfler}
 \theta_{\rm mark} \, \eta_\ell(u_\ell^\star)^2 \le \eta_\ell(\RR_\ell; u_\ell^\star)^2
 \quad \Longrightarrow \quad
 \theta \, \mu_\ell(u_\ell^\kk)^2 \le \mu_\ell(\TT_\ell^{n}[\RR_\ell]; u_\ell^\kk)^2
\end{align}
with $n = m + r$, where $m \in \N_0$ stems from local estimator equivalence~\eqref{eq:equivalence} and $r \in \N_0$ stems from weak stability~\eqref{axiom:weak_stability}.
\end{lemma}

\begin{proof}
Suppose that $\theta_{\rm mark} \, \eta_\ell(u_\ell^\star)^2 \le \eta_\ell(\RR_\ell; u_\ell^\star)^2$.
First, solver contraction~\eqref{eq:algebra:contraction}, \textsl{a-posteriori} error control~\eqref{eq:algebra:aposteriori}, and the solver stopping criterion~\eqref{eq:single:termination} show that
\begin{align*}
 \enorm{u_\ell^\star - u_\ell^\kk}
 \eqreff{eq:algebra:aposteriori}\le
 \crel \, \zeta_\ell(u_\ell^{\kk})
 \eqreff{eq:single:termination}\le
 \lambda \, \crel \, \mu_\ell(u_\ell^\kk)
 &\eqreff{axiom:weak_stability}\le
 \lambda \, \crel \CCstab \, \big[ \mu_\ell(u_\ell^\star) + \enorm{u_\ell^\star - u_\ell^\kk} \big]
 \\&\eqreff{eq:equivalence}\le
 \lambda \, \crel \CCstab \, \big[ \Cwseq \, \eta_\ell(u_\ell^\star) + \enorm{u_\ell^\star - u_\ell^\kk} \big].
\end{align*}
With $\lambda < \lambda^\star = (\crel \CCstab)^{-1}$ from~\eqref{eq1:lemma:doerfler}, it follows that $\lambda \, \crel \CCstab = \lambda/\lambda^\star < 1$ and hence
\begin{align}\label{eq1:doerfler}
 \enorm{u_\ell^\star - u_\ell^\kk}
 \le
 \frac{\lambda/\lambda^\star \, \Cwseq}{1 - \lambda/\lambda^\star} \, \eta_\ell(u_\ell^\star)
 = \frac{\lambda \, \Cwseq}{\lambda^\star - \lambda} \, \eta_\ell(u_\ell^\star).
\end{align}
Second,
weak stability~\eqref{axiom:weak_stability}, local equivalence~\eqref{eq:equivalence}, and the estimate~\eqref{eq1:doerfler} allow for
\begin{align}\label{eq2:doerfler}
 \begin{split}
 &\mu_\ell(u_\ell^\kk)
 \eqreff{axiom:weak_stability}\le
 \CCstab \, \big[ \mu_\ell(u_\ell^\star) + \enorm{u_\ell^\star - u_\ell^\kk} \big]
 \eqreff{eq:equivalence}\le
 \CCstab \, \big[ \Cwseq \, \eta_\ell(u_\ell^\star) + \enorm{u_\ell^\star - u_\ell^\kk} \big]
 \\& \quad
 \eqreff{eq1:doerfler}\le
 \Cwseq \CCstab \, \Big[ 1 + \frac{\lambda}{\lambda^\star - \lambda}\Big] \, \eta_\ell(u_\ell^\star)
 = \Cwseq \CCstab \, \frac{\lambda^\star}{\lambda^\star - \lambda} \, \eta_\ell(u_\ell^\star).
 \end{split}
\end{align}
Third, D\"orfler marking~\eqref{eq2:lemma:doerfler} for $\eta_\ell(u_\ell^\star)$ yields that
\begin{align*}
 \theta_{\rm mark}^{1/2} \, \eta_\ell(u_\ell^\star)
 \eqreff{eq2:lemma:doerfler}\le
 \eta_\ell(\RR_\ell;u_\ell^\star)
 &\eqreff*{eq:equivalence}\le \,
 \Cwseq \, \mu_\ell(\TT_\ell^{m}[\RR_\ell];u_\ell^\star)
 \\&
 \eqreff*{axiom:weak_stability}\le \,
 \Cwseq \CCstab \, \big[\mu_\ell(\TT_\ell^{m+r}[\RR_\ell];u_\ell^\kk) + \enorm{u_\ell^\star - u_\ell^\kk} \big]
 \\&
 \eqreff*{eq1:doerfler}\le \,
 \Cwseq \CCstab \, \mu_\ell(\TT_\ell^{m+r}[\RR_\ell];u_\ell^\kk) + \Cwseq^2 \CCstab \, \frac{\lambda}{\lambda^\star - \lambda} \, \eta_\ell(u_\ell^\star).
\end{align*}
Rearranging this estimate and using $n = m + r$, we arrive at
\begin{align}\label{eq3:doerfler}
 \Big[ \theta_{\rm mark}^{1/2} - \Cwseq^2 \CCstab \, \frac{\lambda}{\lambda^\star - \lambda} \Big] \, \eta_\ell(u_\ell^\star)
 \le \Cwseq \CCstab \, \mu_\ell(\TT_\ell^{n}[\RR_\ell];u_\ell^\kk).
\end{align}
Finally, we combine~\eqref{eq2:doerfler}--\eqref{eq3:doerfler} to see that
\begin{align}\label{eq4:doerfler}
 \notag
 \Big[ \theta_{\rm mark}^{1/2} - \Cwseq^2 \CCstab \, \frac{\lambda}{\lambda^\star - \lambda} \Big] \, \mu_\ell(u_\ell^\kk)
 &\eqreff{eq2:doerfler}\le
 \Cwseq \CCstab \, \frac{\lambda^\star}{\lambda^\star - \lambda} \, \Big[ \theta_{\rm mark}^{1/2} - \Cwseq^2 \CCstab \, \frac{\lambda}{\lambda^\star - \lambda} \Big] \, \eta_\ell(u_\ell^\star)
 \\& 
 \eqreff{eq3:doerfler}\le
 \Cwseq^2 \CCstab^2 \, \frac{\lambda^\star}{\lambda^\star - \lambda} \, \mu_\ell(\TT_\ell^{n}[\RR_\ell];u_\ell^\kk).
\end{align}
By choice of $\theta_{\rm mark}$ in~\eqref{eq1:lemma:doerfler}, we obtain that
\begin{align*}
 \theta^{1/2}
 &\eqreff*{eq1:lemma:doerfler}= \
 \Big[ \theta_{\rm mark}^{1/2} \, \Cwseq^{-2} \CCstab^{-2} \, - \CCstab^{-1} \, \frac{\lambda}{\lambda^\star-\lambda} \Big] \, \frac{\lambda^\star - \lambda}{\lambda^\star}
 \\&
 = \
 \Big[ \theta_{\rm mark}^{1/2} - \Cwseq^2 \CCstab \, \frac{\lambda}{\lambda^\star - \lambda} \Big] \, \Cwseq^{-2} \CCstab^{-2} \, \frac{\lambda^\star - \lambda}{\lambda^\star},
\end{align*}
so that~\eqref{eq4:doerfler} reduces to $\theta^{1/2} \, \mu_\ell(u_\ell^\kk) \le \mu_\ell(\TT_\ell^{m+r}[\RR_\ell];u_\ell^\kk)$. This concludes the proof.
\end{proof}

To formulate
optimal complexity, we require the notion of nonlinear approximation classes as introduced in~\cite{bddp2002, bdd2004, stevenson2007, ckns2008, cfpp2014}:
Given $N \in \N_0$, let $\T(N)$ be the set of all refinements $\TT_H$ of $\TT_0$ with $\#\TT_H - \#\TT_0 \le N$. For any $s > 0$, define
\begin{align}\label{eq:As}
 \norm{u^\star}{\mathbb{A}_s} := \sup_{N \in \N_0} (N+1)^s \inf_{\TT_{\rm opt} \in \T(N)} \eta_{\rm opt}(u_{\rm opt}^\star)
 \in \mathbb{R}_{\ge0} \cup \{\infty\}.
\end{align}
Here, $u_{\rm opt}^\star \in \XX_{\rm opt}$ denotes the exact FE solution to~\eqref{eq:discrete_problem} with respect to the non-available optimal mesh $\TT_{\rm opt}$, where optimality is understood with respect to the residual-based error estimator.
We stress that $\norm{u^\star}{\A_s}$ can equivalently be rephrased in terms of energy error plus data oscillations as used in~\cite{stevenson2007, ckns2008, cn2012}; see, e.g.,~\cite{ffp2014, cfpp2014}.

Based on Lemma~\ref{lemma:doerfler}, minor modifications of the original proof of~\cite[Theorem~8]{ghps2021} yield the following theorem, whose proof is included for the convenience of the reader.

\begin{theorem}[optimal complexity of Algorithm~\ref{algorithm:afem}]\label{theorem:optimal-complexity}
Recall $\lambda^\star$ and $\theta_{\rm mark}$ from~\eqref{eq1:lemma:doerfler}.
Suppose that the adaptivity parameters $\theta, \lambda > 0$ are sufficiently small in the sense that
\begin{equation}\label{eq:optimal_parameters}
    0 < \lambda < \lambda^\star
    \quad \text{and} \quad
    0 < \theta_{\textup{mark}}
    <
    \theta^\star
    \coloneqq
    ( 1 + \Cstab^2 \, \Cdrel^2 )^{-1}.
  \end{equation}
Under the assumptions of Section~\ref{section:abstract_framework}, Algorithm~\ref{algorithm:afem} then guarantees the existence of constants $\copt > 0$ and $\Copt > 0$ such that
  \begin{equation}\label{eq:optimal_complexity}
    \copt \, \| u^\star \|_{\A_s}
    \leq
    \sup_{(\ell, k) \in \QQ} \cost(\ell,k)^s \, \Mu_\ell^{k}
    \leq
    \Copt \, \max \{ \| u^\star \|_{\A_s}, \Mu_{0}^{0} \}.
  \end{equation}
The constant $\Copt$ depends only on $\Cstab$, $\Cdrel$, $\Cmon$, $\Clin$, $\qlin$, $\Cwseq$, $\Cmark$, $\lambda$, $\lambda^\star$, $\Ccls$, $\Cchild$, $\#\TT_0$, $s$, $\theta$, and validity of~\eqref{refinement:shape-regular}.
The constant $\copt$ satisfies $\copt \ge (\Cwseq \Cchild^s)^{-1}$.
\end{theorem}

\begin{proof}[{\bfseries Proof of Theorem~\ref{theorem:optimal-complexity}, upper estimate in~\bf(\ref{eq:optimal_complexity})}]
Without loss of generality, we may suppose that $\norm{u^\star}{\A_s} < \infty$.
The proof is split into three steps.

\textbf{Step~1 (bound for marked elements).}
Let $\ell \in \N_0$ be arbitrary with $(\ell+1, 0) \in \QQ$.
According to the comparison lemma (see, e.g.,~\cite[Lemma~4.14]{cfpp2014}) and the optimality of D\"orfler marking for the residual-based error estimator $\eta_\ell(u_\ell^\star)$ with exact FE solution $u_\ell^\star$ (see, e.g.,~\cite[Proposition~4.12]{cfpp2014}), there exists a set $\RR_\ell \subseteq \TT_\ell$ such that
\begin{align}\label{eq60a}
 \#\RR_\ell \lesssim \norm{u^\star}{\A_s}^{1/s} \, \eta_\ell(u_\ell^\star)^{-1/s}
 \quad \text{and} \quad \theta_{\rm mark} \eta_\ell(u_\ell^\star)^2 \le \eta_\ell(\RR_\ell; u_\ell^\star)^2
\end{align}
for any $0< \theta_{\rm mark} < \theta^\star$, where the hidden constant depends only on $\Cstab$, $\Cdrel$, $\Cmon$, $\theta$, and on the validity of the overlay estimate~\eqref{refinement:overlay}.
According to Lemma~\ref{lemma:doerfler} and the particular choice of $\theta_{\rm mark}$, this yields the D\"orfler-type marking criterion $\theta \mu_\ell(u_\ell^\kk)^2 \le \mu_\ell(\TT_\ell^{n}[\RR_\ell]; u_\ell^\kk)^2$, i.e., $\TT_\ell^{n}[\RR_\ell] \in \mathbb{M}_\ell[\theta, u_\ell^\kk]$ as defined in Algorithm~\ref{algorithm:afem}(iii).
The quasi-minimal choice of $\MM_\ell \in \mathbb{M}_\ell[\theta, u_\ell^\kk]$ and uniform shape regularity~\eqref{refinement:shape-regular} thus guarantee that
\begin{align*}
 \#\MM_\ell \eqreff{eq:doerfler}\le \Cmark \, \# \TT_\ell^{n}[\RR_\ell] \eqreff{refinement:shape-regular}\lesssim \# \RR_\ell \eqreff{eq60a}\lesssim \norm{u^\star}{\A_s}^{1/s} \, \eta_\ell(u_\ell^\star)^{-1/s}.
\end{align*}
From full R-linear convergence~\eqref{eq:single:convergence} and estimator equivalence~\eqref{eq:equivalence}, we infer that
\begin{align*}
 \Mu_{\ell+1}^0
 \eqreff{eq:single:convergence}\lesssim
 \Mu_\ell^\kk
 = \enorm{u_\ell^\star - u_\ell^\kk} + \mu_\ell(u_\ell^\star)
 \eqreff{eq1:doerfler}\lesssim
 \eta_\ell(u_\ell^\star) + \mu_\ell(u_\ell^\star)
 \eqreff{eq:equivalence}\simeq
 \eta_\ell(u_\ell^\star).
\end{align*}
Combining the last two formulas, we see that
\begin{align}\label{eq1:optimality:upper-bound}
 \#\MM_\ell \lesssim \norm{u^\star}{\A_s}^{1/s} \, (\Mu_{\ell+1}^0)^{-1/s}
 \quad \text{for all } \ell \in \N_0 \text{ with } (\ell+1, 0) \in \QQ.
\end{align}
The hidden constant depends only on $\Cstab$, $\Cdrel$, $\Cmon$, $\Clin$, $\qlin$, $\Cwseq$, $\Cmark$, $\lambda/(\lambda^\star-\lambda)$, $\theta$, and on the validity of~\eqref{refinement:overlay} and~\eqref{refinement:shape-regular}.

\textbf{Step~2 (optimal convergence rates).}
Let $(\ell,k) \in \QQ$ 
with $\ell \ge 1$.
The mesh-closure estimate~\eqref{refinement:closure}, full R-linear convergence~\eqref{eq:single:convergence}, and the geometric series yield that
\begin{align*}
 &\# \TT_\ell - \# \TT_0 + 1
 \le 2 \, (\# \TT_\ell - \# \TT_0)
 \eqreff{refinement:closure}\lesssim
 \sum_{\ell'=0}^{\ell-1} \# \MM_{\ell'}
 \\& \qquad
 \eqreff{eq1:optimality:upper-bound}\lesssim
 \norm{u^\star}{\A_s}^{1/s} \, \sum_{\ell'=0}^{\ell-1} (\Mu_{\ell'+1}^0)^{-1/s}
 \le \norm{u^\star}{\A_s}^{1/s} \, \sum_{\substack{(\ell',k') \in \QQ \\ |\ell',k'| \le |\ell,k|}} (\Mu_{\ell'}^{k'})^{-1/s}
 \eqreff{eq:single:convergence}\lesssim
 \norm{u^\star}{\A_s}^{1/s} \, (\Mu_{\ell}^{k})^{-1/s}.
\end{align*}
Rearranging this estimate, we see that
\begin{align}\label{eq2:optimality:upper-bound}
 (\# \TT_\ell - \# \TT_0 + 1)^s \, \Mu_{\ell}^{k} \lesssim \norm{u^\star}{\A_s}
 \quad \text{for all } (\ell,k) \in \QQ \text{ with } \ell \ge 1.
\end{align}
For $\ell = 0$, full R-linear convergence~\eqref{eq:single:convergence} proves that
\begin{align}\label{eq3:optimality:upper-bound}
 (\# \TT_\ell - \# \TT_0 + 1)^s \, \Mu_{\ell}^{k} = \Mu_0^k
 \eqreff{eq:single:convergence}\lesssim
 \Mu_0^0
 \quad \text{for all } (\ell,k) \in \QQ \text{ with } \ell = 0.
\end{align}
Combining~\eqref{eq2:optimality:upper-bound}--\eqref{eq3:optimality:upper-bound}, we thus conclude that
\begin{align}\label{eq5:optimality:upper-bound}
 \sup_{(\ell,k) \in \QQ} (\# \TT_\ell - \# \TT_0 + 1)^s \Mu_\ell^k
 \lesssim \max \{ \norm{u^\star}{\A_s} \,,\, \Mu_0^0 \}.
\end{align}

\textbf{Step~3 (optimal complexity).}
From~\cite[Lemma~22]{bhp2017}, we recall the 
estimate
\begin{align}\label{eq:bhp}
 \#\TT_\ell \le \#\TT_0 \, (\# \TT_\ell - \# \TT_0 + 1)
 \quad \text{for all } (\ell,0) \in \QQ.
\end{align}
Together with~\eqref{eq:corollary:complexity} from Corollary~\ref{corollary:rates:complexity} and~\eqref{eq5:optimality:upper-bound}, we are thus led to
\begin{align*}
 \sup_{(\ell,k) \in \QQ} \cost(\ell,k)^s \, \Mu_\ell^k
 \eqreff{eq:corollary:complexity}\lesssim
 \sup_{(\ell,k) \in \QQ} (\#\TT_\ell)^s \, \Mu_\ell^k
 &\eqreff*{eq:bhp}\le \
 (\#\TT_0)^s
 \sup_{(\ell,k) \in \QQ} (\# \TT_\ell - \# \TT_0 + 1)^s \Mu_\ell^k
 \\&
 \eqreff*{eq5:optimality:upper-bound}\lesssim \
 \max \{ \norm{u^\star}{\A_s} \,,\, \Mu_0^0 \}.
\end{align*}
This concludes the proof of the upper bound of~\eqref{eq:optimal_complexity}, where the overall constant depends only on $\Ccost(s)$, $(\#\TT_0)^s$ and the constant from~\eqref{eq5:optimality:upper-bound}.
\end{proof}

\begin{proof}[{\bfseries Proof of Theorem~\ref{theorem:optimal-complexity}, lower estimate in~\bf(\ref{eq:optimal_complexity})}]
The proof is split into two steps, where Step~1 considers the generic case, while Step~2 considers the remaining cases.

{\bf Step~1 (generic case).}
First, we consider the generic case that $\eell = \infty$ \emph{and} $\mu_{\ell}(u_{\ell}^{\underline k})>0$ for all $\ell\in\N_0$.
According to~\eqref{refinement:splitting}, Algorithm~\ref{algorithm:afem} guarantees that $\#\TT_\ell\to\infty$ as $\ell\to\infty$.
Let $N \in \N_0$.
Choose the maximal $\ell \in \N_0$ such that
$\#\TT_\ell - \#\TT_0 \le N$.
By definition, there holds $\TT_\ell \in \T(N)$.
The choice of $\ell$ guarantees that
\begin{align}
\label{eq:upper bound for As2}
N+1\le \#\TT_{\ell+1} - \#\TT_0
 \le \#\TT_{\ell+1}
 \le \Cchild \#\TT_\ell.
\end{align}
For any $k$ with $(\ell, k) \in \QQ$, estimator equivalence~\eqref{eq:equivalence} 
and $\mu_\ell(u_\ell^\star) \le \Mu_\ell^k$
show that
\[ (N+1)^s \min_{\TT_H \in \T(N)} \eta_H(u_H^\star)
 \eqreff{eq:upper bound for As2}\le
 \Cchild^s(\#\TT_\ell)^s  \,  \eta_\ell(u_\ell^\star)
 \eqreff{eq:equivalence}\le
 \Cwseq \Cchild^s(\#\TT_\ell)^s  \, \Mu_\ell^k.
\]Taking the supremum over all $N \in \N_0$, we conclude that
\begin{align*}
 (\Cwseq \Cchild^s)^{-1} \, \norm{u^\star}{\mathbb{A}_s}
 \le
 \sup_{(\ell,k) \in \QQ}(\#\TT_\ell)^s \, \Mu_\ell^k
 \le \sup_{(\ell,k) \in \QQ} \cost(\ell,k)^s \, \Mu_\ell^k.
\end{align*}
This proves the lower estimate of~\eqref{eq:optimal_complexity} with $\copt = (\Cwseq \Cchild^s)^{-1}$.

\textbf{Step~2 (special cases).}
If $\eell < \infty$, then Corollary~\ref{corollary:lucky} proves that $u^\star = u_\eell^\star$ with $\mu_\eell(u_\eell^\star) = 0$.
If there exists $\ell_0 \in \N_0$ with $\mu_{\ell_0}(u_{\ell_0}^\kk) = 0$, then
it follows that
\begin{align*}
 \enorm{u_{\ell_0}^\star - u_{\ell_0}^\kk}
 \eqreff{eq:algebra:aposteriori}\lesssim
 \zeta_{\ell_0}(u_{\ell_0}^{\kk})
 \eqreff{eq:single:termination}\le
 \lambda \mu_{\ell_0}(u_{\ell_0}^\kk)
 = 0.
\end{align*}
This results in $u_{\ell_0}^\star = u_{\ell_0}^\kk$ and $\mu_{\ell_0}(u_{\ell_0}^\star) = \mu_{\ell_0}(u_{\ell_0}^\kk) = 0$.
In either case, there exists a minimal $\ell' \in \N_0$ with $\mu_{\ell'}(u_{\ell'}^\star) = 0$.
Then, $\eta_{\ell'}(u_{\ell'}^\star) = 0$ by estimator equivalence~\eqref{eq:equivalence}
and hence
\begin{align*}
\norm{u^\star}{\mathbb{A}_s} = \max_{0 \le N < \#\TT_{\ell'}-\#\TT_0}(N+1)^s \min_{\TT_H \in \T(N)} \eta_H(u_H^\star).
\end{align*}
Let $N < \#\TT_{\ell'}-\#\TT_0$. Choose the maximal $\ell \in \N_0$ such that $\#\TT_\ell - \#\TT_0 \le N$.
By definition, there holds $\TT_\ell \in \T(N)$.
The choice of $\ell$ guarantees that
\begin{align*}
N+1\le \#\TT_{\ell+1} - \#\TT_0
 \le \#\TT_{\ell+1}
 \le \Cchild \#\TT_\ell.
\end{align*}
This proves~\eqref{eq:upper bound for As2} also in the present case.
As in Step~1, we conclude the lower estimate of~\eqref{eq:optimal_complexity} with $\copt = (\Cwseq \Cchild^s)^{-1}$.
\end{proof}

\section{Error estimators based on ZZ-type averaging}
\label{section:zz}

In this section, we consider averaging-based error estimators, which are also referred to as ZZ-type error estimators after the seminal work of Zienkiewicz and Zhu \cite{ZZ87}.
These estimators are widely used in computational science and engineering due to their ease of implementation and impressive performance in various
applications;
see~\cite{ao2000,bc2002a,bc2002b,cfpp2014} and the references therein.
For the Poisson model problem, local equivalence to the residual-based error estimator has already been considered in~\cite{ks2011, cn2012} for lowest-order FEM ($p=1$) and in~\cite{cfpp2014} for higher-order FEM ($p \ge 1$).

\subsection{Model problem}
\label{section:zz:pde}

Throughout this section, we suppose scalar diffusion $\boldsymbol{A} = \alpha \boldsymbol{I}$ for some $\alpha \in C(\overline{\Omega})$ with
\begin{align}\label{eq:zz:alpha}
 0 < \alpha_{\min} \le \alpha(x) \le \alpha_{\max} < \infty
 \quad \text{for all }
 x \in \overline\Omega
\end{align}
and $\boldsymbol{f} = 0$ in problem~\eqref{eq:general_second_order_problem}, i.e., we consider the PDE
\begin{equation}\label{eq:model-problem-zz}
  -\div(\alpha \nabla u^\star)
  + cu^\star = f \quad
  \text{in } \Omega \quad \text{with} \quad  u^\star = 0
  \quad \text{on } \partial \Omega.
\end{equation}
As in Section \ref{section:model-problem}, we suppose
$c \in L^\infty(\Omega)$ and $f \in L^2(\Omega)$.
Recall the residual-based estimator $\eta_H$ from~\eqref{eq:residual_based_estimator} with the local contributions
\begin{equation}\label{eq:zz:eta}
  \eta_H(T; v_H)^2
  =
  |T|^{2/d} \, \| R_H(v_H) \|^2_{L^2(T)}
  +
  |T|^{1/d} \,  \| \jump{\alpha \nabla v_H \cdot \boldsymbol{n}} \|^2_{L^2(\partial T \cap \Omega)}
  \quad \text{on } T \in \TT_H
\end{equation}
including the local volume residual
\begin{equation}
  \label{eq:zz:residual}
  R_H(v_H) \coloneqq \div(\alpha \nabla v_H)
  - c v_H + f,
\end{equation}
where $\div(\cdot)$ is understood $\TT_H$-elementwise.
As in Section~\ref{section:model-problem}, the additional regularity assumption $\alpha|_T \in W^{1,\infty}(T)$ for all initial simplices $T \in \TT_0$ 
ensures well-definedness of $\eta_H$.

\subsection{Definition of ZZ-estimator and corresponding main result}
\label{section:zz:definition}
For a mesh $\TT_H \in \T$, the definition of the ZZ-type error estimator employs a local averaging operator $G_H\colon L^2(\Omega) \to \SS^p(\TT_H)$,
where we recall the vertex patch $\Omega_H[z] = \bigcup\TT_H[z]$ from Section~\ref{section:zz:notation} and the nodal basis $\{\phi_{H,z}: z \in \NN_H\}$ of $\SS^p(\TT_H)$ from Section~\ref{section:fem}.

{\Large$\bullet\,$}
    For lowest-order polynomials with $p = 1$, we can consider a patchwise averaging operator $G_H\colon L^2(\Omega) \to \SS^1(\TT_H)$.
    We define
    \begin{equation}\label{eq:G-patch-averaging}
      G_H (v)
      \coloneqq
      \sum_{z \in \VV_H} \bigg(\frac{1}{|\Omega_H[z]|} \int_{\Omega_H[z]} v \d{x}\bigg) \phi_{H,z} \quad \text{for all functions } v \in L^2(\Omega).
    \end{equation}

{\Large$\bullet\,$}
For general $p \in \N$, we can consider an $L^2$-stable variant $G_H$ of the Scott--Zhang projection from \cite{SZ90}.
    With each node $z\in\NN_H$, we associate an attached simplex $S_{H,z} \in \TT_H[z]$.
    Linear algebra yields the existence and uniqueness of dual basis functions
    $\psi_{H,z} \in \spann\{\phi_{H,z'}|_{S_{H,z}} : z' \in S_{H,z}\cap \NN_H\}$ satisfying
    \begin{equation}\label{eq:Scott-Zhang-dual-basis}
      (\psi_{H,z}, \phi_{H,z'})_{L^2(S_{H,z})} = \delta_{zz'} \quad \text{for all } z,z'\in\NN_H.
    \end{equation}
    The $L^2$-stable Scott--Zhang projection $G_H\colon L^2(\Omega) \to \SS^p(\TT_H)$ is then defined by
    \begin{equation}\label{eq:G-Scott-Zhang}
      G_H(v)
      \coloneqq
      \sum_{z\in\NN_H} (\psi_{H,z}, v)_{L^2(S_{H,z})} \, \phi_{H,z}
      \quad \text{for all } v \in L^2(\Omega).
    \end{equation}

Let $q \in \N_0$ with $q \le p-1$.
For any subset $U \subseteq \overline{\Omega}$, let $\PP^q(U)$ denote the space of all polynomials of total degree at most $q$.
Moreover, let $\Pi^{q}(U)\colon L^2(U) \to \PP^{q}(U)$ denote the corresponding $L^2$-orthogonal projection.
For each interior vertex $z \in \VV_H^\Omega$ and each discrete function $v_H \in \XX_H$, we define
\begin{equation}\label{eq:def-r_z}
  r_{H,z}(v_H) \coloneqq \Pi^{q}(\Omega_H[z])(R_H(v_H)) = \argmin_{w_H \in \PP^{q}(\Omega_H[z])} \| R_H(v_H) - w_H \|^2_{L^2(\Omega_H[z])}.
\end{equation}
For $T \in \TT_H$ with vertex $z \in \VV_H^\Omega \cap T$, let
$H(z) \coloneqq |\Omega_H[z]|^{1/d} \simeq |T|^{1/d} = H(T)$.
With this notation, the local contributions of the ZZ-type error estimator are defined by
\begin{equation}\label{eq:ZZ-estimator}
  \mu_H(T; v_H)^2 \coloneqq \| \alpha^{1/2} (1-G_H) \nabla v_H \|^2_{L^2(T)} + \!\!\! \sum_{z \in \VV_H^\Omega \cap T} \!\! \frac{H(z)^2}{\#\TT_H[z]} \, \| R_H(v_H) -
  r_{H,z}(v_H)\|^2_{L^2(\Omega_H[z])},
\end{equation}
where we employ the operator $G_H\colon [L^2(\Omega)]^d \to [\SS^p(\TT_H)]^d$ component-wise to vector-valued functions.
The following theorem is the main result of this section.

\begin{theorem}[convergence of Algorithm~\ref{algorithm:afem} driven by ZZ-type error estimator]\label{cor:optimal-complexity-zz-estimator}
  Suppose that the initial mesh $\TT_0$ is sufficiently fine in the sense that
  $\VV_0^\Omega \cap T \neq \emptyset$ for all $T \in \TT_0$, i.e., each element $T \in \TT_0$ of the initial mesh $\TT_0$ contains at least one interior vertex.
  Let the polynomial degrees $p \in \N$ and $q \le p-1$ be arbitrary.
  For the model problem~\eqref{eq:model-problem-zz}, we consider Algorithm~\ref{algorithm:afem} steered by the ZZ-type error estimator $\mu_\ell$ defined in~\eqref{eq:ZZ-estimator}.
  Then, $\mu_H$ satisfies the assumptions~\eqref{eq:equivalence}--\eqref{axiom:weak_stability} of Section~\ref{section:nonresidual-estimator} (with $m = 2$ and $r = 0$) so that the two main results of Section~\ref{section:convergence} apply:
  Theorem \ref{theorem:full_linear_convergence} guarantees full R-linear convergence of the quasi-error~\eqref{eq2:single:quasi-error} for any choice of the adaptivity parameters $0<\theta\leq1$, $C_{\mathrm{mark}}\geq1$, and $\lambda>0$.
  Theorem \ref{theorem:optimal-complexity} guarantees optimal complexity of Algorithm~\ref{algorithm:afem}
  provided that $\theta$ and $\lambda$ are chosen sufficiently small.
\end{theorem}

The proof of Theorem~\ref{cor:optimal-complexity-zz-estimator} follows as soon as we verify local equivalence~\eqref{eq:equivalence} and weak stability~\eqref{axiom:weak_stability} of the ZZ-type error estimator~\eqref{eq:ZZ-estimator}. This is done in 
Sections~\ref{section:zz:equivalence}--\ref{section:zz:weak-stability}.

\subsection{Proof of local equivalence~(\ref{eq:equivalence}) for ZZ-type error estimator}
\label{section:zz:equivalence}

In order to show equivalence of the ZZ-type error estimator~\eqref{eq:ZZ-estimator} and the residual-based estimator~\eqref{eq:zz:eta},
we recall two auxiliary results, which are essentially contained in \cite[Section~9]{cfpp2014}.
The first lemma provides a bound for the residual in terms of normal jumps and patch oscillations,
where we recall the vertex skeleton $\Sigma_H^\Omega[z] = \bigcup\EE_H^\Omega[z]$ from Section~\ref{section:zz:notation}.

\begin{lemma}[{\cite[Lemma 9.5]{cfpp2014}}]\label{lemma:residual-bound}
There exists a constant $\Cres > 0$ depending only on the dimension $d$, the polynomial degree $q \le p-1$, and uniform shape regularity~\eqref{refinement:shape-regular} such that, for all $\TT_H \in \T$, all $T \in \TT_H$, and all interior vertices
$z \in \VV_H^\Omega \cap T$, it holds that
  \begin{equation}\label{eq:residual-bound}
    \begin{aligned}
      H(T)^2 \, \|R_H(u_H^\star)\|^2_{L^2(T)}
      \leq \Cres \Big[ &H(z) \, \| \jump{\alpha \nabla u_H^\star \cdot \boldsymbol{n}} \|^2_{L^2(\Sigma_H^\Omega[z])} 
      \\&
      + H(z)^2 \, \|
        R_H(u_H^\star) -
      r_{H,z}(u_H^\star)\|^2_{L^2(\Omega_H[z])} \Big].
      \quad \qed
    \end{aligned}
  \end{equation}
\end{lemma}

Based on a seminorm argument, the next lemma shows that the normal jumps are locally equivalent to averaging.
We stress that the statement in~\cite{cfpp2014} is slightly weaker, while the present formulation states what is actually proved there.
The proof 
relies on the fact that mesh refinement by NVB leads only to finitely many shapes of elements and hence patches.
In this respect,  
the constants depend on the use of NVB.

\begin{lemma}[{\cite[Lemma 9.10]{cfpp2014}}]\label{lemma:averaging-equivalent-to-jumps}
There exists a constant $\Cavg > 0$ depending only on the dimension $d$, the polynomial degree $p$, the initial mesh $\TT_0$, and the use of NVB such that for all $\TT_H \in \T$, all $v_H \in \XX_H$, and all $z \in \VV_H$, it holds that
 \begin{subequations}\label{eq:averaging-equivalent-to-jumps}
  \begin{align}\label{eq:averaging-lower-bound}
    H(z) \| \jump{\nabla v_H \cdot \boldsymbol{n}}\|_{L^2(\Sigma_H^\Omega[z])}^2
    &\leq
    \Cavg \,
    \| (1 - G_H) \nabla v_H \|_{L^2(\Omega_H[z])}^2
  \intertext{as well as}
   \label{eq:averaging-upper-bound}
   \| (1 - G_H) \nabla v_H \|_{L^2(\Omega_H[z])}^2
   &\leq
    \Cavg \, \sum_{z' \in \VV_H \cap \Sigma_H^\Omega[z]} H(z') \, \| \jump{\nabla v_H \cdot \boldsymbol{n}}\|_{L^2(\Sigma_H^\Omega[z'])}^2.
  \end{align}
  \end{subequations}
\end{lemma}

\begin{proof}[Sketch of proof]
We argue by equivalence of seminorms on finite-dimensional spaces and a scaling argument. 
If the right-hand side of~\eqref{eq:averaging-lower-bound} vanishes, then $\nabla v_H = G_H\nabla v_H \in \SS^p(\TT_H[z]) \subset C(\Omega_H[z])$.
Hence, also the left-hand side of~\eqref{eq:averaging-lower-bound} vanishes by continuity.
This yields~\eqref{eq:averaging-lower-bound}.
If the right-hand side of~\eqref{eq:averaging-upper-bound} vanishes, then the normal jumps of $\nabla v_H$ vanish on $\Sigma_H^\Omega[z']$ for all vertices $z' \in \VV_H$ that lie on an interior facet with $z$.
Since $v_H \in H^1(\Omega)$, also the tangential jumps of $\nabla v_H$ vanish on $\Sigma_H^\Omega[z']$.
This implies that $\nabla v_H \in \PP^{p-1}(\TT_H[z']) \cap C(\Omega_H[z']) = \SS^{p-1}(\TT_H[z'])$.
In either case, the definition of $G_H$ thus implies $G_H \nabla v_H = \nabla v_H$.
Hence, also the left-hand side of~\eqref{eq:averaging-upper-bound} vanishes.
This yields~\eqref{eq:averaging-upper-bound}.
The constant $\Cavg$ depends only of $p$ and the shapes of the patches involved. 
Since NVB leads only to finitely many shapes, this concludes the proof.
\end{proof}

Finally, we can thus
prove that the ZZ-type error estimator is locally equivalent to the residual-based estimator~$\eta_H$ in the sense of~\eqref{eq:equivalence}.

\begin{proposition}[local equivalence of ZZ-type error estimator]
\label{prop:ZZ-estimator}
Under the assumptions of Theorem~\ref{cor:optimal-complexity-zz-estimator},
the ZZ-type error estimator $\mu_H$ from~\eqref{eq:ZZ-estimator} is equivalent to the residual-based
  estimator~$\eta_H$ from~\eqref{eq:zz:eta} in the sense of~\eqref{eq:equivalence} with $m = 2$.
  The equivalence constant $\Cwseq$ depends only on the polynomial degrees $p$ and $q \le p - 1$,
  $\alpha_{\rm min}$ and $\alpha_{\rm max}$ from~\eqref{eq:zz:alpha},
  the initial mesh $\TT_0$, and the use of NVB.
\end{proposition}

\begin{proof}
  Let $\TT_H \in \T$ and $T \in \TT_H$.
  Since $\alpha \in C(\overline{\Omega})$ satisfies~\eqref{eq:zz:alpha}, there holds, for all $z \in \VV_H$,
  \begin{align}
    \label{eq:remove-alpha:jump}
    \| \jump{\alpha \nabla u_H^\star \cdot \boldsymbol{n}} \|_{L^2(\Sigma_H^\Omega[z])}
    = \| \alpha \, \jump{ \nabla u_H^\star \cdot \boldsymbol{n}} \|_{L^2(\Sigma_H^\Omega[z])}
    &\eqreff{eq:zz:alpha}\simeq
    \| \jump{ \nabla u_H^\star \cdot \boldsymbol{n}} \|_{L^2(\Sigma_H^\Omega[z])},
    \\
    \label{eq:remove-alpha:volume}
    \| \alpha^{1/2} (1-G_H) \nabla u_H^\star \|_{L^2(\Omega_H[z])}
    &\eqreff{eq:zz:alpha}\simeq
    \| (1-G_H) \nabla u_H^\star \|_{L^2(\Omega_H[z])}.
  \end{align}
  Moreover, uniform shape regularity~\eqref{refinement:shape-regular} guarantees that the number of patch elements $\#\TT_H[z] \lesssim 1$ is uniformly bounded for all $z \in \VV_H$ and that $H(z) \simeq H(T')$ for all $T' \in \TT_H[z]$.
  The remainder of the proof is split into five steps.

\textbf{Step~1 (assumption on $\boldsymbol{\TT\!_0}$ is inherited).}
Note that the application of Lemma~\ref{lemma:residual-bound} requires that each element of the considered mesh $\TT_H \in \T$ contains an interior vertex $z \in \VV_H^\Omega \cap T$, since only then the local oscillation term in~\eqref{eq:ZZ-estimator} controls the volume residual $\| R_H(v_H) \|_{L^2(T)}$.
This property is inherited by every NVB refinement: 
If $\VV_0^\Omega \cap T \neq \emptyset$ for all $T \in \TT_0$, then $\VV_h^\Omega \cap T \neq \emptyset$ for all $\TT_h \in \T$ and all $T \in \TT_h$.
Indeed, NVB only adds vertices. Whenever an element $T$ is bisected, whose only interior vertices are the endpoints $z_j, z_k$ of its refinement edge, then the new midpoint $m = (z_j + z_k)/2$ lies in $\Omega$.
Otherwise, $m$ would belong to a facet on $\partial\Omega$.
In this case,
the entire refinement edge would be on $\partial\Omega$ and hence both $z_j$ and $z_k$, contradicting $\{z_j, z_k\} \cap \Omega \neq \emptyset$. In conclusion, both children of $T$ retain an interior vertex.
Overall, the assumptions of Theorem~\ref{cor:optimal-complexity-zz-estimator} thus guarantee that 
$\VV_H^\Omega \cap T \neq \emptyset$ for all $T \in \TT_H$ for \emph{every} mesh $\TT_H \in \T$.

  \textbf{Step~2 ($\boldsymbol{\eta_H \lesssim \mu_H}$, jump term).}
  Let $E \in \EE_H^\Omega$ with $E \subseteq \partial T$ and let $z_E \in \VV_H \cap E$ be an arbitrary vertex of $E$.
  Since $E \subseteq \partial T$ and $z_E \in E$, it follows that $T \in \TT_H[z_E]$ and hence $\TT_H[z_E] \subseteq \TT_H[T]$ as well as $\Omega_H[z_E] \subseteq \Omega_H[T]$.
  Uniform shape regularity~\eqref{refinement:shape-regular} yields that $H(T) \simeq H(z_E)$.
  With $E \subseteq \Sigma_H^\Omega[z_E]$ and the lower bound of~\eqref{eq:averaging-equivalent-to-jumps}, we obtain that
  \begin{align*}
    &H(T) \, \| \jump{\alpha \nabla u_H^\star \cdot \boldsymbol{n}} \|^2_{L^2(E)}
    \eqreff*{eq:remove-alpha:jump}\simeq
    H(z_E) \, \| \jump{\nabla u_H^\star \cdot \boldsymbol{n}} \|^2_{L^2(E)}
\le
    H(z_E) \, \| \jump{\nabla u_H^\star \cdot \boldsymbol{n}} \|^2_{L^2(\Sigma_H^\Omega[z_E])}
    \\& \quad
    \eqreff{eq:averaging-lower-bound}\lesssim
    \| (1 - G_H) \nabla u_H^\star \|^2_{L^2(\Omega_H[z_E])}
    \eqreff*{eq:remove-alpha:volume}\simeq
    \| \alpha^{1/2} (1 - G_H) \nabla u_H^\star \|^2_{L^2(\Omega_H[z_E])}
    \eqreff{eq:ZZ-estimator}\le
    \mu_H(\TT_H[T]; u_H^\star)^2.
  \end{align*}
  Summation over the (at most $d+1$) facets $E \in \EE_H^\Omega$ with $E \subseteq \partial T$ proves that
  \begin{align}\label{eq:zz:jump-term}
    H(T) \, \| \jump{\alpha \nabla u_H^\star \cdot \boldsymbol{n}} \|^2_{L^2(\partial T \cap \Omega)}
    \lesssim
    \mu_H(\TT_H[T]; u_H^\star)^2.
  \end{align}

  \textbf{Step~3 ($\boldsymbol{\eta_H \lesssim \mu_H}$, volume term).}
  Due to Step~1, we may choose an interior vertex $z_T \in \VV_H^\Omega \cap T$. Then,
  Lemma~\ref{lemma:residual-bound} applies and yields that
  \begin{align}\label{eq:zz:volume-split}
    H(T)^2 \, \| R_H(u_H^\star) \|^2_{L^2(T)}
    &\eqreff{eq:residual-bound}\lesssim
    H(z_T) \, \| \jump{\alpha \nabla u_H^\star \cdot \boldsymbol{n}} \|^2_{L^2(\Sigma_H^\Omega[z_T])}
    \nonumber \\ & \qquad
    +
    H(z_T)^2 \, \| R_H(u_H^\star) - r_{H,z_T}(u_H^\star) \|^2_{L^2(\Omega_H[z_T])}.
  \end{align}
  We bound the two terms on the right-hand side of~\eqref{eq:zz:volume-split} separately.
  For the jump contribution~of~\eqref{eq:zz:volume-split}, we argue precisely as in Step~2 (with $z_T$ in place of $z_E$) to see that
  \begin{align*}
    &H(z_T) \, \| \jump{\alpha \nabla u_H^\star \cdot \boldsymbol{n}} \|^2_{L^2(\Sigma_H^\Omega[z_T])}
    \eqreff*{eq:remove-alpha:jump}\simeq
    H(z_T) \, \| \jump{\nabla u_H^\star \cdot \boldsymbol{n}} \|^2_{L^2(\Sigma_H^\Omega[z_T])}
    \lesssim \mu_H(\TT_H[T]; u_H^\star)^2.
  \end{align*}
  For the oscillation contribution~of~\eqref{eq:zz:volume-split}, we note that 
  $z_T \in \VV_H^\Omega \cap T$
  so that the vertex contribution associated with $z_T$ appears in 
  $\mu_H(T; u_H^\star)^2$.
  Since $\#\TT_H[z_T] \lesssim 1$ by~\eqref{refinement:shape-regular},
  the definition~\eqref{eq:ZZ-estimator} of $\mu_H$ reveals that
  \begin{align*}
    H(z_T)^2 \, \| R_H(u_H^\star) - r_{H,z_T}(u_H^\star) \|^2_{L^2(\Omega_H[z_T])}
    \eqreff{eq:ZZ-estimator}\lesssim
    \mu_H(T; u_H^\star)^2.
  \end{align*}
  Combining the last two estimates with~\eqref{eq:zz:volume-split}, we arrive at
  \begin{align*}
    H(T)^2 \, \| R_H(u_H^\star) \|^2_{L^2(T)}
    \lesssim
    \mu_H(\TT_H[T]; u_H^\star)^2.
  \end{align*}
  Together with~\eqref{eq:zz:jump-term} and the definition~\eqref{eq:zz:eta} of $\eta_H$, this proves that
  \begin{align}\label{eq:zz:eta-le-mu:local}
    \eta_H(T; u_H^\star)^2
    \lesssim
    \mu_H(\TT_H[T]; u_H^\star)^2,
  \end{align}
  where the hidden constant depends only on the polynomial degrees $p$ and $q$, the quotient $\alpha_{\max}^2 / \alpha_{\min}$, the initial mesh~$\TT_0$, and the use of NVB.

  \textbf{Step~4 ($\boldsymbol{\mu_H \lesssim \eta_H}$).}
  We treat the two terms in~\eqref{eq:ZZ-estimator} individually.
  According to Step~1, let
  $z \in \VV_H^\Omega \cap T$ be an arbitrary interior vertex of $T$.
  In particular, $T \subseteq \Omega_H[z]$ and $\TT_H[z] \subseteq \TT_H[T]$.
  For each $z' \in \VV_H \cap \Sigma_H^\Omega[z]$, it holds that $z' \in \Omega_H[z]$ and therefore $\TT_H[z'] \subseteq \TT_H[\Omega_H[z]] = \TT_H^2[z] \subseteq \TT_H^2[T]$.
  Hence, the upper bound of~\eqref{eq:averaging-equivalent-to-jumps} yields that
  \begin{align*}
      &\| \alpha^{1/2} (1-G_H) \nabla u_H^\star \|^2_{L^2(T)}
      \eqreff{eq:remove-alpha:volume}\lesssim
      \| (1-G_H) \nabla u_H^\star \|^2_{L^2(\Omega_H[z])}
      \eqreff{eq:averaging-upper-bound}{\lesssim}
      \!\!\!\! \sum_{z' \in \VV_H \cap \Sigma_H^\Omega[z]} \!\!\!\!
      H(z') \, \| \jump{\nabla u_H^\star \cdot \boldsymbol{n}}\|_{L^2(\Sigma_H^\Omega[z'])}^2
      \\& \qquad
      \eqreff{eq:remove-alpha:jump}{\simeq}
      \!\!\!\! \sum_{z' \in \VV_H \cap \Sigma_H^\Omega[z]} \!\!\!\!
      H(z') \, \| \jump{\alpha \nabla u_H^\star \cdot \boldsymbol{n}}\|_{L^2(\Sigma_H^\Omega[z'])}^2
      \eqreff{refinement:shape-regular}{\lesssim}
      \eta_H(\TT_H^2[T]; u_H^\star)^2,
  \end{align*}
  where the last estimate follows from the finite overlap of the patches $\TT_H[z'] \subseteq \TT_H^2[T]$ together with $H(z') \simeq H(T')$ for all $T' \in \TT_H[z']$ due to uniform shape regularity~\eqref{refinement:shape-regular}.
  For the second term of~\eqref{eq:ZZ-estimator}, recall that $r_{H,z}(u_H^\star) = \Pi^{q}(\Omega_H[z]) \, R_H(u_H^\star)$.
  Thus, the best approximation property of the $L^2$-orthogonal projection $\Pi^{q}(\Omega_H[z])$, uniform shape regularity~\eqref{refinement:shape-regular}, and the finite overlap of the patches $\Omega_H[z]$ prove that
  \begin{align*}
      &\sum_{z \in \VV_H^\Omega \cap T} \frac{H(z)^2}{\#\TT_H[z]} \, \| R_H(u_H^\star) - r_{H,z}(u_H^\star)\|^2_{L^2(\Omega_H[z])}
      \le \sum_{z \in \VV_H^\Omega \cap T} H(z)^2 \, \| R_H(u_H^\star) \|^2_{L^2(\Omega_H[z])}
      \\& \qquad
      \eqreff*{refinement:shape-regular}\lesssim
      \sum_{z \in \VV_H^\Omega \cap T} \, \sum_{T' \in \TT_H[z]} H(T')^2 \, \| R_H(u_H^\star) \|^2_{L^2(T')}
      \eqreff*{eq:zz:eta}\lesssim \;\; \eta_H(\TT_H[T]; u_H^\star)^2.
  \end{align*}
  Combining the last two estimates with $\TT_H[T] \subseteq \TT_H^2[T]$, we conclude that
  \begin{align}\label{eq:zz:mu-le-eta:local}
      \mu_H(T; u_H^\star)^2
      \lesssim
      \eta_H(\TT_H^2[T]; u_H^\star)^2,
  \end{align}
  where the hidden constant depends only on the polynomial degrees $p$ and $q$, the quotient $\alpha_{\max} / \alpha_{\min}^2$, the initial mesh~$\TT_0$, and the use of NVB.

  \textbf{Step~5 (from elementwise to setwise equivalence).}
  Let $\UU_H \subseteq \TT_H$ be arbitrary.
  Summing~\eqref{eq:zz:eta-le-mu:local} over all $T \in \UU_H$ and exploiting the finite overlap of the patches $\TT_H[T] \subseteq \TT_H[\UU_H]$ guaranteed by uniform shape regularity~\eqref{refinement:shape-regular}, we obtain that
  \begin{align*}
    &\eta_H(\UU_H; u_H^\star)^2
    = \sum_{T \in \UU_H} \eta_H(T; u_H^\star)^2
    \eqreff{eq:zz:eta-le-mu:local}\lesssim
    \sum_{T \in \UU_H} \mu_H(\TT_H[T]; u_H^\star)^2
    \eqreff{refinement:shape-regular}\lesssim
    \mu_H(\TT_H[\UU_H]; u_H^\star)^2.
  \end{align*}
  Analogously, summation of~\eqref{eq:zz:mu-le-eta:local} over all $T \in \UU_H$ 
  proves that
  \begin{align*}
    \mu_H(\UU_H; u_H^\star)^2
    &= \sum_{T \in \UU_H} \mu_H(T; u_H^\star)^2
    \eqreff{eq:zz:mu-le-eta:local}\lesssim
    \sum_{T \in \UU_H} \eta_H(\TT_H^2[T]; u_H^\star)^2
    \eqreff{refinement:shape-regular}\lesssim
    \eta_H(\TT_H^2[\UU_H]; u_H^\star)^2.
  \end{align*}
  This concludes the proof of~\eqref{eq:equivalence} with $m = 1$ and $m = 2$, respectively, where the equivalence constant $\Cwseq$ depends only on the polynomial degrees $p$ and $q \le p-1$, on $\alpha_{\min}$ and $\alpha_{\max}$ from~\eqref{eq:zz:alpha}, on the initial mesh $\TT_0$, and on the use of NVB.
\end{proof}

\subsection{Proof of weak stability~(\ref{axiom:weak_stability}) for ZZ-estimator}
\label{section:zz:weak-stability}

To conclude the proof of Theorem~\ref{cor:optimal-complexity-zz-estimator}, it 
remains to establish the weak stability~\eqref{axiom:weak_stability} of the ZZ-type error estimator.

\begin{proposition}[weak stability of ZZ-estimator]
\label{prop:ZZ-weak-stability}
The ZZ-type error estimator from~\eqref{eq:ZZ-estimator} satisfies weak stability~\eqref{axiom:weak_stability} with $r = 0$.
The constant $\CCstab$ depends only on the dimension \(d \in \N\), the polynomial degree $p$,
the ellipticity constant $\Cell > 0$,
$\alpha_{\rm min}$ and $\alpha_{\rm max}$ from~\eqref{eq:zz:alpha},
and uniform shape regularity~\eqref{refinement:shape-regular}.
\end{proposition}

\begin{proof}
  Let $\TT_H \in \T$. The proof is divided into three steps.

  \textbf{Step 1 (local $\boldsymbol{L^2}$-stability).}
  Either choice of the averaging operator $G_H \colon L^2(\Omega) \to L^2(\Omega)$ in Section~\ref{section:zz:definition} guarantees that
  \begin{equation*}
    \begin{aligned}
      \| G_H v \|_{L^2(T)} \;
  \lesssim
      \| v \|_{L^2(\Omega_H[T])}
  \quad \text{for all } T \in \TT_H \text{ and all } v \in L^2(\Omega),
    \end{aligned}
  \end{equation*}
  where the hidden constant depends only on $d$, $p$, and uniform shape regularity~\eqref{refinement:shape-regular} of $\TT_H \in \T$.
  With~\eqref{eq:zz:alpha}, this leads to
  \begin{equation}\label{eq:G1-L2-stable}
    \begin{aligned}
      \| \alpha^{1/2} (1-G_H) v \|_{L^2(T)}
      \eqreff{eq:zz:alpha}\simeq
      \| (1-G_H) v \|_{L^2(T)}
      \lesssim
      \| v \|_{L^2(\Omega_H[T])}
      \eqreff{eq:zz:alpha}\simeq
      \| \alpha^{1/2} v \|_{L^2(\Omega_H[T])}.
    \end{aligned}
  \end{equation}

\textbf{Step 2 (local $\boldsymbol{H^1}$-stability).}
  Let $T \in \TT_H$ and $v_H, w_H \in \XX_H$.
  For the averaging contribution to the ZZ-estimator from~\eqref{eq:ZZ-estimator}, the triangle inequality proves that
  \begin{align}\label{eq:G-triangle-inequality}
    \begin{split}
      &\big| \| \alpha^{1/2} (1-G_H) \nabla v_H \|_{L^2(T)} - \| \alpha^{1/2} (1-G_H) \nabla w_H \|_{L^2(T)} \big|
      \\& \quad
      \le \| \alpha^{1/2} (1-G_H) \nabla (v_H - w_H) \|_{L^2(T)}
      ~\eqreff{eq:G1-L2-stable}\lesssim \| \alpha^{1/2} \nabla (v_H - w_H) \|_{L^2(\Omega_H[T])}.
    \end{split}
  \end{align}
  Next, we consider the second term of~\eqref{eq:ZZ-estimator}.
  Since $r_{H,z}(\cdot) = \Pi^{q}(\Omega_H[z]) \, R_H(\cdot)$,
  the linearity and the best approximation property of the orthogonal projection $\Pi^{q}(\Omega_H[z])$ prove that
  \begin{align}\label{eq:R-triangle-inequality}
       &\Big|\Big( \sum_{z \in \VV_H^\Omega \cap T}  \frac{H(z)^2}{\#\TT_H[z]} \| R_H(v_H) - r_{H,z}(v_H)
      \|^2_{L^2(\Omega_H[z])}  \Big)^{1/2}
      \nonumber
      \\ & \qquad\qquad
      -
      \Big( \sum_{z \in \VV_H^\Omega \cap T} \frac{H(z)^2}{\#\TT_H[z]} \| R_H(w_H) - r_{H,z}(w_H) \|^2_{L^2(\Omega_H[z])}  \Big)^{1/2} \Big|
      \nonumber
      \\& \quad
      \le \Big( \sum_{z \in \VV_H^\Omega \cap T} \frac{H(z)^2}{\#\TT_H[z]} \| R_H(v_H) - R_H(w_H) \|^2_{L^2(\Omega_H[z])} \Big)^{1/2}.
  \end{align}
  As in the proof of stability~\eqref{axiom:stability} of the residual-based estimator (see, e.g., \cite[Corollary~3.4]{ckns2008}), an inverse estimate proves
  \begin{equation}\label{eq:R-L2-stability}
    \begin{aligned}
      \frac{H(z)^2}{\#\TT_H[z]} \, \| R_H(v_H) - R_H(w_H) \|^2_{L^2(\Omega_H[z])} \,\lesssim\, \| \alpha^{1/2} \nabla (v_H - w_H) \|^2_{L^2(\Omega_H[z])}.
    \end{aligned}
  \end{equation}

\textbf{Step 3. (finite overlap)}
  The combination of the previous estimates~\eqref{eq:G-triangle-inequality}--\eqref{eq:R-L2-stability} together with
  a finite overlap of the patches $\Omega_H[z]$ due to the uniform shape regularity~\eqref{refinement:shape-regular} yields that,
  for all $\UU_H \subseteq \TT_H$ and all $v_H, w_H \in \XX_H$,
  \begin{equation*}
    \begin{aligned}
     \big| \mu(\UU_H, v_H) - \mu(\UU_H, w_H) \big|
     &\lesssim
     \Big( \sum_{T \in \UU_H} \sum_{z \in \VV_H^\Omega \cap T} \| \alpha^{1/2} \nabla (v_H - w_H) \|^2_{L^2(\Omega_H[z])} \Big)^{1/2}
     \\&
     \eqreff*{refinement:shape-regular}\lesssim
     \, 
     \| \alpha^{1/2} \nabla (v_H - w_H) \|_{L^2(\Omega)}
     \eqreff{eq:zz:alpha}\simeq
    \enorm{v_H - w_H}.
    \end{aligned}
  \end{equation*}
  Hence, we conclude weak stability~\eqref{axiom:weak_stability} of the ZZ-estimator with $r = 0$.
  Overall, the constant $\CCstab$ depends only on the polynomial degree $p$, the dimension \(d\),
  the quotient $\alpha_{\max}/\alpha_{\min}$,
  and uniform shape regularity~\eqref{refinement:shape-regular} of $\TT_H \in \T$.
  \qedhere
\end{proof}

\section{Error estimators based on equilibrated fluxes}
\label{section:equiflux}

In this section, we consider error estimators based on local flux equilibration; see, e.g., \cite{bps09, ev15} and the references therein.
They do not only provide guaranteed upper bounds for the Galerkin error  but also $p$-robust lower bounds \cite{bps09, ev20}.
For the Poisson model problem and exact Galerkin solutions, local equivalence to the residual-based error estimator has already been considered in \cite{ks2011,
cn2012} for lowest-order FEM ($p= 1$) and in \cite{bb20,cdgv26} for higher-order FEM ($p\ge1$).

\subsection{Definition of equilibrated-flux estimator and corresponding main result}
\label{section:equiflux:definition}

Let $\TT_H \in \T$ and $v_H \in \XX_H$. Suppose that a given flux $\flux \in \boldsymbol{H}(\div;\Omega)$ satisfies that
\begin{equation} \label{eq:divergence property of flux:continuous level}
\div \flux = f - c v_H.
\end{equation}
Then, together with the weak formulation \eqref{eq:weak_formulation} and the identity \eqref{eq:divergence property of flux:continuous level}, integration by parts proves the following upper bound for the error:
\begin{align} \label{eq:upper bound for error}
\begin{split}
\enorm{u^\star - v_H}^2 &= a(u^\star - v_H, u^\star - v_H) \\
& \eqreff*{eq:weak_formulation}= \dual{\boldsymbol{f} - \boldsymbol{A} \nabla v_H}{\nabla (u^\star - v_H)}_\Omega + \dual{f - c v_H}{u^\star-v_H}_\Omega \\
& \eqreff*{eq:divergence property of flux:continuous level}= \dual{\boldsymbol{f} - \boldsymbol{A} \nabla v_H - \flux}{\nabla (u^\star-v_H)}_\Omega \\
& \leq \| \boldsymbol{A}^{-1/2} (\flux + \boldsymbol{A} \nabla v_H - \boldsymbol{f}) \|_{L^2(\Omega)} \norm{\boldsymbol{A}^{1/2}\nabla(u^\star - v_H)}{L^2(\Omega)}.
\end{split}
\end{align}
Note that \eqref{eq:continuity_ellipticity} yields the existence of an ellipticity constant $C_{\rm ell}'>0$ such that
\begin{equation}\label{eq:ellipticity2}
\norm{\boldsymbol{A}^{1/2} \nabla v}{L^2(\Omega)} \le C_{\rm ell}' \enorm{v} \quad \text{for all }v\in H_0^1(\Omega).
\end{equation}
This proves reliability
\begin{equation}
	\enorm{u^\star - v_H} \le C_{\rm ell}' \norm{\boldsymbol{A}^{-1/2}(\flux + \boldsymbol{A} \nabla v_H - \boldsymbol{f})}{L^2(\Omega)},
	\quad \text{where $C_{\rm ell}'=1$ if $c\ge0$.}
\end{equation}

For all vertices $z \in \VV_H$, we introduce the space
\begin{equation*}
  \boldsymbol{H}_0(\div;\Omega_H[z]) \coloneqq
  \begin{cases}
    \{\boldsymbol{\tau} \in \boldsymbol{H}(\div;\Omega_H[z]) : \boldsymbol{\tau} \cdot \boldsymbol{n} = 0 \text{ on } \partial \Omega_H[z]\} & \text{if } z
    \in \VV_H \cap \Omega,\\
    \{\boldsymbol{\tau} \in \boldsymbol{H}(\div;\Omega_H[z]) : \boldsymbol{\tau} \cdot \boldsymbol{n} = 0 \text{ on } \partial \Omega_H[z] \setminus
    \partial \Omega\} & \text{if } z \in \VV_H \cap \partial \Omega,
  \end{cases}
\end{equation*}
with corresponding divergence given as
\[ \div\big(\boldsymbol{H}_0(\div;\Omega_H[z])\big) = L^2_*(\Omega_H[z]) \coloneqq
\begin{cases}
    \set{v \in L^2(\Omega_H[z]) : \dual{v}{1}_{\Omega_H[z]}=0} &\hspace{-4mm} \text{if } z \in \VV_H \cap \Omega, \\
    L^2(\Omega_H[z]) &\hspace{-4mm} \text{if } z \in \VV_H \cap \partial \Omega.
\end{cases} \]We note that any $\boldsymbol{\tau} \in \boldsymbol{H}_0(\div;\Omega_H[z])$ can be extended by zero to $\boldsymbol{\tau} \in \boldsymbol{H}(\div;\Omega)$.
For a given polynomial degree $q \geq p$, we 
define the local Raviart--Thomas subspace
\[ \RT_0^q(\TT_H[z]) \coloneqq\boldsymbol{H}_0(\div;\Omega_H[z]) \cap \big([\PP^q(\TT_H[z])]^d + \boldsymbol{x} \PP^q(\TT_H[z]) \big), \]
and the piecewise polynomial subspace
\[ \PP^q_*(\TT_H[z]) \coloneqq L^2_*(\Omega_H[z]) \cap \PP^q(\TT_H[z]). \]
Recall the hat functions $\varphi_{H,z}\in\SS^1(\TT_H)$ from Section~\ref{section:fem}.
Denote by $\Pi_H^{q}\colon L^2(\Omega) \to \PP^q(\TT_H)$ the $L^2(\Omega)$-orthogonal projection onto $\PP^q(\TT_H)$ and, for all $z \in \VV_H$,
\begin{equation} \label{eq:definition of g}
g[v_H] \coloneqq f - c v_H \quad \text{and} \quad g_{H,z}[v_H] \coloneqq\varphi_{H,z} g[v_H] - \nabla \varphi_{H,z} \cdot (\boldsymbol{A} \nabla v_H - \boldsymbol{f}).
\end{equation}

Following, e.g.,~\cite{ev15}, we define for $v_H = u_H^\star$ the local equilibrated-flux reconstruction
\begin{equation} \label{eq:local flux}
\flux_{H,z}[u^\star_H] \coloneqq\argmin_{\substack{\boldsymbol{\tau}_H \in \RT_0^q(\TT_H[z])\\ \div \boldsymbol{\tau}_H = \Pi_H^{q} g_{H,z}[u^\star_H] }}
\| \boldsymbol{A}^{-1/2}(\boldsymbol{\tau}_H + \varphi_{H,z} \, (\boldsymbol{A} \nabla u^\star_H - \boldsymbol{f}))\|_{L^2(\Omega_H[z])}.
\end{equation}
Extending $\flux_{H,z}[u^\star_H]$ by zero outside $\Omega_H[z]$, we further define the global equilibrated flux
\begin{equation} \label{eq:global flux}
\flux_H[u^\star_H] \coloneqq\sum_{z \in \VV_H} \flux_{H,z}[u^\star_H] \in \boldsymbol{H}(\div;\Omega).
\end{equation}
For $z \in \VV_H^\Omega$, we have that $\varphi_{H,z} \in \XX_H$ and thus
\[ \dual{\Pi_H^{q} g_{H,z}[u^\star_H]}{1}_{\Omega_H[z]} \eqreff*{eq:definition of g}= \dual{f - c u_H^\star}{\varphi_{H,z}}_{\Omega} + \dual{\boldsymbol{f} - \boldsymbol{A} \nabla u_H^\star}{\nabla \varphi_{H,z}}_\Omega \eqreff{eq:discrete_problem}= 0. \]
In particular, the constraint set in \eqref{eq:local flux} is non-empty, closed, and convex, so that the minimization problem~\eqref{eq:local flux} is well-posed.

For general $v_H \in \XX_H$, it is no longer true that $\dual{\Pi_H^{q} g_{H,z}[v_H]}{1}_{\Omega_H[z]}=0$ for all $z \in \VV_H^\Omega$.
To adapt the side constraint of \eqref{eq:local flux} in this case,
we introduce, for all $z \in \VV_H$, a correction term $\gamma_{H,z}[v_H] \in L^2(\Omega_H[z])$ satisfying the following properties
\begin{subequations} \label{eq:correction term properties} \begin{align} \label{eq:correction term property:integral mean}
\dual{g_{H,z}[v_H] - \gamma_{H,z}[v_H]}{1}_{\Omega_H[z]} &=0,\\
\label{eq:correction term property:exact solution}
\gamma_{H,z}[u_H^\star] &= 0,\\
\label{eq:correction term property:stability}
\diam(\Omega_H[z]) \| \gamma_{H,z}[v_H] - \gamma_{H,z}[w_H] \|_{L^2(\Omega_H[z])}
&\leq C_{\textup{cor}} \| v_H - w_H \|_{H^1(\Omega_H^{n}[z])}
\end{align} \end{subequations}
for some uniform $C_{\textup{cor}} >0$ and $n \in \N$, all $z \in \VV_H$, and all $v_H, w_H \in \XX_H$. Here, \eqref{eq:correction term property:integral mean} is needed for the well-posedness of the minimization problem \eqref{eq:local flux:inex sol}.
Moreover, \eqref{eq:correction term property:exact solution} ensures that definitions \eqref{eq:local flux} and \eqref{eq:local flux:inex sol} yield the same equilibrated-flux reconstruction for the exact solution $u_H^\star$.
Stability of the correction term \eqref{eq:correction term property:stability} is needed for the proof of weak stability~\eqref{axiom:weak_stability} in Proposition~\ref{prop:weak-stability-equilibrated-flux}.

\begin{remark}\label{rem:correction}
An example of such a correction term is the integral mean, i.e., $\gamma_{H,z}[v_H] \coloneqq{1 \over |\Omega_H[z]|} \dual{g_{H,z}[v_H]}{1}_{\Omega_H[z]}$, where \eqref{eq:correction term properties} is satisfied, since \eqref{eq:correction term property:stability} follows from \eqref{eq:proof of weak stability-equilibrated-flux:step2.2} below with $n=1$ and $C_{\textup{cor}}$ depending only on $\| \boldsymbol{A} \|_{L^\infty(\Omega)}$, $\| c \|_{L^\infty(\Omega)}$, the diameter of $\Omega$, the space dimension~$d$, and the uniform shape regularity~\eqref{refinement:shape-regular}.
In Section~\ref{section:equiflux:correction term}, we provide yet another example from \cite{psv18, prvw2020}, which allows to define a guaranteed upper bound for the algebraic and the total error.
\end{remark}

With the correction term~\eqref{eq:correction term properties}, we define the local equilibrated-flux reconstruction
\begin{equation} \label{eq:local flux:inex sol}
\flux_{H,z}[v_H] \coloneqq\argmin_{\substack{\boldsymbol{\tau}_H \in \RT_0^{q}(\TT_H[z])\\ \div \boldsymbol{\tau}_H = \Pi_H^{q} (g_{H,z}[v_H] - \gamma_{H,z}[v_H])}}
\| \boldsymbol{A}^{-1/2}( \boldsymbol{\tau}_H + \varphi_{H,z} \, (\boldsymbol{A} \nabla v_H - \boldsymbol{f}))\|_{L^2(\Omega_H[z])}.
\end{equation}
Extending $\flux_{H,z}[v_H]$ by zero outside $\Omega_H[z]$, we further define the global equilibrated flux
\begin{equation} \label{eq:global flux:inex sol}
\flux_H[v_H] \coloneqq\sum_{z \in \VV_H} \flux_{H,z}[v_H] \in \boldsymbol{H}(\div;\Omega) .
\end{equation}
The fact that $\set{\varphi_{H,z} : z \in \VV_H}$ forms a partition of unity on $\Omega$ immediately gives that
\begin{equation} \label{eq:divergence property of flux}
\div \flux_H[v_H] \eqreff{eq:definition of g}= \Pi_H^{q} \Big( g[v_H] - \sum_{z \in \VV_H} \gamma_{H,z}[v_H] \Big) \quad \text{for all} \ v_H \in \XX_H .
\end{equation}
Standard arguments, see e.g., \cite{ev15}, show that the (unique) minimizer to \eqref{eq:local flux:inex sol} can be computed as the first component of the following mixed formulation:
Find $(\flux_{H,z}[v_H], \rho_{H,z}[v_H]) \in \RT_0^q(\TT_H[z]) \times \PP^q_*(\TT_H[z])$ such that
\begin{align} \label{eq:local flux mixed formulation}
\hspace{-1mm}\dual{\boldsymbol{A}^{-1}\flux_{H,z}[v_H]}{\boldsymbol{\tau}_H}_{\Omega_H[z]} + \dual{\div \boldsymbol{\tau}_H}{\rho_{H,z}[v_H]}_{\Omega_H[z]} & = -\dual{\boldsymbol{A}^{-1}\boldsymbol{\tau}_H}{\varphi_{H,z}(\boldsymbol{A} \nabla v_H - \boldsymbol{f})}_{\Omega_H[z]} \notag
\\
\dual{\div \flux_{H,z}[v_H]}{w_H}_{\Omega_H[z]} & = \dual{g_{H,z}[v_H] - \gamma_{H,z}[v_H]}{w_H}_{\Omega_H[z]}
\end{align}
for all $\boldsymbol{\tau}_H \in \RT^q_0(\TT_H[z])$ and $w_H \in \PP^q(\TT_H[z])$.

For all $T \in \TT_H$ and 
$v_H \in \XX_H$, with (a lower bound of) the minimal eigenvalue $\lambda_{\rm min}(\boldsymbol{A}|_T)$ of $\boldsymbol{A}|_T$, the corresponding local contribution of the equilibrated-flux estimator 
reads
\begin{equation} \label{eq:equilibrated-flux-estimator}
\mu_H(T;v_H) \coloneqq\| \boldsymbol{A}^{-1/2}(\flux_H[v_H] + \boldsymbol{A} \nabla v_H - \boldsymbol{f} ) \|_{L^2(T)} + \frac{\diam(T)}{\pi\lambda_{\rm min}(\boldsymbol{A}|_T)^{1/2}} \|(1-\Pi_H^{q})g[v_H]\|_{L^2(T)}.
\end{equation}
With $C_{\rm ell}'>0$ from \eqref{eq:ellipticity2}, we particularly highlight the guaranteed reliability estimate
\begin{equation} \label{eq:reliability estimate for the equilibrated-flux estimator}
\enorm{u^\star - u^\star_H} \leq C_{\rm ell}'\mu_H(u^\star_H),
\quad \text{where $C_{\rm ell}'=1$ if $c\ge0$,}
\end{equation}
which follows from \eqref{eq:upper bound for error}, using the identity \eqref{eq:divergence property of flux} and the Poincaré inequality on convex domains.
A much deeper result is that $\mu_H(u_H^\star)$ is also locally efficient with $p$-robust constant; see \cite{bps09,ev20}.
The following theorem is the main result of this section.

\begin{theorem}[convergence of Algorithm~\ref{algorithm:afem} driven by equilibrated-flux estimator]
\label{thm:equilibrated-flux:convergence and optimality}
Let 
$p,q \in \N$ 
with $q \geq p$. Suppose that $\boldsymbol{A} \in [\PP^{q-p}(\TT_0)]^{d \times d}$ and $\boldsymbol{f}\in [\PP^{q-1}(\TT_0)]^d+\boldsymbol{x}\PP^{q-1}(\TT_0)$. For the model problem \eqref{eq:weak_formulation}, we consider Algorithm~\ref{algorithm:afem} steered by the equilibrated-flux estimator $\mu_H$ defined in \eqref{eq:equilibrated-flux-estimator}. Then, $\mu_H$ satisfies the assumptions \eqref{eq:equivalence}--\eqref{axiom:weak_stability} of Section~\ref{section:nonresidual-estimator} (with $m=1$ and $r=0$) so that the two main results of Section~\ref{section:convergence} apply: Theorem~\ref{theorem:full_linear_convergence} guarantees full R-linear convergence of the quasi-error \eqref{eq2:single:quasi-error} for any choice of the adaptivity parameters $0<\theta \leq 1$, $\Cmark \geq 1$, and $\lambda > 0$. Theorem~\ref{theorem:optimal-complexity} guarantees optimal complexity of Algorithm~\ref{algorithm:afem} provided that $\theta$ and $\lambda$ are chosen sufficiently small.
\end{theorem}

The proof of Theorem~\ref{thm:equilibrated-flux:convergence and optimality} follows as soon as we verify local equivalence \eqref{eq:equivalence} and  weak stability \eqref{axiom:weak_stability} of the equilibrated-flux estimator \eqref{eq:equilibrated-flux-estimator}. This is done in 
Sections~\ref{section:equiflux:equivalence}--\ref{section:equiflux:weak-stability}.

\subsection{Proof of local equivalence~(\ref{eq:equivalence}) for equilibrated-flux estimator}
\label{section:equiflux:equivalence}

We prove that the equilibrated-flux estimator is locally equivalent to the residual-based estimator $\eta_H$ in the sense of \eqref{eq:equivalence}.

\begin{proposition}[local equivalence of equilibrated-flux estimator] \label{prop:local equivalence equilibrated-flux}
Let the polynomial degrees $p,q \in \N$ be arbitrary such that $q \geq p$. Let $\boldsymbol{A} \in [\PP^{q-p}(\TT_0)]^{d \times d}$ and $\boldsymbol{f} \in [\PP^{q-1}(\TT_0)]^d+\boldsymbol{x}\PP^{q-1}(\TT_0)$. Then, the equilibrated-flux estimator $\mu_H$ from \eqref{eq:equilibrated-flux-estimator} is equivalent to the residual-based estimator $\eta_H$ from~\eqref{eq:residual_based_estimator} in the sense of~\eqref{eq:equivalence} with $m=1$. The equivalence constant $\Cwseq$ depends only on the space dimension~$d$, the minimal and maximal eigenvalues of $\boldsymbol{A}$, uniform shape regularity~\eqref{refinement:shape-regular}, and the polynomial degree~$q$.
\end{proposition}
\begin{proof}
Let $\TT_H \in \T$ and $T \in \TT_H$.
We split the proof into three steps. In Step 1, we prove that $\eta_H(T;u_H^\star) \lesssim \mu_H(\TT_H[T];u_H^\star)$. In Step~2--3 we prove the converse bound.

\textbf{Step 1 ($\boldsymbol{\eta_H \lesssim \mu_H}$).}
A standard inverse estimate provides the following bound for the volume term of the residual-based estimator
\begin{equation} \label{eq:proof of local equivalence for equilibrated-flux:step 1.1} \begin{aligned}
& |T|^{1/d} \| \div (\boldsymbol{A} \nabla u_H^\star - \boldsymbol{f}) - cu_H^\star + f \|_{L^2(T)}
\eqreff*{eq:definition of g}= |T|^{1/d} \| \div (\boldsymbol{A} \nabla u_H^\star - \boldsymbol{f}) + g[u_H^\star] \|_{L^2(T)} \\
& \quad \eqreff*{eq:divergence property of flux}= |T|^{1/d} \| \div \left( \flux_H[u^\star_H] + \boldsymbol{A} \nabla u_H^\star - \boldsymbol{f} \right) + (1 - \Pi_H^{q})g[u_H^\star] \|_{L^2(T)} \\
& \quad \lesssim \| \boldsymbol{A}^{-1/2}(\flux_H[u^\star_H] + \boldsymbol{A} \nabla u_H^\star - \boldsymbol{f}) \|_{L^2(T)} + {\diam(T) \over {\pi\lambda_{\rm min}(\boldsymbol{A}|_T)^{1/2}}} \| (1-\Pi_H^{q}) g[u^\star_H] \|_{L^2(T)},
\end{aligned} \end{equation}
where the hidden constant depends only on the space dimension~$d$, the minimal and maximal eigenvalues of $\boldsymbol{A}$, the polynomial degree~$q$, and uniform shape regularity~\eqref{refinement:shape-regular}.

Since $\flux_H[u_H^\star] \in \boldsymbol{H}(\div;\Omega)$, its normal jumps across internal edges vanish. Hence, the trace inequality and an inverse estimate show 
that
\begin{align}\label{eq:proof of local equivalence for equilibrated-flux:step 1.2}
|T|^{1/(2d)} \| \lbrack \! \lbrack (\boldsymbol{A} \nabla u_H^\star - \boldsymbol{f}) \cdot \boldsymbol{n} \rbrack \! \rbrack \|_{L^2(\partial T \cap \Omega)}
&= |T|^{1/(2d)} \| \lbrack \! \lbrack (\flux_H[u_H^\star] + \boldsymbol{A} \nabla u_H^\star - \boldsymbol{f}) \cdot \boldsymbol{n} \rbrack \! \rbrack \|_{L^2(\partial T \cap \Omega)} \notag\\
& \lesssim \| \flux_H[u^\star_H] + \boldsymbol{A} \nabla u_H^\star - \boldsymbol{f} \|_{L^2(\Omega_H[T])},
\end{align}
where the hidden constant depends on the same quantities as the constant in \eqref{eq:proof of local equivalence for equilibrated-flux:step 1.1}.
Combining \eqref{eq:proof of local equivalence for equilibrated-flux:step 1.1} and \eqref{eq:proof of local equivalence for equilibrated-flux:step 1.2}, we obtain the estimate
\begin{align}\label{eq260715a}
 \eta_H(T;u_H^\star) \lesssim \mu_H(\TT_H[T];u_H^\star)
 \quad \text{for all } T \in \TT_H.
\end{align}

\textbf{Step 2 ($\boldsymbol{\mu_H \lesssim \eta_H}$, flux term).}
Consider the continuous counterpart of Problem \eqref{eq:local flux mixed formulation}: Find $(\flux_z[u_H^\star], \rho_z[u_H^\star]) \in \boldsymbol{H}_0(\div; \Omega_H[z]) \times L^2_*(\Omega_H[z])$ such that
\begin{equation} \label{eq:proof of local equivalence for equilibrated-flux:step 2.1} \begin{aligned}
\dual{\boldsymbol{A}^{-1}\flux_z[u_H^\star]}{\boldsymbol{\tau}}_{\Omega_H[z]} + \dual{\div \boldsymbol{\tau}}{\rho_z[u_H^\star]}_{\Omega_H[z]} &
= -\dual{\boldsymbol{A}^{-1}\boldsymbol{\tau}}{\varphi_{H,z}(\boldsymbol{A} \nabla u_H^\star - \boldsymbol{f})}_{\Omega_H[z]} \\
\dual{\div \flux_z[u_H^\star]}{w}_{\Omega_H[z]} & = \dual{\Pi_H^{q} g_{H,z}[u_H^\star]}{w}_{\Omega_H[z]}
\end{aligned} \end{equation}
for all $\boldsymbol{\tau} \in \boldsymbol{H}_0(\div; \Omega_H[z])$ and all $w \in L^2(\Omega_H[z])$. The first equation of \eqref{eq:proof of local equivalence for equilibrated-flux:step 2.1} implies that $\rho_z[u_H^\star]$ has a weak gradient given by
\begin{equation} \label{eq:proof of local equivalence for equilibrated-flux:step 2.2}
\nabla \rho_z[u_H^\star] = \boldsymbol{A}^{-1}(\flux_z[u_H^\star] + \varphi_{H,z}(\boldsymbol{A} \nabla u_H^\star - \boldsymbol{f})) \in \big[L^2(\Omega_H[z])\big]^d ,
\end{equation}
which implies that
\[ \rho_z[u_H^\star] \in H^1_*(\Omega_H[z]) \coloneqq\begin{cases}
H^1(\Omega_H[z]) \cap L^2_*(\Omega_H[z]) \quad & \text{if} \ z \in \VV_H \cap \Omega, \\
\set{v \in H^1(\Omega_H[z]) : v|_{\partial \Omega_H[z] \cap \partial \Omega} = 0 } \quad & \text{if} \ z \in \VV_H \cap \partial\Omega.
\end{cases} \]
The definition of the local flux \eqref{eq:local flux}, the $q$-robust polynomial stability result from~\cite{bps09, ev20} for $d=2,3$, 
and the identity \eqref{eq:proof of local equivalence for equilibrated-flux:step 2.2} imply that
\begin{align} \label{eq:proof of local equivalence for equilibrated-flux:step 2.3}
&\| \boldsymbol{A}^{-1/2}(\flux_{H,z}[u_H^\star] + \varphi_{H,z}( \boldsymbol{A} \nabla u_H^\star - \boldsymbol{f})) \|_{L^2(\Omega_H[z])} \notag \\
&\quad = \min_{\substack{\boldsymbol{\tau}_H \in \RT_0^q(\TT_H[z]) \\ \div \boldsymbol{\tau}_H = \Pi_H^{q} g_{H,z}[u_H^\star]}} \| \boldsymbol{A}^{-1/2}(\boldsymbol{\tau}_H + \varphi_{H,z}(\boldsymbol{A} \nabla u_H^\star - \boldsymbol{f})) \|_{L^2(\Omega_H[z])}  \\
& \quad \lesssim \min_{\substack{\boldsymbol{\tau} \in \boldsymbol{H}_0(\div; \Omega_H[z]) \\ \div \boldsymbol{\tau} = \Pi_H^{q} g_{H,z}[u_H^\star]}}
\| \boldsymbol{A}^{-1/2} (\boldsymbol{\tau} + \varphi_{H,z}(\boldsymbol{A} \nabla u_H^\star - \boldsymbol{f})) \|_{L^2(\Omega_H[z])}
= \| \boldsymbol{A}^{1/2} \nabla \rho_z[u_H^\star] \|_{L^2(\Omega_H[z])}. \notag
\end{align}
The employed argument hinges on the inclusion $\varphi_{H,z}(\boldsymbol{A} \nabla u_H^\star - \boldsymbol{f})\in [\PP^q(\TT_H)]^d + \boldsymbol{x} \PP^q(\TT_H)$, which follows from the assumptions on $\boldsymbol{A}$ and $\boldsymbol{f}$.

The identity \eqref{eq:proof of local equivalence for equilibrated-flux:step 2.2}, integration by parts, and problem \eqref{eq:proof of local equivalence for equilibrated-flux:step 2.1} show, for all $w \in H^1_*(\Omega_H[z])$, 
\begin{align*}
\dual{\boldsymbol{A}\nabla \rho_z[u_H^\star]}{\nabla w}_{\Omega_H[z]} & \eqreff*{eq:proof of local equivalence for equilibrated-flux:step 2.2}= \dual{\flux_z[u_H^\star] + \varphi_{H,z}(\boldsymbol{A} \nabla u_H^\star - \boldsymbol{f})}{\nabla w}_{\Omega_H[z]} \\
& = - \dual{\div \flux_z[u_H^\star]}{w}_{\Omega_H[z]} + \dual{\varphi_{H,z}(\boldsymbol{A} \nabla u_H^\star - \boldsymbol{f})}{\nabla w}_{\Omega_H[z]} \\
& \eqreff*{eq:proof of local equivalence for equilibrated-flux:step 2.1}= - \dual{\Pi_H^{q} g_{H,z}[u_H^\star]}{w}_{\Omega_H[z]} + \dual{\varphi_{H,z}(\boldsymbol{A} \nabla u_H^\star - \boldsymbol{f})}{\nabla w}_{\Omega_H[z]}.
\end{align*}
Since $w \in H^1_*(\Omega_H[z])$, we know that $\varphi_{H,z} w$ extended by zero outside $\Omega_H[z]$ is a suitable test function in $H^1_0(\Omega)$ for problem \eqref{eq:weak_formulation}, and hence we have that
\begin{align*}
& \dual{g_{H,z}[u_H^\star]}{w}_{\Omega_H[z]} \eqreff{eq:definition of g}= \dual{g[u_H^\star]}{\varphi_{H,z}w}_{\Omega_H[z]} - \dual{\boldsymbol{A} \nabla u_H^\star - \boldsymbol{f}}{w \nabla \varphi_{H,z}}_{\Omega_H[z]} \\
& \eqreff*{eq:definition of g}= \dual{f - c u^\star}{\varphi_{H,z}w}_{\Omega_H[z]} - \dual{\boldsymbol{A} \nabla u_H^\star - \boldsymbol{f}}{w \nabla \varphi_{H,z}}_{\Omega_H[z]}
\\ & \qquad + \dual{c (u^\star - u_H^\star)}{\varphi_{H,z}w}_{\Omega_H[z]} \\
& \eqreff*{eq:weak_formulation}= \dual{\boldsymbol{A} \nabla u^\star - \boldsymbol{f}}{\nabla (\varphi_{H,z} w)}_{\Omega_H[z]} - \dual{\boldsymbol{A} \nabla u_H^\star - \boldsymbol{f}}{w \nabla \varphi_{H,z}}_{\Omega_H[z]}
\\ & \qquad + \dual{c (u^\star - u_H^\star)}{\varphi_{H,z}w}_{\Omega_H[z]} \\
& \eqreff*{eq:definition of a}= \dual{\varphi_{H,z} (\boldsymbol{A} \nabla u_H^\star - \boldsymbol{f})}{\nabla w}_{\Omega_H[z]} + a(u^\star - u_H^\star, \varphi_{H,z}w).
\end{align*}
Combining the two latter identities, and using again definition \eqref{eq:definition of g},
we arrive at
\begin{equation} \label{eq:proof of local equivalence for equilibrated-flux:step 2.4}
\dual{\boldsymbol{A}\nabla \rho_z[u_H^\star]}{\nabla w}_{\Omega_H[z]} = -a(u^\star - u_H^\star, \varphi_{H,z}w) + \dual{(1-\Pi_H^{q})(\varphi_{H,z} g[u_H^\star])}{w}_{\Omega_H[z]}.
\end{equation}
The standard reliability proof for the residual-based estimator shows that
\begin{equation} \label{eq:proof of local equivalence for equilibrated-flux:step 2.5}
|a(u^\star - u_H^\star, \varphi_{H,z} w)| \lesssim \eta_H(\TT_H[z];u_H^\star) \| \nabla(\varphi_{H,z} w) \|_{L^2(\Omega_H[z])} .
\end{equation}
The fact that $\| \varphi_{H,z} \|_{L^\infty(\Omega)} =1$ and $\|\nabla \varphi_{H,z} \|_{L^\infty(\Omega)}\lesssim\diam(T)^{-1}$ and the Poincaré--Friedrichs inequality applied to $w \in H^1_*(\Omega_H[z])$ give that
\begin{equation} \label{eq:proof of local equivalence for equilibrated-flux:step 2.6}
\| \nabla(\varphi_{H,z} w) \|_{L^2(\Omega_H[z])} \lesssim \| \nabla w \|_{L^2(\Omega_H[z])}.
\end{equation}
Note that the constant depends only on the space dimension $d$ and uniform shape regularity~\eqref{refinement:shape-regular}; see, e.g., \cite{cf00}. Choosing $w=\rho_z[u_H^\star]$ and using \eqref{eq:proof of local equivalence for equilibrated-flux:step 2.4}--\eqref{eq:proof of local equivalence for equilibrated-flux:step 2.6} as well as the Poincaré--Friedrichs inequality, we conclude that
\begin{equation} \label{eq:proof of local equivalence for equilibrated-flux:step 2.7}
\| \boldsymbol{A}^{1/2}\nabla \rho_z[u_H^\star] \|_{L^2(\Omega_H[z])} \lesssim \eta_H(\TT_H[z];u_H^\star) + \diam(T) \| (1-\Pi_H^{q})(\varphi_{H,z} g[u_H^\star]) \|_{L^2(\Omega_H[z])} .
\end{equation}

Finally, the definition of the global flux \eqref{eq:global flux}, the fact that $\set{\varphi_{H,z} : z \in \VV_H \cap T}$ forms a partition of unity on $T$, and \eqref{eq:proof of local equivalence for equilibrated-flux:step 2.3} and \eqref{eq:proof of local equivalence for equilibrated-flux:step 2.7} show that
\begin{align} \label{eq:proof of local equivalence for equilibrated-flux:step 2 final goal}
& \| \boldsymbol{A}^{-1/2}(\flux_H[u_H^\star] + \boldsymbol{A} \nabla u_H^\star - \boldsymbol{f}) \|_{L^2(T)}
\notag \\ & \quad
\leq \sum_{z \in \VV_H \cap T} \| \boldsymbol{A}^{-1/2}(\flux_{H,z}[u_H^\star] + \varphi_{H,z}(\boldsymbol{A} \nabla u_H^\star - \boldsymbol{f})) \|_{L^2(\Omega_H[z])} \notag \\
& \quad \lesssim \eta_H(\TT_H[T];u_H^\star)
\notag \\ & \qquad + \sum_{z \in \VV_H \cap T} \diam(T) \| (1-\Pi_H^{q})(\varphi_{H,z} g[u_H^\star]) \|_{L^2(\Omega_H[z])},
\end{align}
where the hidden constant depends only on the space dimension $d$, the minimal and maximal eigenvalues of $\boldsymbol{A}$, and uniform shape regularity~\eqref{refinement:shape-regular}.

\textbf{Step 3 ($\boldsymbol{\mu_H \lesssim \eta_H}$, oscillation term).}
The assumptions on $\boldsymbol{A}$ and $\boldsymbol{f}$ guarantee that $\div(\boldsymbol{A} \nabla u_H^\star - \boldsymbol{f})|_T \in \PP^{q-2}(\{T\})$.
Together with stability of $1-\Pi_H^{q}$, this proves that
\begin{align*}
{\diam(T) \over \pi} \| (1-\Pi_H^{q})g[u_H^\star] \|_{L^2(T)} & = {\diam(T) \over \pi} \| (1-\Pi_H^{q})(g[u_H^\star] + \div(\boldsymbol{A} \nabla u_H^\star - \boldsymbol{f})) \|_{L^2(T)} \\
& \lesssim |T|^{1/d} \| \div(\boldsymbol{A} \nabla u_H^\star - \boldsymbol{f}) + g[u_H^\star] \|_{L^2(T)} \eqreff{eq:residual_based_estimator}\leq \eta_H(T;u_H^\star).
\end{align*}
The assumptions on $\boldsymbol{A}$ and $\boldsymbol{f}$ guarantee that $\varphi_{H,z}|_T\div(\boldsymbol{A} \nabla u_H^\star - \boldsymbol{f})|_T \in \PP^{q-1}(\{T\})$.
The same argument thus also shows for all $T' \in \TT_H[T]$  that
\begin{align*}
&{\diam(T) \over \pi} \| (1-\Pi_H^{q})(\varphi_{H,z} g[u_H^\star]) \|_{L^2(T')}
\\
&\quad= {\diam(T) \over \pi} \| (1-\Pi_H^{q})(\varphi_{H,z} (g[u_H^\star] + \div(\boldsymbol{A} \nabla u_H^\star - \boldsymbol{f}))) \|_{L^2(T')}
\lesssim \eta_H(T';u_H^\star).
\end{align*}
This bounds
the second term of the right-hand side of \eqref{eq:proof of local equivalence for equilibrated-flux:step 2 final goal}. With Step 2, we see that
\begin{align}\label{eq260715b}
 \mu_H(T;u_H^\star) \lesssim \eta_H(\TT_H[T];u_H^\star)
 \quad \text{for all } T \in \TT_H.
\end{align}

\textbf{Step~4 (from elementwise to setwise equivalence).}
Arguing as in Step~5 of the proof of Proposition~\ref{prop:ZZ-estimator}, the elementwise equivalence estimates~\eqref{eq260715a} and~\eqref{eq260715b} transfer to subsets $\UU_H \subseteq \TT_H$. Overall, this proves~\eqref{eq:equivalence} with $m=1$ and
concludes the proof.
\end{proof}

 \begin{remark}
The additional assumptions on $\boldsymbol{A}\in [\PP^{q-p}(\TT_0)]^{d\times d}$ and $\boldsymbol{f}\in [\PP^{q-1}(\TT_0)]^d+\boldsymbol{x}\PP^{q-1}(\TT_0)$ are necessary for Step 2--3 of the proof of Proposition~\ref{prop:local equivalence equilibrated-flux}. In Step 1, a weaker assumption is sufficient, namely, $\boldsymbol{A}$ and $\boldsymbol{f}$ just need to be piecewise polynomial on $\TT_0$, regardless of the polynomial degree. Under the additional assumptions $c \in \PP^{q-p}(\TT_0)$ and $f \in \PP^q(\TT_0)$, the proof of Proposition~\ref{prop:local equivalence equilibrated-flux} would just follow from a simple scaling argument; see, e.g., \cite[Proposition 5.16]{dadic25} for the detailed argument.
\end{remark}

\subsection{Proof of weak stability~(\ref{axiom:weak_stability}) for equilibrated-flux estimator}
\label{section:equiflux:weak-stability}

To conclude the proof of Theorem~\ref{thm:equilibrated-flux:convergence and optimality}, it only remains to establish the weak stability~\eqref{axiom:weak_stability} of the equilibrated-flux estimator from \eqref{eq:equilibrated-flux-estimator}. The proof relies on the following lemma and on the stability assumption \eqref{eq:correction term property:stability} of the correction term $\gamma_{H,z}[v_H]$.

\begin{lemma}[stability of local flux]
There exists a positive constant $C_{\rm{flux}}>0$, depending only on the space dimension $d$, the minimal and maximal eigenvalues of $\boldsymbol{A}$, uniform shape regularity~\eqref{refinement:shape-regular}, and the polynomial degree~$q$ such that
\begin{equation}\label{eq:local flux apriori estimate} \begin{aligned}
& \|\flux_{H,z}[v_H] - \flux_{H,z}[w_H] \|_{\boldsymbol{H}(\div;\Omega_H[z])} \leq C_{\rm{flux}} \big(\| \varphi_{H,z} \boldsymbol{A} \nabla( v_H - w_H) \|_{L^2(\Omega_H[z])} \\
& \quad \quad \quad + \diam(\Omega_H[z]) \, \|g_{H,z}[v_H] - g_{H,z}[w_H] - \gamma_{H,z}[v_H] + \gamma_{H,z}[w_H] \|_{L^2(\Omega_H[z])}\big)
\end{aligned} \end{equation}
for all $z \in \VV_H$ and all $v_H,w_H \in \XX_H$.
\end{lemma}
\begin{proof}
For all $\boldsymbol{\tau}_H \in \RT_0^q(\TT_H[z])$ and all $y_H \in \PP^q(\TT_H[z])$, the difference of the saddle-point problems \eqref{eq:local flux mixed formulation} for $\flux_{H,z}[v_H]$ and $\flux_{H,z}[w_H]$ reads as
\begin{equation} \label{eq:proof of stability of local flux:1} \begin{aligned}
& \dual{\boldsymbol{A}^{-1}(\flux_{H,z}[v_H] - \flux_{H,z}[w_H])}{\boldsymbol{\tau}_H}_{\Omega_H[z]} + \dual{\div \boldsymbol{\tau}_H}{\rho_{H,z}[v_H] - \rho_{H,z}[w_H]}_{\Omega_H[z]} \\
&\mspace{130mu}= - \dual{\boldsymbol{A}^{-1}\boldsymbol{\tau}_H}{\varphi_{H,z} (\boldsymbol{A} \nabla (v_H - w_H))}_{\Omega_H[z]} \\
& \dual{\div(\flux_{H,z}[v_H] - \flux_{H,z}[w_H])}{y_H}_{\Omega_H[z]} \\
&\mspace{130mu}= \dual{g_{H,z}[v_H] - g_{H,z}[w_H] - \gamma_{H,z}[v_H] + \gamma_{H,z}[w_H]}{y_H}_{\Omega_H[z]}.
\end{aligned} \end{equation}
It is well known that the saddle-point problem \eqref{eq:proof of stability of local flux:1} satisfies the assumptions of the Babu\v{s}ka--Brezzi theorem. In particular, there holds the inf-sup stability
\[ \inf_{y_H \in \PP^q_\star(\TT_H[z])} \sup_{\boldsymbol{\tau}_H \in \RT^q_0(\TT_H[z])} \frac{|\dual{\div \boldsymbol{\tau}_H}{y_H}_{\Omega_H[z]}|}{\| y_H \|_{L^2(\Omega_H[z])} \| \boldsymbol{\tau}_H \|_{\boldsymbol{H}(\div;\Omega_H[z])} } \geq C \diam(\Omega_H[z])^{-1} , \]
with a constant $C$ that depends only on the space dimension $d$, the uniform shape regularity~\eqref{refinement:shape-regular}, and the polynomial degree $q$; see \cite[Lemma~51.10 and Remark~51.15]{egII2021}.
Hence, \eqref{eq:local flux apriori estimate} follows from a standard stability estimate, e.g., \cite[Theorem 49.13]{egII2021}.
\end{proof}

\begin{proposition}[weak stability of equilibrated-flux estimator] \label{prop:weak-stability-equilibrated-flux}
The equilibrated-flux estimator defined in \eqref{eq:equilibrated-flux-estimator} satisfies weak stability \eqref{axiom:weak_stability} with $r = 0$. The constant $\CCstab$ depends only on the space dimension~$d$, the diameter of $\Omega$, the ellipticity constant $\Cell$, the minimal and maximal eigenvalues of $\boldsymbol{A}$, $\| c \|_{L^\infty(\Omega)}$, uniform shape regularity~\eqref{refinement:shape-regular}, the polynomial degree~$q$, and the constant $C_{\textup{cor}}$ and $n$ from \eqref{eq:correction term property:stability}.
\end{proposition}
\begin{proof}
Let $\TT_H \in \T$, $v_H, w_H \in \XX_H$, and $\UU_H \subseteq \TT_H$. The proof consists of two steps.

  \textbf{Step 1 (stability of local flux).}
The triangle inequality and \eqref{eq:global flux:inex sol} yield that
\begin{equation}\label{eq:proof of weak stability-equilibrated-flux:step2.0}
\|\boldsymbol{A}^{-1/2}(\flux_H[v_H] - \flux_H[w_H])\|_{L^2(T)} \eqreff{eq:global flux:inex sol}\lesssim \sum_{z \in \VV_H \cap T} \|\flux_{H,z}[v_H] - \flux_{H,z}[w_H]\|_{L^2(\Omega_H[z])}.
\end{equation}
Stability of the local flux \eqref{eq:local flux apriori estimate} and the identity $\| \varphi_{H,z} \|_{L^\infty(\Omega)}=1$ show that
\begin{equation}\label{eq:proof of weak stability-equilibrated-flux:step2.1} \begin{aligned}
& \| \flux_{H,z}[v_H] - \flux_{H,z}[w_H] \|_{\boldsymbol{H}(\div;\Omega_H[z])} \eqreff{eq:local flux apriori estimate}\lesssim \| \boldsymbol{A} \nabla (v_H - w_H) \|_{L^2(\Omega_H[z])} \\
&  + \diam(\Omega_H[z]) \Big(\| g_{H,z}[v_H] - g_{H,z}[w_H] \|_{L^2(\Omega_H[z])} + \| \gamma_{H,z}[v_H]-\gamma_{H,z}[w_H] \|_{L^2(\Omega_H[z])} \Big).
\end{aligned} \end{equation}
The definition of $g_{H,z}$ in \eqref{eq:definition of g}, the identity $\| \varphi_{H,z} \|_{L^\infty(\Omega)}=1$, and the inequality $\| \nabla \varphi_{H,z} \|_{L^\infty(\Omega)}$ $\lesssim$ $\diam(\Omega_H[z])^{-1}$ yield that
\begin{equation} \label{eq:proof of weak stability-equilibrated-flux:step2.2}
\| g_{H,z}[v_H] - g_{H,z}[w_H] \|_{L^2(\Omega_H[z])} \lesssim  (1 + \diam(\Omega_H[z])^{-1}) \| v_H - w_H \|_{H^1(\Omega_H[z])}.
\end{equation}
Therefore, the combination of the estimates \eqref{eq:proof of weak stability-equilibrated-flux:step2.0}--\eqref{eq:proof of weak stability-equilibrated-flux:step2.2} and \eqref{eq:correction term property:stability} shows that
\begin{equation}\label{eq:proof of weak stability-equilibrated-flux:step2.4} \begin{aligned}
\|\boldsymbol{A}^{-1/2}(\flux_H[v_H] - \flux_H[w_H])\|_{L^2(T)} \lesssim \sum_{z \in \VV_H \cap T} \| v_H - w_H \|_{H^1(\Omega_H^{n}[z])}.
\end{aligned} \end{equation}
where the hidden constant depends only on the space dimension~$d$, the diameter of $\Omega$, the ellipticity constant $\Cell$, the minimal and maximal eigenvalues of $\boldsymbol{A}$, $\| c \|_{L^\infty(\Omega)}$, uniform shape regularity~\eqref{refinement:shape-regular}, the polynomial degree~$q$, and the constant $C_{\textup{cor}}$.

\textbf{Step 2 (stability).}
The reverse triangle inequality shows that
\begin{equation}\label{eq:proof of weak stability-equilibrated-flux:step1.0}
\begin{aligned}
& |\mu_H(\UU_H; v_H) - \mu_H(\UU_H; w_H)|^2 \, \leq \, \sum_{T \in \UU_H} \big |\mu_H(T; v_H) - \mu_H(T; w_H)\big|^2 \\
& \quad \eqreff*{eq:equilibrated-flux-estimator}{\leq} \;\;  \sum_{T \in \UU_H} \Big(3 \, \|\boldsymbol{A}^{-1/2}(\flux_H[v_H] - \flux_H[w_H])\|_{L^2(T)}^2 + 3 \, \|\boldsymbol{A}^{1/2} \nabla (v_H - w_H) \|_{L^2(T)}^2 \\
&\quad \quad \quad + 3 \, \frac{\diam(T)^2}{\pi^2\lambda_{\rm min}(\boldsymbol{A}|_T)} \, \| (1-\Pi_H^{q}) (g[v_H] - g[w_H])\|_{L^2(T)}^2\Big).
\end{aligned}
\end{equation}
By definition~\eqref{eq:definition of g}, we can bound the third term of the sum via
\begin{equation*}
\| (1-\Pi_H^{q}) (g[v_H] - g[w_H])\|_{L^2(T)} \leq \|c\|_{L^\infty(\Omega)} \, \| v_H - w_H \|_{L^2(T)} \quad \text{for all } T \in \UU_H.
\end{equation*}
Together with Step 1, uniform shape regularity~\eqref{refinement:shape-regular} and 
ellipticity~\eqref{eq:continuity_ellipticity}
imply that
\[ |\mu_H(\UU_H; v_H) - \mu_H(\UU_H; w_H)| \lesssim \| v_H - w_H \|_{H^1(\Omega)} \lesssim \enorm{ v_H - w_H}. \]
This concludes the proof.
\end{proof}

\subsection{Flux equilibration for algebraic and total error}
\label{section:equiflux:correction term}

Following \cite{psv18, prvw2020,dv23}, we finally mention a different possible correction term, which also allows to compute a guaranteed upper bound for both the algebraic and the total error; cf.\ Remark~\ref{rem:correction}.

Recall the Lagrange nodes $\NN_H$ and the associated basis functions $\phi_{H,z}$ from Section~\ref{section:fem}.
As in \cite[Equation (6.7)]{prvw2020}, we define the piecewise polynomial algebraic residual lifting $r_H[v_H] \in \PP^p(\TT_H)$ with $r_H[v_H] = 0$ on $\partial\Omega$ locally by
\begin{equation} \label{eq:algebraic residual functional}
\dual{r_H[v_H]}{\phi_{H,z}}_T = \frac{|T|}{|\Omega_H[z]|} \left(\dual{f}{\phi_{H,z}}_{\Omega_H[z]} + \dual{\boldsymbol{f}}{\nabla\phi_{H,z}}_{\Omega_H[z]} - a(v_H,\phi_{H,z}) \right),
\end{equation}
for all $z \in \NN_H \cap \Omega \cap T$ and all $T \in \TT_H$.
We note that $r_H[v_H]$ is a functional representation of the algebraic residual in the sense that
\begin{equation} \label{eq:algebraic residual functional alternative definition}
\dual{r_H[v_H]}{w_H}_\Omega = \sum_{z \in \NN_H} w_H(z) \dual{r_H[v_H]}{\phi_{H,z}}_\Omega
\eqreff*{eq:discrete_problem}= a(u^\star_H - v_H,w_H) \text{ for all} \ w_H \in \XX_H.
\end{equation}
This also implies that $r_H[u^\star_H] = 0$. The correction term is now defined as
\begin{equation} \label{eq:correction term:algebraic residual}
\gamma_{H,z}[v_H] \coloneqq\varphi_{H,z} r_H[v_H] \in L^2(\Omega_H[z]) \quad \text{for all } z \in \VV_H.
\end{equation}

\begin{lemma}[assumptions for correction term] \label{lemma:Lipschitz continuity of the algebraic residual functional}
The correction term defined in~\eqref{eq:correction term:algebraic residual} satisfies the assumptions \eqref{eq:correction term properties} with $n=2$. The constant $C_{\textup{cor}}$ depends only on the boundedness constant $\Cbnd$ of $a(\cdot,\cdot)$ from \eqref{eq:continuity_ellipticity}, the space dimension~$d$, the polynomial degree~$p$, the diameter of $\Omega$, and uniform shape regularity~\eqref{refinement:shape-regular}.
\end{lemma}
\begin{proof}
Assumption \eqref{eq:correction term property:exact solution} 
follows by construction \eqref{eq:algebraic residual functional alternative definition}.
Assumption \eqref{eq:correction term property:integral mean} follows from
\begin{align*}
\dual{g_{H,z}[v_H] - \gamma_{H,z}[v_H]}{1}_{\Omega_H[z]}
\eqreff*{eq:definition of g}= \dual{f - c v_H - r_H[v_H]}{\varphi_{H,z}}_\Omega - \dual{\boldsymbol{A} \nabla v_H - \boldsymbol{f}}{\nabla \varphi_{H,z}}_\Omega \\
= F(\varphi_{H,z}) - a(v_H,\varphi_{H,z}) - \dual{r_H[v_H]}{\varphi_{H,z}}_\Omega \eqreff{eq:algebraic residual functional alternative definition}= F(\varphi_{H,z}) - a(u_H^\star,\varphi_{H,z}) \eqreff{eq:discrete_problem}= 0.
\end{align*}

For Assumption~\eqref{eq:correction term property:stability}, let $T\in\TT_H$.
Since $(r_H[v_H] - r_H[w_H])|_T \in \PP^p(T)$ with $(r_H[v_H] - r_H[w_H])|_T = \sum_{z \in \NN_H \cap T} \alpha_z \phi_{H,z}$, a standard scaling argument shows that
\begin{equation} \label{eq:proof of Lipschitz continuity of the algebraic residual functional:1}
\sum_{z \in \NN_H \cap T} |\alpha_z| \lesssim |T|^{-{1/2}} \| r_H[v_H] - r_H[w_H] \|_{L^2(T)} .
\end{equation}
Note that $\alpha_z = 0$ for all $z \in (\NN_H \cap T) \backslash \Omega$ by definition of $r_H[\cdot]$.
Let $\NN_H^T := \NN_H \cap \Omega \cap T$.
A straightforward computation shows that
\begin{align*}
& \| r_H[v_H] - r_H[w_H] \|_{L^2(T)}^2 = \dual{r_H[v_H] - r_H[w_H]}{\sum_{z \in \NN_H^T} \alpha_z \phi_{H,z}}_T \\
& \quad \eqreff*{eq:algebraic residual functional}= \sum_{z \in \NN_H^T} \alpha_z \frac{|T|}{|\Omega_H[z]|} a(w_H-v_H,\phi_{H,z}) \eqreff*{eq:continuity_ellipticity}\lesssim \sum_{z \in \NN_H^T} |\alpha_z| \| v_H - w_H \|_{H^1(\Omega_H[z])} \| \phi_{H,z} \|_{H^1(\Omega_H[z])} \\
& \quad \lesssim \sum_{z \in \NN_H^T} |\alpha_z| \| v_H - w_H \|_{H^1(\Omega_H[T])} \left(1 + \diam(\Omega_H[T])^{-1}\right) |\Omega_H[T]|^{1/2} \\
& \quad \eqreff*{eq:proof of Lipschitz continuity of the algebraic residual functional:1}\lesssim
\frac{|\Omega_H[T]|^{1/2}}{|T|^{1/2}} \left(1 + \diam(\Omega_H[T])^{-1}\right) \| r_H[v_H] - r_H[w_H] \|_{L^2(T)} \| v_H - w_H \|_{H^1(\Omega_H[T])}.
\end{align*}
Since $|\Omega_H[T]| \simeq |T|$ by uniform shape regularity~\eqref{refinement:shape-regular}, we obtain that
\[ \| r_H[v_H] - r_H[w_H] \|_{L^2(T)} \lesssim \left(1 + \diam(\Omega_H[T])^{-1}\right) \| v_H - w_H \|_{H^1(\Omega_H[T])}. \]
By definition~\eqref{eq:correction term:algebraic residual} of $\gamma_{H,z}[v_H]$ and $\| \varphi_{H,z} \|_{L^\infty(\Omega)}=1$, we conclude the proof of \eqref{eq:correction term property:stability}.
\end{proof}

The algebraic residual lifting $r_H[v_H]$ can also be used to build an algebraic-flux reconstruction $\flux_H^{\mathrm{alg}}[v_H] \in [\PP^p(\TT_H)]^d + \boldsymbol{x} \PP^p(\TT_H)$ satisfying
\begin{equation} \label{eq:divergence property of algebraic flux}
\div \flux_H^{\mathrm{alg}}[v_H] = r_H[v_H] ;
\end{equation}
see, e.g., \cite[Section 4.1]{prvw2020} for an efficient multilevel construction on local patches.
This readily yields the following upper bound for the algebraic and the total error; see also \cite[Theorem~7.1]{prvw2020} for the particular case of the Poisson model problem.

\begin{proposition}[guaranteed upper bounds for algebraic and total error]
For each $v_H\in\XX_H$, let $\flux_H^{\rm alg}[v_H] \in \boldsymbol{H}(\div;\Omega)$ be a computable algebraic-flux reconstruction satisfying \eqref{eq:divergence property of algebraic flux}.
Recall the ellipticity constant $C_{\rm ell}'$ from~\eqref{eq:ellipticity2} and $C_{\rm ell}' = 1$ if $c\ge0$.
Then,
\begin{equation}\label{eq:zeta2}
\zeta_H(v_H) \coloneqq\| \boldsymbol{A}^{-1/2}\flux_H^{\mathrm{alg}}[v_H] \|_{L^2(\Omega)}
\end{equation}
defines a reliable estimator for the algebraic error in the sense that
\begin{equation}\label{eq:guaranteed_algebraic}
\enorm{u_H^\star - v_H} \le C_{\rm ell}' \zeta_H(v_H).
\end{equation}
Moreover, there holds that
\begin{equation}\label{eq:guaranteed_total}
\enorm{u^\star - v_H} \leq C_{\rm ell}' \big(\mu_H(v_H) + \zeta_H(v_H)\big).
\end{equation}
\end{proposition}

\begin{proof}
Using \eqref{eq:algebraic residual functional alternative definition} and \eqref{eq:divergence property of algebraic flux}, we see that
\begin{align*}
&\enorm{u_H^\star - v_H}^2 = a(u_H^\star - v_H , u_H^\star - v_H) \eqreff{eq:algebraic residual functional alternative definition}= \dual{r_H[v_H]}{u_H^\star - v_H}_\Omega 
\eqreff*{eq:divergence property of algebraic flux}= \dual{\div \flux_H^{\mathrm{alg}}[v_H]}{u_H^\star - v_H}_\Omega 
\\&\quad
= - \dual{\flux_H^{\mathrm{alg}}[v_H]}{\nabla(u_H^\star - v_H)}_\Omega 
\leq \| \boldsymbol{A}^{-1/2}\flux_H^{\mathrm{alg}}[v_H] \|_{L^2(\Omega)} \norm{\boldsymbol{A}^{1/2}\nabla (u_H^\star - v_H)}{L^2(\Omega)}.
\end{align*}
Together with ellipticity~\eqref{eq:ellipticity2}, this shows~\eqref{eq:guaranteed_algebraic}.

Noting that \eqref{eq:divergence property of flux} now reads as $\div \flux_H[v_H] = \Pi_H^{q} g[v_H] - r_H[v_H]$ (because of $q \ge p$), the combination with \eqref{eq:divergence property of algebraic flux} shows that
\begin{equation} \label{eq:divergence property of total flux}
\div (\flux_H[v_H] + \flux_H^{\mathrm{alg}}[v_H]) = \Pi_H^{q} g[v_H] \quad \text{for all} \ v_H \in \XX_H.
\end{equation}
As in \eqref{eq:upper bound for error} and \eqref{eq:reliability estimate for the equilibrated-flux estimator}, one can now derive the following upper bound for the total error
\begin{align*}
\enorm{u^\star - v_H}^2 
\leq \big(\mu_H(v_H) + \zeta_H(v_H)\big) \norm{\boldsymbol{A}^{1/2}\nabla(u^\star - v_H)}{L^2(\Omega)}.
\end{align*}
Together with ellipticity~\eqref{eq:ellipticity2}, this shows~\eqref{eq:guaranteed_total}.
\end{proof}

\begin{remark}
According to \cite[Theorem 7.4]{prvw2020}, the algebraic error estimator $\zeta_H(v_H)$ from \eqref{eq:zeta2} is also efficient~\eqref{eq:algebra:aposteriori}. However, at least theoretically, the corresponding constant $\ceff$ depends on the level of the considered mesh, so that it is unclear if using this algebraic estimator in Algorithm~\ref{algorithm:afem} would still guarantee R-linear convergence at optimal rate.
\end{remark}

\section{Numerical experiments}
\label{section:numerics}

This section validates the presented theory by appropriate 2D experiments.
To this end, we extended the Matlab 
package \texttt{MooAFEM} \cite{MooAFEM} by an interface for patchwise saddle-point problems like~\eqref{eq:local flux} and Raviart--Thomas discretization of arbitrary order.
The latter is inspired by the explicit construction of edge- and volume-oriented basis functions in \cite{ervin12}.
The code for reproducing the numerical experiments is published on the Code Ocean platform \cite{codeocean}.
The experiments are performed on a compute cluster with AMD EPYC 7642 48-core processors and 2 TB RAM.
All time values refer to the wall-clock time of one of three runs, which was determined by the median overall runtime.

We investigate the 2D Poisson model problem on the L-shaped domain $\Omega = (-1,1)^2 \setminus ( [0,1 ) \times (-1,0] )$ with exact solution $u^\star \in H^1_0(\Omega)$ given in polar coordinates $(r,\varphi)$ by
\begin{equation}
    \label{eq:exact_solution}
    u^\star(r,\varphi) = r^{2/3} \sin(2\varphi/3) \, (1 - r^2 \cos(\varphi)^2) \, (1 - r^2 \sin(\varphi)^2);
\end{equation}
see Figure~\ref{subfig:solution}.
The right-hand side $f \coloneqq -\Delta u^\star \in L^2(\Omega)$ is determined by $u^\star$.

Recall the standard residual-based error estimator $\eta_\ell$ from~\eqref{eq:residual_based_estimator} in Subsection~\ref{section:residual-estimator}.
Let $\mu_\ell^{\textup{ZZ}}$ denote the ZZ-type error estimator from~\eqref{eq:ZZ-estimator} in Section~\ref{section:zz} using the Scott--Zhang projection $G_\ell$ from~\eqref{eq:G-Scott-Zhang} and polynomial degree $q \coloneqq p - 1$ for the oscillation term.
The equilibrated-flux estimator $\mu_\ell^{\textup{eq}}$ from~\eqref{eq:equilibrated-flux-estimator} in Section~\ref{section:equiflux} is computed with Raviart--Thomas elements of order $q \coloneqq p$.
If not stated otherwise, we chose $p = 2$ for the polynomial degree of the discretization and the adaptivity parameters $\theta = 0.5$ and $\lambda = 0.2$ in Algorithm~\ref{algorithm:afem}.
The discrete linear solver $\Psi_\ell$ employs the $hp$-robust, uniformly contractive geometric multigrid method from~\cite{imps2022} with its built-in algebraic error estimator \(\zeta_\ell(u_\ell^k)\).

Figures~\ref{subfig:mesh:residual}--\ref{subfig:mesh:equiflux} display the meshes generated by the adaptive algorithm for the three estimators.
They show comparable meshes with increased refinement towards the re-entrant corner of the L-shaped domain.
At the displayed level, however, the ZZ-type estimator in Figure~\ref{subfig:mesh:ZZ} produces slightly less refinement in the remaining vertices compared to the other two estimators.

Nevertheless, both non-residual error estimators achieve optimal convergence rates in terms of the accumulated number of degrees of freedom
\[
    \cost_p(\ell,k)
    \coloneqq
    \sum_{\substack{(\ell',k') \in \QQ \\ |\ell',k'| \leq |\ell,k|}} \dim \XX_{\ell'}
\]
and the median cumulative runtime (in seconds); see Figures~\ref{fig:convergence:ZZ} and~\ref{fig:convergence:equiflux}.
We note that $\cost_p(\ell,k)$ is equivalent, up to a constant that scales with $p^2$, to the above definition~\eqref{eq:def:cost}, but provides a fairer measure for varying polynomial degrees $p \in \N$.
For polynomial degree $p = 5$, we even observe a super-optimal pre-asymptotic rate of about $-2.8$.

The Figures~\ref{fig:efficiency} present the efficiency indices 
$\mu_\ell(u_\ell^k) / \enorm{u^\star - u_\ell^k}$
of the estimators
$\mu_\ell \in \{\eta_\ell, \mu_\ell^{\textup{ZZ}}, \mu_\ell^{\textup{eq}}\}$
and confirm the superior efficiency indices of the equilibrated-flux estimator $\mu_\ell^{\textup{eq}}$ close to one.
In particular, the well-known $p$-robustness of the equilibrated-flux estimator \cite{bps09,ev20} is observed.

Figure~\ref{fig:iterations} investigates the number of iterations $\kk[\ell]$ of the discrete linear solver $\Psi_\ell$ on mesh $\TT_\ell$ in Algorithm~\ref{algorithm:afem} and reveals that the equilibrated-flux estimator $\mu_\ell^{\textup{eq}}$ requires more iterations than the ZZ-type estimator $\mu_\ell^{\textup{ZZ}}$ for the same polynomial degree~$p$.
This is likely due to the smaller effectivity indices leading to a significantly smaller upper bound $\zeta_\ell(u_\ell^k) \leq \lambda \, \mu_\ell^{\textup{eq}}(u_\ell^k)$ in the stopping criterion~\eqref{eq:single:termination} of the 
algebraic iteration of Algorithm~\ref{algorithm:afem}.

\begin{figure}
    \begin{subfigure}{0.45\textwidth}
        \centering
        \begin{tikzpicture}
    \pgfplotsset{/pgf/number format/fixed}
    \begin{axis}[%
        width=6.2cm,%
        xmin=-1.1, xmax=1.1,%
        ymin=-1.1, ymax=1.1,%
        zmin=-0.01, zmax=0.50,%
        font=\footnotesize,%
    ]
        \addplot3 graphics [%
            points={%
    (-1.00000000, 1.00000000, 0.00000000) => (1.00488189, 123.48330709)
    (0.00000000, -1.00000000, 0.00000000) => (258.01275591, 68.68015748)
    (1.00000000, 0.00000000, 0.00000000) => (281.63480315, 148.52267717)
    (-0.33000000, 0.40000000, 0.48215500) => (115.80803150, 279.86125984)
            }%
            ]
            {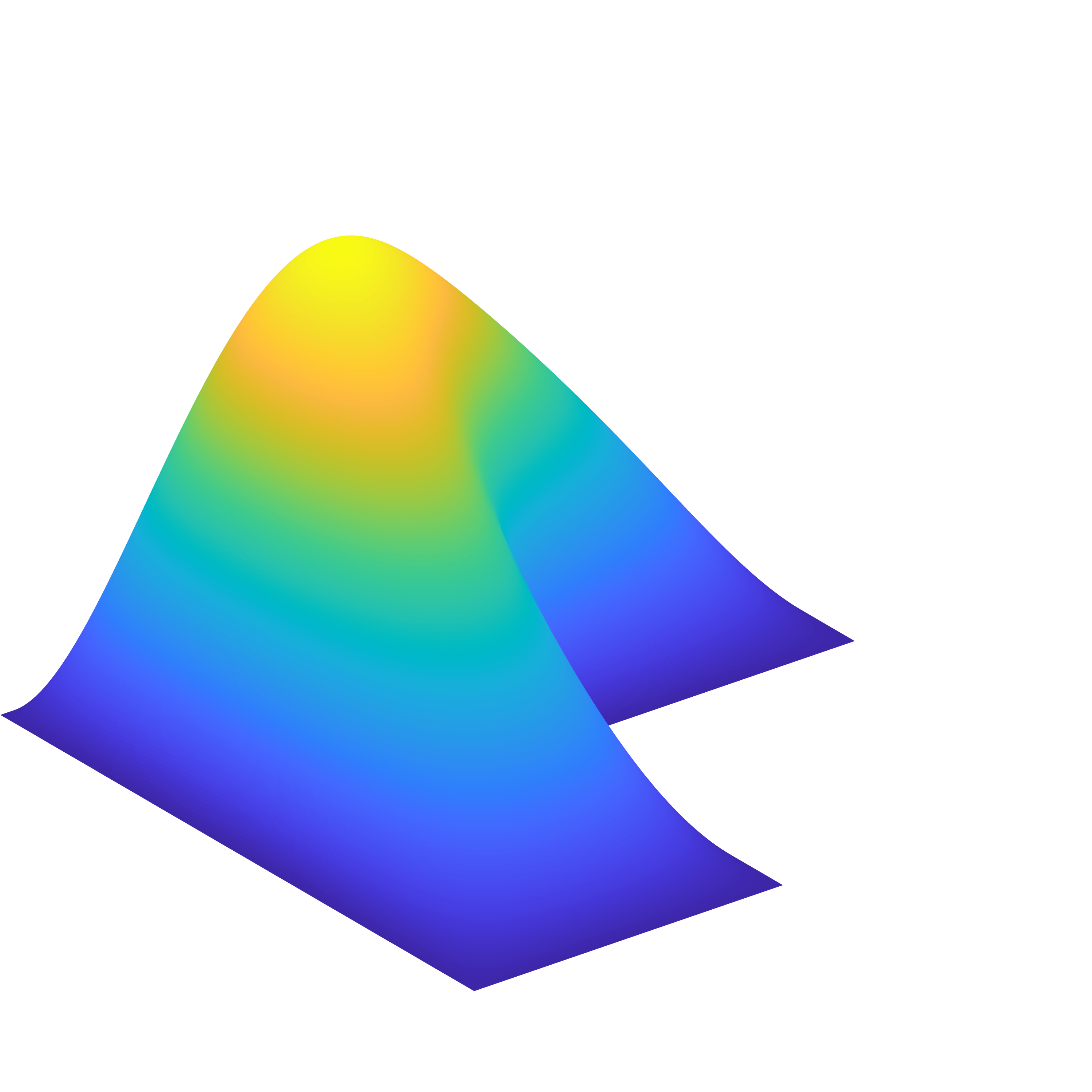};
    \end{axis}
\end{tikzpicture}
        \caption{Exact solution $u^\star$\newline}
        \label{subfig:solution}
    \end{subfigure}
    \hfil
    \begin{subfigure}{0.45\textwidth}
        \centering
        \begin{tikzpicture}
    \begin{axis}[%
        axis equal image,%
        width=6.1cm,%
        xmin=-1.1, xmax=1.1,%
        ymin=-1.1, ymax=1.1,%
        font=\footnotesize%
    ]
        \addplot graphics [xmin=-1, xmax=1, ymin=-1, ymax=1]
        {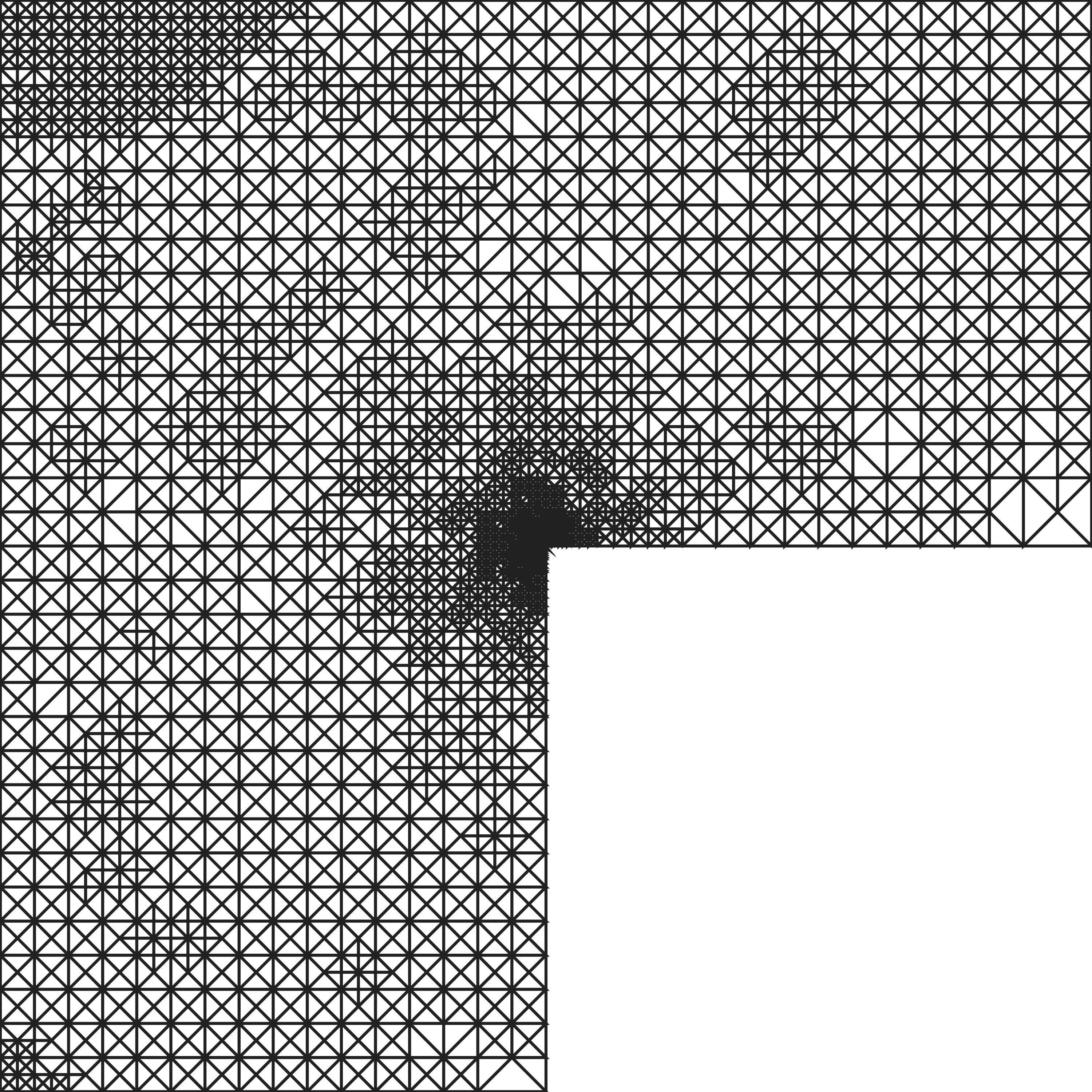};
    \end{axis}
\end{tikzpicture}
        \caption{$\TT_{12}$ with $\#\TT_{12} = 5\,750$ generated by $\eta_\ell$ from~\eqref{eq:residual_based_estimator}}
        \label{subfig:mesh:residual}
    \end{subfigure}
    \bigskip

    \begin{subfigure}{0.45\textwidth}
        \centering
        \begin{tikzpicture}
    \begin{axis}[%
        axis equal image,%
        width=6.1cm,%
        xmin=-1.1, xmax=1.1,%
        ymin=-1.1, ymax=1.1,%
        font=\footnotesize%
    ]
        \addplot graphics [xmin=-1, xmax=1, ymin=-1, ymax=1]
        {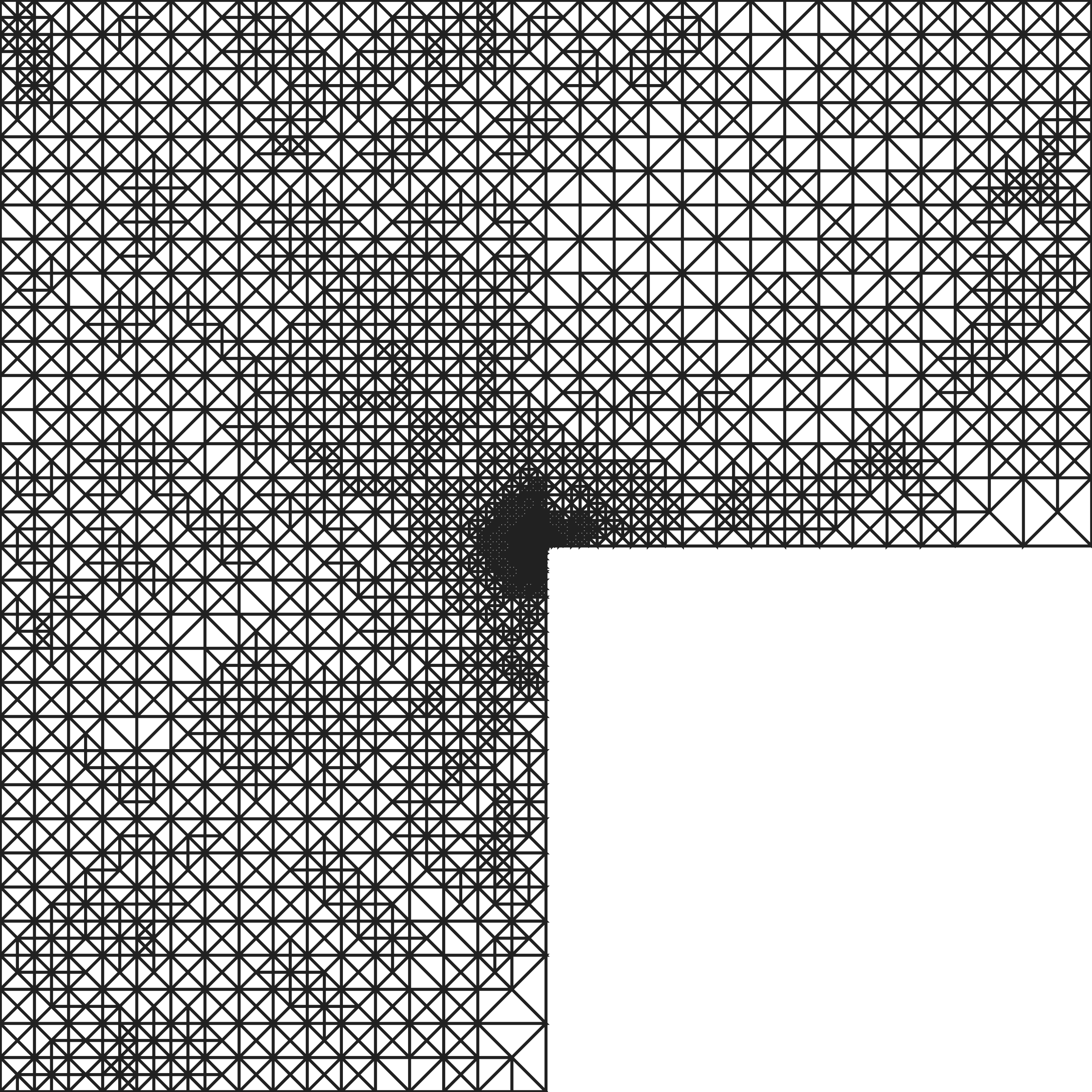};
    \end{axis}
\end{tikzpicture}
        \caption{$\TT_9$ with $\#\TT_9 = 5\,723$ generated by $\mu_\ell^{\textup{ZZ}}$ with $q \coloneqq p-1$ from~\eqref{eq:ZZ-estimator}}
        \label{subfig:mesh:ZZ}
    \end{subfigure}
    \hfil
    \begin{subfigure}{0.45\textwidth}
        \centering
        \begin{tikzpicture}
    \begin{axis}[%
        axis equal image,%
        width=6.1cm,%
        xmin=-1.1, xmax=1.1,%
        ymin=-1.1, ymax=1.1,%
        font=\footnotesize%
    ]
        \addplot graphics [xmin=-1, xmax=1, ymin=-1, ymax=1]
        {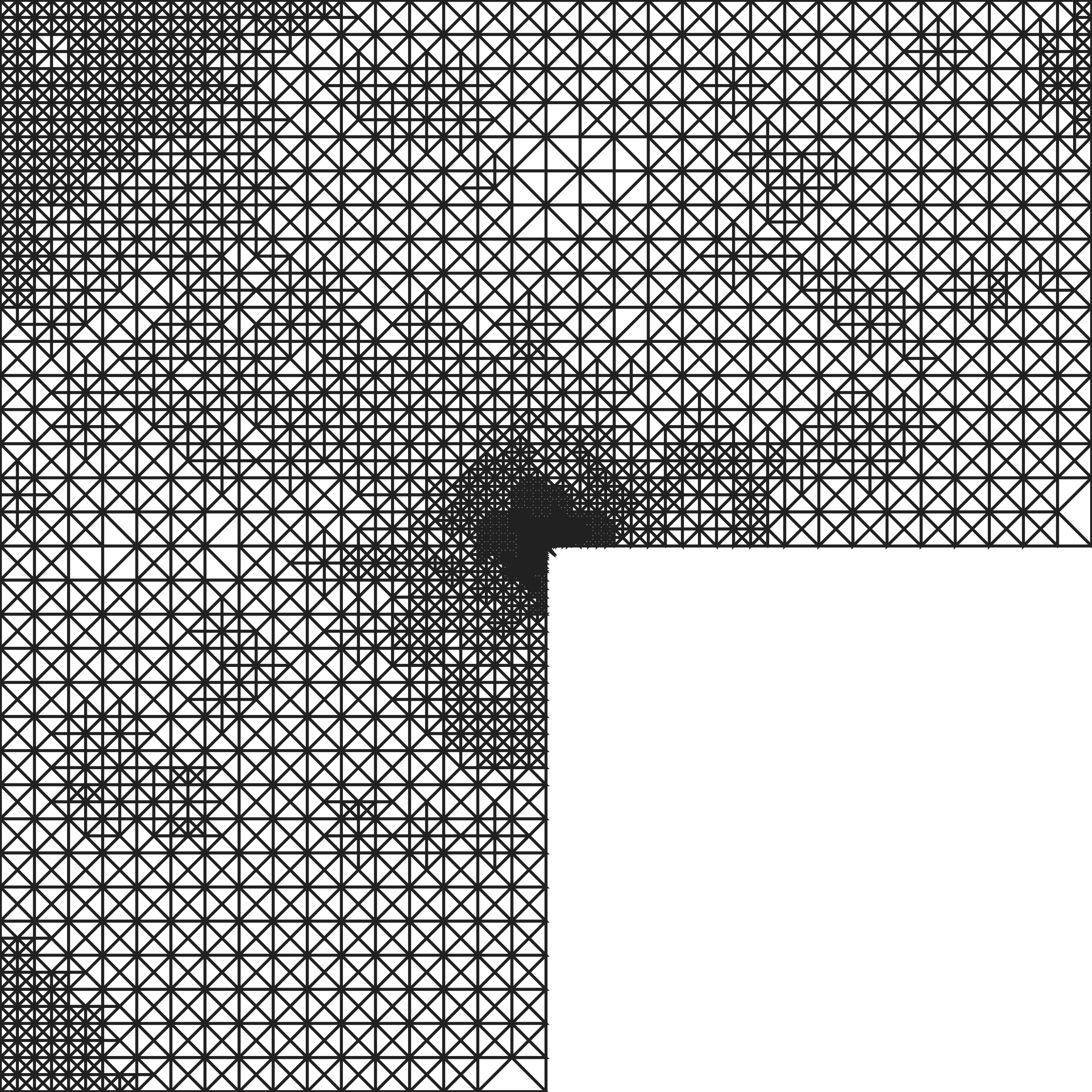};
    \end{axis}
\end{tikzpicture}
        \caption{$\TT_{14}$ with $\#\TT_{14} = 6\,687$ generated by $\mu_\ell^{\textup{eq}}$ with $q \coloneqq p$  from~\eqref{eq:equilibrated-flux-estimator}}
        \label{subfig:mesh:equiflux}
    \end{subfigure}
    \caption{Solution plot of $u^\star$ from~\eqref{eq:exact_solution}. Mesh plots generated by Algorithm~\ref{algorithm:afem} for the three considered estimators $\eta_\ell$, $\mu_\ell^{\textup{ZZ}}$, and $\mu_\ell^{\textup{eq}}$ with parameters $p = 2$, $\theta = 0.5$, and $\lambda = 0.2$.}
    \label{fig:mesh_solution}
\end{figure}

\begin{figure}
    \begin{subfigure}{0.49\textwidth}
        \centering
        \input{figures/Fig2_scottZhang_convergence_cost.tex}
        \caption{Convergence with respect to cost}
        \label{subfig:convergence:ZZ:cost}
    \end{subfigure}
    \hfil
    \begin{subfigure}{0.49\textwidth}
        \centering
        \input{figures/Fig2_scottZhang_convergence_runtime.tex}
        \caption{Convergence with respect to runtime}
        \label{subfig:convergence:ZZ:runtime}
    \end{subfigure}
    \bigskip

    \begin{subfigure}{0.79\textwidth}
        \centering
        \begin{tikzpicture}[>=stealth]
    %
    %
    \colorlet{col1}{pyBlue}
    \colorlet{col2}{pyOrange}
    \colorlet{col3}{pyGreen}
    \colorlet{col4}{pyRed}
    \colorlet{col5}{pyCyan}
    \colorlet{col6}{pyPurple}
    \colorlet{col7}{pyYellow}
    \colorlet{col8}{pyBrown}
    \colorlet{col9}{pyPink}
    \colorlet{col10}{pyGrey}
    %
    %
    \pgfplotsset{%
        %
        %
        markerdefault/.style = {%
            every mark/.append style = {solid},%
            gray,%
            every mark/.append style = {fill = gray!60!white}%
        },%
        marker1a/.style = {%
            markerdefault,%
            mark = o,%
            mark size = 1.66pt,%
            #1,%
            every mark/.append style = {fill = #1!60!white}%
        },%
        marker1b/.style = {%
            marker1a = #1,%
            mark = halfcircle*,%
        },%
        marker1c/.style = {%
            marker1a = #1,%
            mark = *,%
        },%
        marker2a/.style = {%
            markerdefault,%
            mark = square,%
            mark size = 1.5pt,%
            #1,%
            every mark/.append style = {fill = #1!60!white}%
        },%
        marker2b/.style = {%
            marker2a = #1,%
            mark = halfsquare*,%
            mark size = 2.2pt,%
            every mark/.append style = {rotate = 45}%
        },%
        marker2c/.style = {%
            marker2a = #1,%
            mark = square*,%
        },%
        marker3a/.style = {%
            markerdefault,%
            mark = triangle,%
            mark size = 2.2pt,%
            #1,%
            every mark/.append style = {fill = #1!60!white}%
        },%
        marker3b/.style = {%
            marker3a = #1,%
            mark = halftriangle*,%
        },%
        marker3c/.style = {%
            marker3a = #1,%
            mark = triangle*,%
        },%
        marker4a/.style = {%
            markerdefault,%
            mark = diamond,%
            mark size = 2.75pt,%
            #1,%
            every mark/.append style = {fill = #1!60!white}%
        },%
        marker4b/.style = {%
            marker4a = #1,%
            mark = halfdiamond*,%
        },%
        marker4c/.style = {%
            marker4a = #1,%
            mark = diamond*,%
        },%
        marker5a/.style = {%
            markerdefault,%
            mark = pentagon,%
            mark size = 2pt,%
            #1,%
            every mark/.append style = {fill = #1!60!white}%
        },%
        marker5b/.style = {%
            marker5a = #1,%
            mark = halfpentagon*,%
        },%
        marker5c/.style = {%
            marker5a = #1,%
            mark = pentagon*,%
        },%
        marker6a/.style = {%
            markerdefault,%
            mark = fivestar,%
            mark size = 2.5pt,%
            #1,%
            every mark/.append style = {fill = #1!60!white}%
        },%
        marker6b/.style = {%
            marker6a = #1,%
            mark = halffivestar*,%
        },%
        marker6c/.style = {%
            marker6a = #1,%
            mark = fivestar*,%
        },%
        marker7a/.style = {%
            markerdefault,%
            mark = sixstar,%
            mark size = 2.25pt,%
            #1,%
            every mark/.append style = {fill = #1!60!white}%
        },%
        marker7b/.style = {%
            marker7a = #1,%
            mark = halfsixstar*,%
        },%
        marker7c/.style = {%
            marker7a = #1,%
            mark = sixstar*,%
        },%
        %
        %
        uniform/.style = {%
            dashed,%
            every mark/.append style = {fill = black!20!white}%
        },%
        adaptive/.style = {%
            solid%
        },%
        reference/.style = {%
            thick,%
            dashed%
        }%
    }

    \matrix(m) [
        matrix of nodes,
        anchor = center,
        font = \footnotesize,
        column 1/.style={anchor=base east},
    ] at (0,0) {
        & $\enorm{u^\star - u_\ell^k}$
        & {\quad$\mu_\ell(u_\ell^k)$\quad}
        & {\quad$\eta_\ell(u_\ell^k)$\quad}
        \\
        \hline
        $p = 1$
        & \ref*{leg:convergence:scottZhang:S1adapt:err}
        & \ref*{leg:convergence:scottZhang:S1adapt:mu}
        & \ref*{leg:convergence:scottZhang:S1adapt:res}
        \\
        $p = 3$
        & \ref*{leg:convergence:scottZhang:S3adapt:err}
        & \ref*{leg:convergence:scottZhang:S3adapt:mu}
        & \ref*{leg:convergence:scottZhang:S3adapt:res}
        \\
        $p = 5$
        & \ref*{leg:convergence:scottZhang:S5adapt:err}
        & \ref*{leg:convergence:scottZhang:S5adapt:mu}
        & \ref*{leg:convergence:scottZhang:S5adapt:res}
        \\
    };
\end{tikzpicture}
        \caption{Legend}
        \label{subfig:convergence:ZZ:legend}
    \end{subfigure}

    \caption{%
        Convergence history for all iterates $(u_\ell^k)_{(\ell,k) \in \QQ}$ of Algorithm~\ref{algorithm:afem} driven by the ZZ-type estimator $\mu_\ell^{\textup{ZZ}}$ with $q \coloneqq p - 1$ from~\eqref{eq:ZZ-estimator} for different polynomial degrees $p$ with respect to the accumulated number of elements $\cost_p(\ell,k)$ and the median cumulative runtime (in seconds). 
        The remaining parameters read $\theta = 0.5$ and $\lambda = 0.2$.
        The residual-based error estimator $\eta_\ell$ is only shown for comparison.
    }
    \label{fig:convergence:ZZ}
\end{figure}

\begin{figure}
    \begin{subfigure}{0.49\textwidth}
        \centering
        \input{figures/Fig3_equiflux_convergence_cost.tex}
        \caption{Convergence with respect to cost}
        \label{subfig:convergence:equiflux:cost}
    \end{subfigure}
    \hfil
    \begin{subfigure}{0.49\textwidth}
        \centering
        \input{figures/Fig3_equiflux_convergence_runtime.tex}
        \caption{Convergence with respect to runtime}
        \label{subfig:convergence:equiflux:runtime}
    \end{subfigure}
    \caption{%
        Convergence history for all iterates $(u_\ell^k)_{(\ell,k) \in \QQ}$ of Algorithm~\ref{algorithm:afem} driven by the equilibrated-flux estimator $\mu_\ell^{\textup{eq}}$ with $q \coloneqq p$ from~\eqref{eq:equilibrated-flux-estimator} for different polynomial degrees with respect to the accumulated number of elements $\cost_p(\ell,k)$ and the median cumulative runtime (in seconds). 
        The remaining parameters read $\theta = 0.5$ and $\lambda = 0.2$.
        The residual-based error estimator $\eta_\ell$ is only shown for comparison.
        The plots employ the same legend from Subfigure~\ref{subfig:convergence:ZZ:legend}.
    }
    \label{fig:convergence:equiflux}
\end{figure}

\begin{figure}
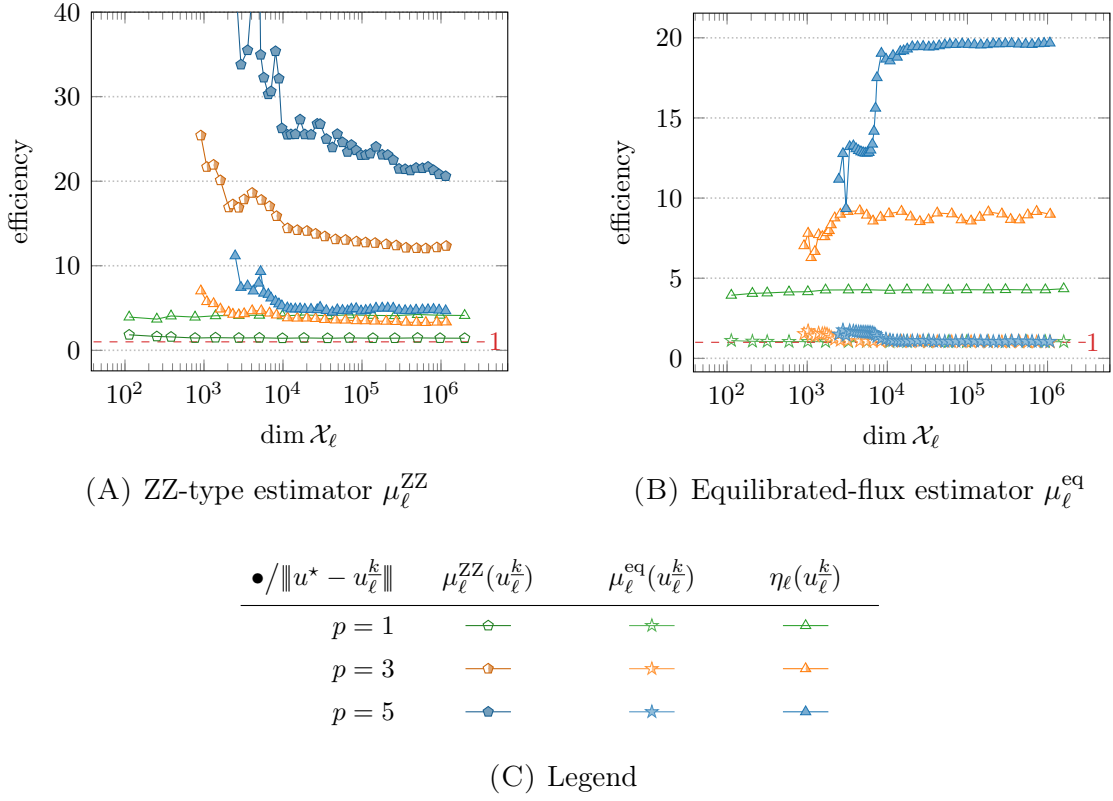

    \begin{subfigure}{0.49\textwidth}
        \centering
        \input{figures/Fig4_efficiency_scottZhang.tex}
        \caption{ZZ-type estimator $\mu_\ell^{\textup{ZZ}}$}
        \label{subfig:efficiency:scottZhang}
    \end{subfigure}
    \hfil
    \begin{subfigure}{0.49\textwidth}
        \centering
        \input{figures/Fig4_efficiency_equiflux.tex}
        \caption{Equilibrated-flux estimator $\mu_\ell^{\textup{eq}}$}
        \label{subfig:efficiency:equiflux}
    \end{subfigure}
    \bigskip

    \begin{subfigure}{0.79\textwidth}
        \centering
        \begin{tikzpicture}[>=stealth]
    %
    %
    \colorlet{col1}{pyBlue}
    \colorlet{col2}{pyOrange}
    \colorlet{col3}{pyGreen}
    \colorlet{col4}{pyRed}
    \colorlet{col5}{pyCyan}
    \colorlet{col6}{pyPurple}
    \colorlet{col7}{pyYellow}
    \colorlet{col8}{pyBrown}
    \colorlet{col9}{pyPink}
    \colorlet{col10}{pyGrey}
    %
    %
    \pgfplotsset{%
        %
        %
        markerdefault/.style = {%
            every mark/.append style = {solid},%
            gray,%
            every mark/.append style = {fill = gray!60!white}%
        },%
        marker1a/.style = {%
            markerdefault,%
            mark = o,%
            mark size = 1.66pt,%
            #1,%
            every mark/.append style = {fill = #1!60!white}%
        },%
        marker1b/.style = {%
            marker1a = #1,%
            mark = halfcircle*,%
        },%
        marker1c/.style = {%
            marker1a = #1,%
            mark = *,%
        },%
        marker2a/.style = {%
            markerdefault,%
            mark = square,%
            mark size = 1.5pt,%
            #1,%
            every mark/.append style = {fill = #1!60!white}%
        },%
        marker2b/.style = {%
            marker2a = #1,%
            mark = halfsquare*,%
            mark size = 2.2pt,%
            every mark/.append style = {rotate = 45}%
        },%
        marker2c/.style = {%
            marker2a = #1,%
            mark = square*,%
        },%
        marker3a/.style = {%
            markerdefault,%
            mark = triangle,%
            mark size = 2.2pt,%
            #1,%
            every mark/.append style = {fill = #1!60!white}%
        },%
        marker3b/.style = {%
            marker3a = #1,%
            mark = halftriangle*,%
        },%
        marker3c/.style = {%
            marker3a = #1,%
            mark = triangle*,%
        },%
        marker4a/.style = {%
            markerdefault,%
            mark = diamond,%
            mark size = 2.75pt,%
            #1,%
            every mark/.append style = {fill = #1!60!white}%
        },%
        marker4b/.style = {%
            marker4a = #1,%
            mark = halfdiamond*,%
        },%
        marker4c/.style = {%
            marker4a = #1,%
            mark = diamond*,%
        },%
        marker5a/.style = {%
            markerdefault,%
            mark = pentagon,%
            mark size = 2pt,%
            #1,%
            every mark/.append style = {fill = #1!60!white}%
        },%
        marker5b/.style = {%
            marker5a = #1,%
            mark = halfpentagon*,%
        },%
        marker5c/.style = {%
            marker5a = #1,%
            mark = pentagon*,%
        },%
        marker6a/.style = {%
            markerdefault,%
            mark = fivestar,%
            mark size = 2.5pt,%
            #1,%
            every mark/.append style = {fill = #1!60!white}%
        },%
        marker6b/.style = {%
            marker6a = #1,%
            mark = halffivestar*,%
        },%
        marker6c/.style = {%
            marker6a = #1,%
            mark = fivestar*,%
        },%
        marker7a/.style = {%
            markerdefault,%
            mark = sixstar,%
            mark size = 2.25pt,%
            #1,%
            every mark/.append style = {fill = #1!60!white}%
        },%
        marker7b/.style = {%
            marker7a = #1,%
            mark = halfsixstar*,%
        },%
        marker7c/.style = {%
            marker7a = #1,%
            mark = sixstar*,%
        },%
        %
        %
        uniform/.style = {%
            dashed,%
            every mark/.append style = {fill = black!20!white}%
        },%
        adaptive/.style = {%
            solid%
        },%
        reference/.style = {%
            thick,%
            dashed%
        }%
    }

    \matrix(m) [
        matrix of nodes,
        anchor = center,
        font = \footnotesize,
        column 1/.style={anchor=base east},
    ] at (0,0) {
        $\bullet \big/ \enorm{u^\star - u_\ell^\kk}$
        & {\quad$\mu_\ell^{\textup{ZZ}}(u_\ell^\kk)$\quad}
        & {\quad$\mu_\ell^{\textup{eq}}(u_\ell^\kk)$\quad}
        & {\quad$\eta_\ell(u_\ell^\kk)$\quad}
        \\
        \hline
        $p = 1$
        & \ref*{leg:efficiency:scottZhang:S1adapt:scottZhang}
        & \ref*{leg:efficiency:equiflux:S1adapt:equiflux}
        & \ref*{leg:efficiency:equiflux:S1adapt:res}
        \\
        $p = 3$
        & \ref*{leg:efficiency:scottZhang:S3adapt:scottZhang}
        & \ref*{leg:efficiency:equiflux:S3adapt:equiflux}
        & \ref*{leg:efficiency:equiflux:S3adapt:res}
        \\
        $p = 5$
        & \ref*{leg:efficiency:scottZhang:S5adapt:scottZhang}
        & \ref*{leg:efficiency:equiflux:S5adapt:equiflux}
        & \ref*{leg:efficiency:equiflux:S5adapt:res}
        \\
    };
\end{tikzpicture}
        \caption{Legend}
        \label{subfig:efficiency:legend}
    \end{subfigure}

    \caption{%
        Effectivity indices $\mu_\ell(u_\ell^\kk)/\enorm{u^\star - u_\ell^\kk}$ of $\mu_\ell \in \{\eta_\ell, \mu_\ell^{\textup{ZZ}}, \mu_\ell^{\textup{eq}}\}$ for the final iterates $(u_\ell^\kk)_{(\ell,\kk) \in \QQ}$ of Algorithm~\ref{algorithm:afem} driven by the ZZ-type estimator $\mu_\ell^{\textup{ZZ}}$ with $q \coloneqq p-1$ from~\eqref{eq:ZZ-estimator} and the equilibrated-flux estimator $\mu_\ell^{\textup{eq}}$ with $q \coloneqq p$ from~\eqref{eq:equilibrated-flux-estimator} for various polynomial degrees $p$.
        The remaining parameters read $\theta = 0.5$ and $\lambda = 0.2$.
        The effectivity indices of the residual-based error estimator $\eta_\ell$ are only shown for comparison.
    }
    \label{fig:efficiency}
\end{figure}

\begin{figure}
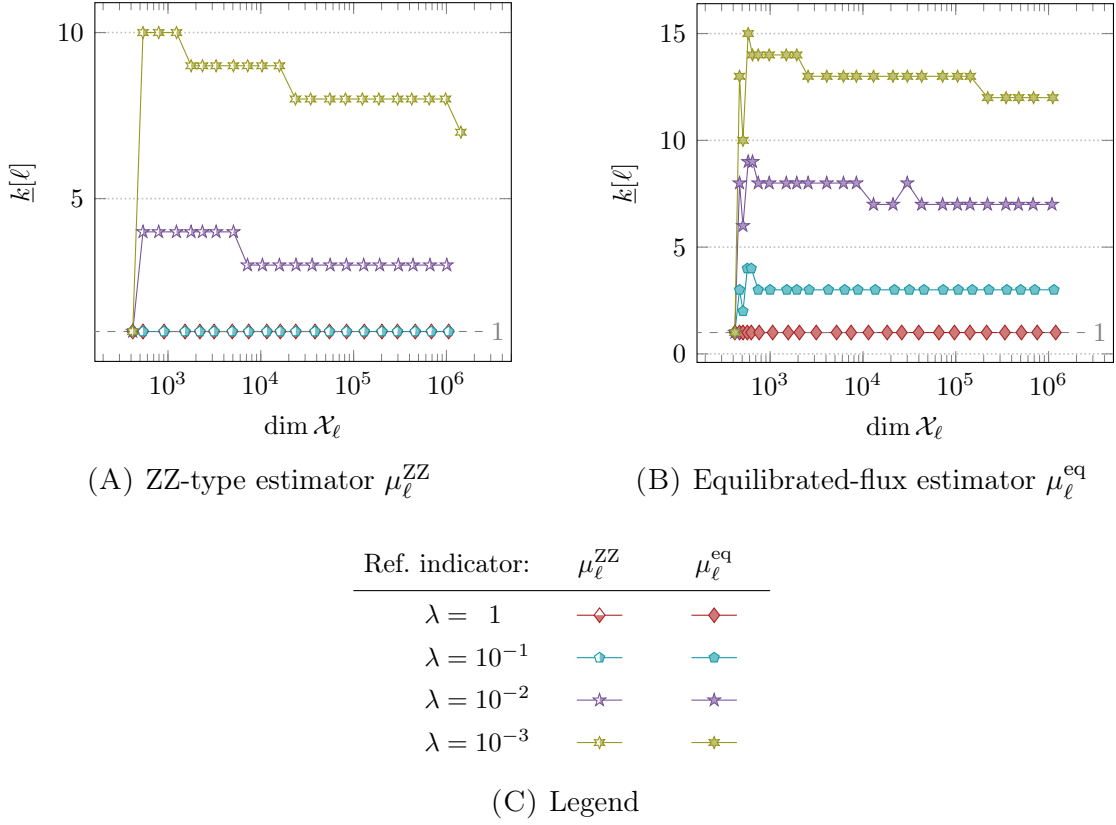

    \begin{subfigure}{0.49\textwidth}
        \centering
        \input{figures/Fig5_iterations_scottZhang.tex}
        \caption{ZZ-type estimator $\mu_\ell^{\textup{ZZ}}$}
        \label{subfig:iterations:scottZhang}
    \end{subfigure}
    \hfil
    \begin{subfigure}{0.49\textwidth}
        \centering
        \input{figures/Fig5_iterations_equiflux.tex}
        \caption{Equilibrated-flux estimator $\mu_\ell^{\textup{eq}}$}
        \label{subfig:iterations:equiflux}
    \end{subfigure}
    \bigskip

    \begin{subfigure}{0.79\textwidth}
        \centering
        \begin{tikzpicture}[>=stealth]
    %
    %
    \colorlet{col1}{pyBlue}
    \colorlet{col2}{pyOrange}
    \colorlet{col3}{pyGreen}
    \colorlet{col4}{pyRed}
    \colorlet{col5}{pyCyan}
    \colorlet{col6}{pyPurple}
    \colorlet{col7}{pyYellow}
    \colorlet{col8}{pyBrown}
    \colorlet{col9}{pyPink}
    \colorlet{col10}{pyGrey}
    %
    %
    \pgfplotsset{%
        %
        %
        markerdefault/.style = {%
            every mark/.append style = {solid},%
            gray,%
            every mark/.append style = {fill = gray!60!white}%
        },%
        marker1a/.style = {%
            markerdefault,%
            mark = o,%
            mark size = 1.66pt,%
            #1,%
            every mark/.append style = {fill = #1!60!white}%
        },%
        marker1b/.style = {%
            marker1a = #1,%
            mark = halfcircle*,%
        },%
        marker1c/.style = {%
            marker1a = #1,%
            mark = *,%
        },%
        marker2a/.style = {%
            markerdefault,%
            mark = square,%
            mark size = 1.5pt,%
            #1,%
            every mark/.append style = {fill = #1!60!white}%
        },%
        marker2b/.style = {%
            marker2a = #1,%
            mark = halfsquare*,%
            mark size = 2.2pt,%
            every mark/.append style = {rotate = 45}%
        },%
        marker2c/.style = {%
            marker2a = #1,%
            mark = square*,%
        },%
        marker3a/.style = {%
            markerdefault,%
            mark = triangle,%
            mark size = 2.2pt,%
            #1,%
            every mark/.append style = {fill = #1!60!white}%
        },%
        marker3b/.style = {%
            marker3a = #1,%
            mark = halftriangle*,%
        },%
        marker3c/.style = {%
            marker3a = #1,%
            mark = triangle*,%
        },%
        marker4a/.style = {%
            markerdefault,%
            mark = diamond,%
            mark size = 2.75pt,%
            #1,%
            every mark/.append style = {fill = #1!60!white}%
        },%
        marker4b/.style = {%
            marker4a = #1,%
            mark = halfdiamond*,%
        },%
        marker4c/.style = {%
            marker4a = #1,%
            mark = diamond*,%
        },%
        marker5a/.style = {%
            markerdefault,%
            mark = pentagon,%
            mark size = 2pt,%
            #1,%
            every mark/.append style = {fill = #1!60!white}%
        },%
        marker5b/.style = {%
            marker5a = #1,%
            mark = halfpentagon*,%
        },%
        marker5c/.style = {%
            marker5a = #1,%
            mark = pentagon*,%
        },%
        marker6a/.style = {%
            markerdefault,%
            mark = fivestar,%
            mark size = 2.5pt,%
            #1,%
            every mark/.append style = {fill = #1!60!white}%
        },%
        marker6b/.style = {%
            marker6a = #1,%
            mark = halffivestar*,%
        },%
        marker6c/.style = {%
            marker6a = #1,%
            mark = fivestar*,%
        },%
        marker7a/.style = {%
            markerdefault,%
            mark = sixstar,%
            mark size = 2.25pt,%
            #1,%
            every mark/.append style = {fill = #1!60!white}%
        },%
        marker7b/.style = {%
            marker7a = #1,%
            mark = halfsixstar*,%
        },%
        marker7c/.style = {%
            marker7a = #1,%
            mark = sixstar*,%
        },%
        %
        %
        uniform/.style = {%
            dashed,%
            every mark/.append style = {fill = black!20!white}%
        },%
        adaptive/.style = {%
            solid%
        },%
        reference/.style = {%
            thick,%
            dashed%
        }%
    }

    \matrix(m) [
        matrix of nodes,
        anchor = center,
        font = \footnotesize,
        column 1/.style={anchor=base east},
    ] at (0,0) {
        Ref.\ indicator:
        & {\quad$\mu_\ell^{\textup{ZZ}}$\quad}
        & {\quad$\mu_\ell^{\textup{eq}}$\quad}
        \\
        \hline
        $\lambda = \phantom{0}1\phantom{{}^{-1}}$
        & \ref*{leg:iterations:scottZhang:S2adapt:lambdaE0}
        & \ref*{leg:iterations:equiflux:S2adapt:lambdaE0}
        \\
        $\lambda = 10^{-1}$
        & \ref*{leg:iterations:scottZhang:S2adapt:lambdaE-1}
        & \ref*{leg:iterations:equiflux:S2adapt:lambdaE-1}
        \\
        $\lambda = 10^{-2}$
        & \ref*{leg:iterations:scottZhang:S2adapt:lambdaE-2}
        & \ref*{leg:iterations:equiflux:S2adapt:lambdaE-2}
        \\
        $\lambda = 10^{-3}$
        & \ref*{leg:iterations:scottZhang:S2adapt:lambdaE-3}
        & \ref*{leg:iterations:equiflux:S2adapt:lambdaE-3}
        \\
    };
\end{tikzpicture}
        \caption{Legend}
        \label{subfig:iterations:legend}
    \end{subfigure}

    \caption{%
        Number of iterations $\kk[\ell]$ of the discrete linear solver $\Psi_\ell$ on mesh $\TT_\ell$ in Algorithm~\ref{algorithm:afem} driven by the ZZ-type estimator $\mu_\ell^{\textup{ZZ}}$ with $q \coloneqq p - 1$ from~\eqref{eq:ZZ-estimator} and the equilibrated-flux estimator $\mu_\ell^{\textup{eq}}$ with $q \coloneqq p$ from~\eqref{eq:equilibrated-flux-estimator}. 
        The remaining parameters read $\theta = 0.5$ and $p = 2$.
    }
    \label{fig:iterations}
\end{figure}


\printbibliography

\end{document}